\documentclass[reqno, 11pt]{amsart}
\usepackage{amsmath,amsthm,amsfonts,amscd,amssymb,comment,eucal,latexsym,mathrsfs}
\usepackage{stmaryrd}
\usepackage[all]{xy}
\newtheorem{theorem}{Theorem}[section]

\theoremstyle{definition}

\newtheorem{remark}[theorem]{Remark}
\newtheorem{notation}[theorem]{Notation}
\newtheorem{example}[theorem]{Example}
\newtheorem{exercise}[theorem]{Exercise}

\theoremstyle{plain}

\newtheorem{thm}{Theorem}[section]
\newtheorem{prop}[thm]{Proposition}
\newtheorem{lem}[thm]{Lemma}
\newtheorem{cor}[thm]{Corollary}

\theoremstyle{definition}
\newtheorem{defn}[thm]{Definition}
\theoremstyle{remark}

\newcommand{\Aut}{{\rm Aut}}

\newcommand{\cB}{{\mathcal B}}

\newcommand{\cF}{{\mathcal F}}

\newcommand{\cC}{{\mathcal C}}

\newcommand{\cT}{{\mathcal T}}
\newcommand{\cQ}{{\mathcal Q}}
\newcommand{\cP}{{\mathcal P}}
\newcommand{\cR}{{\mathcal R}}
\newcommand{\cS}{{\mathcal S}}

\newcommand{\Nb}{{\mathbb N}}

\newcommand{\diam}{{\rm diam}}

\newcommand{\tr}{{\rm tr}}

\newcommand{\sF}{{\mathscr F}}
\newcommand{\sG}{{\mathscr G}}
\newcommand{\sH}{{\mathscr H}}

\newcommand{\sK}{{\mathscr K}}

\newcommand{\EE}{{\mathbb{E}}}

\newcommand{\Hom}{{\rm Hom}}

\newcommand{\sym}{{\rm Sym}}

\newcommand{\so}{\mathfrak{s}}
\newcommand{\ra}{\mathfrak{r}}

\def\cc{{\curvearrowright}}

\def\P{{\mathbb{P}}}
\newcommand{\C}{\mathbb{C}}
\newcommand{\N}{\mathbb{N}}
\newcommand{\R}{\mathbb{R}}

\newcommand{\Map}{\textrm{Map}}

\newcommand{\Z}{\mathbb{Z}}

\newcommand{\lb}{\llbracket}
\newcommand{\rb}{\rrbracket}

\newcommand{\tdelta}{\tilde{\delta}}
\allowdisplaybreaks

\begin{document}
\title[Entropy theory for sofic groupoids I: the foundations]{Entropy theory for sofic groupoids I:\\ the foundations}
\author{Lewis Bowen }
%\author{University of Hawaii}
%\author[Lewis Bowen]{Lewis Bowen$\dagger$}
\address{Department of Mathematics\\
University of Texas, Austin\\
} %one \address command per author
\email{lpbowen@math.utexas.edu}
\thanks{$\dagger$ Supported in part by NSF grant DMS-0968762 and NSF CAREER Award DMS-0954606.}
\begin{abstract}
This is the first part in a series in which sofic entropy theory is generalized to class-bijective extensions of sofic groupoids. Here we define topological and measure entropy and prove invariance. We also establish the variational principle, compute the entropy of Bernoulli shift actions and answer a question of Benjy Weiss pertaining to the isomorphism problem for non-free Bernoulli shifts. The proofs are independent of previous literature.
\end{abstract}
\maketitle

\maketitle

\tableofcontents

\section{Introduction}

A major motivation for this work comes from Rudolph-Weiss' discovery \cite{RW00} that, while entropy is not an orbit-equivalence invariant, relative entropy with respect to the orbit-change sigma-algebra is invariant. From this fact it is possible to generalize classical entropy theory to extensions of measured equivalence relations \cite{Da01, DP02}. These new ideas have led to the solution of a number of open problems \cite{RW00, Da01, DP02, DG02,Av05,Av10,Bo12} as well as providing shorter proofs to known results. Recently, classical entropy theory has been extended to actions of sofic groups \cite{Bo10b,KL11}. It is therefore natural to extend sofic entropy theory to equivalence relations and more generally, groupoids, which is the goal of the present paper. We begin with a short introduction to sofic groups, entropy theory (classical and sofic) and the Rudolph-Weiss result.

%The introduction of orbit-equivalence techniques into entropy theory has been quite profitable \cite{RW00, Da01, DG02, DP02, Av05, Av10, Bo12}. Entropy theory has recently been extended to actions of sofic groups \cite{Bo10b, KL11} which naturally motivates the problem of finding a orbit-theory of sofic entropy, which is the subject of the present paper.
 %Of course, they are also important objects of study in their own right  \cite{Pa99, KM04, Ke10, Ga10, Fu11}. This paper is the first in a series in which entropy theory is extended to sofic measured equivalence relations and groupoids.

A countable discrete group $G$ is {\em sofic} if there exists a sequence $\Sigma=\{\sigma_i\}_{i=1}^\infty$ of set maps $\sigma_i: G \to \sym(d_i)$ (the symmetric group on $\{1,\ldots, d_i\}$) such that
\begin{eqnarray*}
0 &=& \lim_{i\to\infty} d_i^{-1} | \{1 \le p \le d_i:~ \sigma_i(g)p = p \}| \quad \forall g \in G \setminus \{e\}\\
1 &=& \lim_{i\to\infty} d_i^{-1} | \{ 1\le p \le d_i:~ \sigma_i(g)\sigma_i(h)p = \sigma_i(gh)p \} | \quad \forall g,h \in G\\
+\infty &=& \lim_{i\to\infty} d_i.
\end{eqnarray*}
Such a sequence is called a {\em sofic approximation} to $G$. This class of groups was defined implicitly by M. Gromov \cite{Gr99}, explicitly by B. Weiss \cite{We00} and proven to satisfy a number of important conjectures such as Gottshalk's surjunctivity conjecture \cite{Gr99, We00}, Connes' embedding conjecture \cite{ES05}, the determinant conjecture \cite{ES05} and Kaplansky's direct finiteness conjecture \cite{ES04}. All amenable groups and residually finite groups are sofic. By Mal'cev's Theorem \cite{Ma40}, all finitely generated linear groups are residually finite. Since a group is sofic if and only if all of its finitely generated subgroups are sofic, this implies all countable linear groups are sofic. It is unknown whether all countable groups are sofic. The concept of soficity was generalized from groups to unimodular random rooted networks and measured equivalence relations in \cite{AL07} (see also \cite{EL10}) and to groupoids in \cite{DKP11}. V. Pestov has written an illuminating survey article \cite{Pe08}.

Entropy is an important invariant for classifying dynamical systems. To explain, let $G$ be a countable group, $(X,\mu)$ a standard Borel probability space and $T:G \to \Aut(X,\mu)$ a homomorphism into the group of pmp (probability-measure-preserving) transformations of $(X,\mu)$. Suppose that $\cP$ is a finite Borel partition of $X$. Define
\begin{eqnarray*}
H_\mu(\cP) &:=& - \sum_{P \in \cP} \mu(P) \log \mu(P)\\
h_\mu(T,\cP) &:=& \inf_{F \subset G} |F|^{-1} H_\mu\left( \bigvee_{f\in F} f\cP \right)\\
h_\mu(T) &:=& \sup_\cP h_\mu(T,\cP)
\end{eqnarray*}
where $\bigvee_{f\in F} f\cP$ denotes the common refinement of the partitions $f\cP$ (for $f\in F$), the infimum in the second line is over all nonempty finite subsets of $G$ and the supremum in the last line is over all finite Borel partitions $\cP$ of $X$. 

The quantity $h_\mu(T)$ is the {\em entropy rate} of $T$ and is clearly an invariant of the action. While this definition makes sense for actions of any group $G$, it has only proven useful for actions of amenable groups. Recall that $G$ is {\em amenable} if there exists a sequence $\{F_i\}_{i=1}^\infty$ of finite nonempty subsets $F_i \subset G$ such that for any nonempty finite $K \subset G$,
$$\lim_{i\to\infty} \frac{ |F_i \cap F_i K| }{ |F_i| } = 1.$$
Such a sequence is called a F\o lner sequence. 

A partition $\cP$ is {\em generating} for $T$ if the smallest $T(G)$-invariant sigma-algebra containing $\cP$ is the sigma-algebra of all measurable sets (up to measure zero). The Kolmogorov-Sinai Theorem (initially proven for $G=\Z$ and extended to amenable groups by other authors) states that  if $\cP$ is a finite generating partition for $T$ then 
$$h_\mu(T,\cP)=h_\mu(T)=\lim_{i\to\infty} |F_i|^{-1} H_\mu\left( \bigvee_{f\in F_i} f\cP \right)$$
where $\{F_i\}_{i=1}^\infty$ is any F\o lner sequence (for example, see \cite{Ol85}). This result is of fundamental importance because it makes the computation of entropy possible.

 The Kolmogorov-Sinai Theorem does not hold for non-amenable groups, but in this case $h_\mu(T)$ as defined above is really not the appropriate definition. For example, let $G$ be a countable group, $K$ be a finite set and $\kappa$ be a probability measure on $K$. Let $G$ act on the product space $(K,\kappa)^G$ by $(g\cdot x)(f)=x(g^{-1}f)$ for $x \in K^G, g,f \in G$ (where we interpret an element of $x\in K^G$ to be a function $x:G \to K$). This is called the {\em Bernoulli shift} action over $G$ with base space $(K,\kappa)$. Let $\cP_K=\{P_k:~k\in K\}$ be the canonical partition where $P_k=\{x\in K^G:~x(e)=k\}$. Then $\cP_K$ is generating. However if $G$ is non-amenable and $F \subset G$ is a large enough finite set, then $h_{\kappa^G}(T,\cP_K) < h_{\kappa^G}(T, \bigvee_{f\in F} f\cP_K)$. In fact, it can be shown that $h_{\kappa^G}(T) = +\infty$ (unless $(K,\kappa)$ is trivial). By contrast, if $G$ is amenable then $h_{\kappa^G}(T) = H_{\kappa^G}(\cP_K)<\infty$.

Sofic entropy theory is a generalization of the classical Kolmogorov-Sinai entropy theory initiated by the author \cite{Bo10b} in the measure-theoretic setting. It was extended to the topological setting and developed further by D. Kerr and H. Li \cite{KL11, KL2}. The main idea is to replace the F\o lner sequence from the amenable case with a sofic approximation. Thus sofic entropy quantifies the exponential growth rate  of the number of ``finite approximations'' to the system. There are many different interpretations of the phrase ``finite approximation'' which lead to many different but equivalent definitions of sofic entropy. Indeed, we now have definitions based on partitions \cite{Bo10b, Ke12}, on topological models and pseudo-metrics \cite{KL2}, on sequences of $L^\infty$ functions or continuous functions \cite{KL11}, on homomorphisms from $C(X)$ to $\C^d$ \cite{KL11}, and on open covers \cite{Zh11}. 

In general, the sofic entropy of an action depends on a choice of sofic approximation, which naturally leads to the question ``what is the best choice?'' One approach to this question is to say that a good choice should satisfy various identities. For example, it should be additive under direct product, satisfy a subgroup formula and behave appropriately with respect to ergodic decompositions. In general, such identities do not hold. However, when the group $G$ is a finitely generated free group then there is a {\em random} sofic approximation for which the corresponding sofic entropy, known as the $f$-invariant, satisfies these identities. This invariant was introduced in \cite{Bo10a} via an explicit formula from which it is easily seen to be additive under direct products. In \cite{Bo10c} it is shown to satisfy the analogue of Rohlin's formula, in \cite{BG12} Yuzvinskii's addition formula, in \cite{Se12a} a subgroup formula and in \cite{Se12b} an ergodic decomposition formula. Moreover, in \cite{Bo10d} it is shown that it is sofic entropy with respect to a random sofic approximation. Because of the importance of the $f$-invariant, we work with random sofic approximations in this paper.

%The first definition of measure sofic entropy was presented in \cite{Bo10}. It required the existence of a finite-entropy generating partition. This requirement was removed, the definition extended to topological sofic entropy and the variational principle established in \cite{KL11}. The main goal of this paper is to extend these results to sofic groupoids.

%More precisely, given a group $G$ with a sofic approximation $\Sigma$ and a probability-measure-preserving action $G \cc (X,\mu)$, we obtain a measure-conjugacy invariant $h_\Sigma(G \cc (X,\mu))$ called {\em sofic entropy}. This number may depend on the choice of sofic approximation $\Sigma$. 

In \cite{RW00}, Rudolph and Weiss proved the following. If, for $i=1,2$, $G_i$ are countable amenable groups, $T_i:G_i \to \Aut(X,\mu)$ (for $i=1,2$) are free ergodic probability-measure-preserving actions with the same orbits and the cocycle $\alpha: G_1 \times X \to G_2$ defined by $\alpha(g,x) = h$ if $T_1(g)x=T_2(h)x$ is measurable with respect to the counting measures on $G_1,G_2$ and a sub-sigma algebra $\cF$ on $X$ which is both $T_1(G_1)$ and $T_2(G_2)$-invariant then
$$h_\mu(T_1,\cP|\cF ) = h_\mu(T_2,\cP|\cF)$$
for any finite partition $\cP$. This implies that the relative entropy of a class-bijective extension of two discrete amenable measured equivalence relations is well-defined. This fact has proven to be very useful in extending classical results about $\Z$-actions to actions of arbitrary amenable groups \cite{RW00, Da01, DP02, DG02, Av05, Av10}. Our main results expand on this work by defining relative entropy for extensions of sofic measured groupoids. Because groups and measured equivalence relations are special cases of measured groupoids, the results here extend many previous results.

%Inspired by this result, A. I. Danilenko showed how to define relative entropy for extensions of amenable measured equivalence relations \cite{Da01}. This led to short proofs of many difficult classical results \cite{Da01, DP02}. 

The various definitions of entropy all depend (apriori) on the choice of an auxiliary object. In the classical case, the auxiliary object is the partition $\cP$. This paper uses pseudo-metrics to define topology entropy. We also present two definitions of measure entropy: one based on a choice of sigma-algebra and the other based on a choice of pseudo-metric. The main results of \S \ref{sec:top}-\ref{sec:meas2} and of the paper are that these choices are irrelevant and  the first and second definitions of measure entropy coincide (Theorems \ref{thm:K-top}, \ref{thm:K}, \ref{thm:1=2}). Like the Kolmogorov-Sinai Theorem, these results are of fundamental importance to the theory because they show that one can compute or estimate entropy using whatever auxiliary object is most convenient.

Our definition of topological entropy is modeled after \cite[Definition 2.3]{KL2}, our first definition of measure entropy is modeled after \cite{Ke12} and our second definition of measure entropy is modeled after \cite[Definition 3.3]{KL2}. The proofs are independent of previous literature. While some aspects of the proofs follow \cite{KL11,Ke12}, other parts are new. In particular, we avoid the operator-theoretic point of view of \cite{KL11}. The main advantage of our first definition of measure entropy is that it does not depend on a choice of topological model while our second definition is much more closely associated with the definition of topological entropy and, in particular, is useful in establishing the variational principle (Theorem \ref{thm:var-principle}). 

Kolmogorov introduced entropy to dynamical systems theory in order to classify Bernoulli shift over the group $\Z$. In \S \ref{sec:bernoulli} we define Bernoulli shifts over an arbitrary discrete probability-measure-preserving groupoid and compute their entropy (Theorem \ref{thm:bernoulli}), which as expected coincides with the Shannon entropy of the base space. This enables us to answer a question of Benjy Weiss on the isomorphism problem of non-free Bernoulli shifts (Theorem \ref{thm:nonfree}). The reader who is only interested in these two results need only read \S 2,3,4,7 for background.

\subsection{Organization}

We begin by defining groupoids in \S \ref{sec:groupoids}, sofic approximations in \S \ref{sec:sofic-approximations}, extensions of groupoids in \S \ref{sec:actions} and spanning and separating sets in \S \ref{sec:sss}. Sections \S\ref{sec:top} - \ref{sec:meas2} introduce the definitions of topological and measure sofic entropy and show that they do not depend on the choice of generating pseudo-metric or sigma-algebra. In \S \ref{sec:variational} the variational principle is established. In \S \ref{sec:measure-zero} we show how to define entropy for extensions that are class-bijective almost everywhere but not necessarily class-bijective. In  \S \ref{sec:bernoulli} we compute the entropy of a Bernoulli shift. In \S \ref{sec:non-free} we show that two isomorphic non-free Bernoulli shifts with sofic stabilizer distribution must have the same base space entropy. The last two sections are independent of \S\ref{sec:sss}, \ref{sec:top}, \ref{sec:semi}, \ref{sec:meas2}, \ref{sec:variational}.

{\bf Acknowledgements}. This paper owes a debt to David Kerr and Dykema-Kerr-Pichot for sharing early versions of \cite{DKP11, Ke12} from which I learned a lot about sofic groupoids and the partition approach to sofic entropy. Also a big thanks to Hanfeng Li for finding errors in previous versions.

\section{Discrete groupoids}\label{sec:groupoids}

A {\em groupoid}  is a small category in which every morphism is invertible. More precisely, a groupoid is a set of morphisms, denoted by $\sH^1$ together with a set of objects $\sH^0$, source and range maps $\so, \ra:\sH^1 \to \sH^0$, an injective inclusion map $i:\sH^0 \to \sH^1$, a set of composable pairs $\sH^2 \subset \sH^1\times \sH^1$ and a composition map $c:\sH^2 \to \sH^1$ satisfying
\begin{enumerate}
\item  $\so(i(x))=\ra(i(x))=x$ for all $x\in \sH^0$;
\item $\sH^2 = \{ (f,g):~ \so(f)=\ra(g)\}$;
\item $\so(c(f,g))=\so(g), \ra(c(f,g))=\ra(f)$ $\forall (f,g) \in \sH^2$,
\item for every $f\in \sH^1$ there is a unique element, denoted $f^{-1} \in \sH^1$ such that $c(f^{-1},f)  = i(\so(f))$ and $c(f,f^{-1})=i(\ra(f))$.
\end{enumerate}
To simplify notation, we let $\sH$ denote $\sH^1$ and identify $\sH^0$ as a subset of $\sH$ via the inclusion map. If $(f,g) \in \sH^2$ then we will write $fg :=c(f,g)$. For example, the last item above can be expressed by $f^{-1}f=\so(f)$ and $ff^{-1}=\ra(f)$.

%For a groupoid $\sH$ we denote the source and range maps by $\so$ and $\ra$, respectively,and write $\sH^0$ for the set of units of $\sH$.

%A groupoid $\sH$ induces an equivalence relation $E_\sH \subset \sH^0 \times \sH^0$ on $\sH^0$ in the usual way: $(x,y) \in E_\sH \Leftrightarrow \exists g \in \sH$ ($gx=y$). We let $[x]_\sH$ denote the  equivalence class of $x \in \sH^0$.

\subsection{Measurable groupoids and pmp groupoids}

A {\it measurable groupoid} is a groupoid $\sH$ with the structure of a standard Borel space
such that $\sH^0$ is a Borel set, $\sH^2$ is a Borel subset of $\sH^1\times \sH^1$ and the source, range, composition, and inversion maps are all Borel.

Let $\lb\sH\rb$ denote the set of all Borel subsets $f \subset \sH$ such that the restrictions of the source and range maps to $f$ are Borel isomorphisms onto their respective images. For $f \in \lb \sH \rb$, define $f^{-1}:=\{h^{-1}:~ h \in f\}$. The composition of $f,g \in \lb\sH\rb$ is defined by $fg:=\{h \in \sH:~ h=f'g' \textrm{ for some }f'\in f, g' \in g\}$. This makes $\lb\sH\rb$ an inverse semi-group called the {\em semi-group of partial automorphisms}. Observe that $\sH \subset \lb\sH\rb$ and every Borel subset $P \subset \sH^0$ is an element of $\lb\sH\rb$. In particular, if $f \in \lb \sH\rb$ and $x\in \so(f)$ then $fx$ is well-defined. Moreover, $fx = \so^{-1}(x) \cap f$. Note that $fx$ need not be in $\sH^0$. To remedy this, we define $f \cdot x := \ra(fx) \in \sH^0$. Similarly, if $P \subset \sH^0$ then we define $f \cdot P := \ra(f P)$.

%Unfortunately, this sets up a small ambiguity in the notation. 
%The reason for doing this is that for any $f\in \lb\sH\rb$, $1_{f\cdot P} = f 1_P f^{-1}$ but $1_{f\cdot P}$ is not necessarily equal to $f1_P$. 

We let $[\sH] \subset \lb\sH\rb$ denote the space of all Borel subsets $f \subset \sH$ such that the source and range maps restricted to $f$ are each Borel isomorphisms onto $\sH^0$. The set $[\sH]$ is a group under composition. We call it the {\em full group} of $\sH$.

%Through this association, we regard elements of $[\sH]$ as maps from 

%We consider an element $\phi \in [\sH]$ as 

%A {\it discrete measured groupoid} is a discrete measurable groupoid $\sH$ equipped with a Borel
%probability measure $\mu$ on $\sH$ such that $\mu (\sH^0 ) = 1$ and
%\[
%\mu (B) = \int_{\sH^0} \big| s^{-1} (\{ x \} )\cap B \big| \, d\mu (x) 
%= \int_{\sH^0} \big| r^{-1} (\{ x \} )\cap B \big| \, d\mu (x) 
%\]
%for all Borel sets $B\subseteq\sH$. 
A groupoid $\sH$ is {\em discrete} if $\so^{-1}(x)$ and $\ra^{-1}(x)$ are countable for every $x\in \sH^0$. A {\it discrete probability measured groupoid} is a discrete measurable groupoid $\sH$ paired with a Borel
probability measure $\nu$ on $\sH^0$ such that if $\nu_\so, \nu_\ra$ are the measures on $\sH$ given by
\begin{eqnarray*}
\nu_\so(B)&=&\int_{\sH^0} | \so^{-1} (x)\cap B | \, d\nu (x)\\
\nu_\ra(B)&=&\int_{\sH^0} | \ra^{-1} (x)\cap B | \, d\nu (x)
\end{eqnarray*}
for every Borel set $B\subseteq\sH$ then $\nu_\so$ is equivalent to $\nu_\ra$. If, in addition, $\nu_\so=\nu_\ra$ then we say $(\sH,\nu)$ is {\em pmp (probability-measure-preserving)}. In this article, we work exclusively with pmp groupoids. So we let $\nu$ denote $\nu_\so=\nu_\ra$ and note that $\nu$ restricted to $\sH^0$ is $\nu$, so no confusion should arise.

Given $f \in \lb\sH\rb$, the {\em trace} of $f$ is defined by $\tr_\sH(f):=\nu(\sH^0 \cap f)$. Also we define $|f|_\sH = \nu(f)$.

\subsection{Discrete topological groupoids}

A {\em discrete topological groupoid} is a discrete groupoid $\sH$ so that $\sH$ is equipped with a topology in which the structure maps (source, range, inverse and composition) are continuous. %We will also require that the composition map $c:\sH^2 \to \sH^1$ is open.

A {\em bisection} is an open subset $f \subset \sH$ such that the source and range maps restricted to $f$ are homeomorphisms onto their images which are open subsets of $\sH^0$. We say that $\sH$ is {\em \'etale} if every $g \in \sH$ is contained in a bisection. For most of the paper, the discrete topological groupoids $\sH$ that we study are \'etale and $\sH^0$ is compact and metrizable.

Let $[\sH]_{top}$ denote the set of all bisections $U \subset \sH$ such that the source and range maps restricted to $U$ are homeomorphisms onto $\sH^0$. Using Lemma \ref{lem:composition} below it can be checked that $[\sH]_{top}$ is a subgroup of $[\sH]$.

Let us suppose now that $(\sH,\nu)$ is a discrete topological pmp groupoid. Given a Borel set $A \subset \sH^0$, let $\partial A = \overline{A} \cap \overline{\sH^0\setminus A}$. Let $\cB_\partial(\sH^0,\nu)$ be the collection of all Borel subsets $A \subset \sH^0$ with $\nu(\partial A) = 0$.  Let $\lb \sH \rb_{top}$ be the set of all  elements of $\lb \sH \rb$ of the form $f = \cup_{i=1}^n f_i$ where 
%We show in Lemma \ref{lem:boundary-zero} that if $\sH^0$ is compact and metrizable then $\cB_\partial(\sH^0,\nu)$ is a dense subalgebra of the measure algebra of $(\sH^0,\nu)$.
\begin{itemize}
\item for each $i$ there exists a bisection $U_i$ with $\overline{f_i} \subset U_i$, $\overline{\so(f_i)} \subset \so(U_i)$ and $\overline{\ra(f_i)} \subset \ra(U_i)$;
\item $\{\so(f_i)\}_{i=1}^n \subset \cB_\partial(\sH^0,\nu)$ are pairwise disjoint;
\item $\{\ra(f_i)\}_{i=1}^n \subset \cB_\partial(\sH^0,\nu)$ are pairwise disjoint.
\end{itemize}
This definition is designed in order to make our two different definitions of measure entropy agree (in \S \ref{sec:meas1} and \ref{sec:meas2}); which is crucial to the proof of the variational principle. 

We would like to show that $\lb\sH\rb_{top}$ is closed under composition and inverses. First we need to show that $\cB_\partial(\sH^0,\nu)$ is an algebra:
\begin{lem}\label{lem:algebra}
Let $X$ be a topological space and $\lambda$ a Borel measure on $X$. Let $\cB_\partial(X,\lambda)$ be the collections of all Borel subsets $Y \subset X$ such that $\lambda(\partial Y)=0$ where $\partial Y = \overline{Y} \cap \overline{X \setminus Y}$. Then $\cB_\partial(X,\lambda)$ is closed under complementation, finite unions and finite intersections.
\end{lem}
\begin{proof}
To simply notation, for any $Y \subset X$, let $Y^c = X\setminus Y$. Because $\partial Y^c = \partial Y$, it is clear that $\cB_\partial(X,\lambda)$ is closed under complementation.

Let $A,B \in \cB_\partial(X,\lambda)$. Observe that $\partial(A \cup B) \subset \partial A \cup \partial B$. Hence $\lambda(\partial(A\cup B))=0$. Also, $\partial(A\cap B) \subset \partial A \cup \partial B$ which implies $\lambda(\partial(A\cap B))=0$. To see this, let $x \in \partial(A \cap B)$. Then there exist elements $\{y_n\}_{n=1}^\infty\subset (A \cap B)^c$ with $\lim_{n\to\infty} y_n =x$. Because $\{y_n\}_{n=1}^\infty \subset A^c \cup B^c$, either $\{y_n\}_{n=1}^\infty \cap A^c$ is infinite or $\{y_n\}_{n=1}^\infty \cap B^c$ is infinite. In the first case, $x \in \partial A$ and in the second $x \in \partial B$ which proves the claim. So $\cB_\partial(X,\lambda)$ is closed under finite unions and intersections.

\end{proof}

%In the next lemma, we make important use of the hypothesis that the composition map is open.

\begin{lem}\label{lem:composition}
If $U,V \subset \sH$ are bisections, then $UV$ is a bisection. If $f,g \in \lb \sH \rb_{top}$ then $f^{-1},fg \in \lb \sH \rb_{top}$. 
\end{lem}
\begin{proof}
Let $U,V \subset \sH$ be bisections. It is straightforward to check that the source and range of $UV$ are open sets and the source and range maps restricted to $UV$ are continuous bijections onto their images. We must show that the inverses of these restricted maps are continuous. For $x \in \so(UV)$, let $\phi(x)$ be the unique element in $\so^{-1}(x)\cap V$. For $v$ in the image of $\phi$, let $\psi(v) = (u,v)$ where $u\in U$ is the unique element with $(u,v) \in \sH^2$. Then $c\circ\psi\circ\phi:\so(UV) \to UV$ is the inverse of the source map (restricted to $UV$) (where $c:\sH^2 \to \sH$ is the composition map $(u,v) \mapsto uv$). Observe that $\phi$ is continuous because it agrees with $(\so|_V)^{-1}$ on its domain, where $(\so|_V)^{-1}$ is the inverse of the source map restricted to $V$. Also $c$ is continuous by definition of topological group. If we let $\Psi(v)=u$ if $\psi(v)=(u,v)$ then $\psi$ is continuous if and only if $\Psi$ is continuous. However $\Psi(v) = (\so|_U)^{-1} \circ \ra(v)$ is continuous. This shows that $\psi$ and therefore $(\so|_{UV})^{-1}$ is continuous. Similarly, $(\ra|_{UV})^{-1}$ is continuous. 

It remain to show that $UV$ is open. Let $u\in U, v\in V$ be composable and suppose $\{k_i\}_{i=1}^\infty$ is a sequence in $\sH$ with $\lim_{i\to\infty} k_i = uv$. Note $\ra(uv) \in \ra(U)$. Since $\ra(U)$ is open, if $i$ is sufficiently large then $\ra(k_i) \in \ra(U)$. Let $u_i \in U$ be the unique element such that $\ra(u_i) = \ra(k_i)$ (for $i \gg 0$). Because $\ra(k_i) \to \ra(uv)=\ra(u)$ it follows that $\ra(u_i) \to \ra(u)$ and therefore $u_i \to u$ as $i\to\infty$. So $\lim_{i\to\infty} u_i^{-1}k_i = v$. Because $V$ is open, $u_i^{-1}k_i \in V$ for $i\gg 0$. Thus $k_i = u_i (u_i^{-1}k_i) \in UV$ for $i\gg 0$ which shows that $UV$ is open. We have now verified that $UV$ is a bisection.

Let $f,g \in \lb \sH\rb_{top}$. By definition there are Borel sets $f_1,\ldots, f_n, g_1,\ldots, g_m \subset \sH$ and bisections $U_1,\ldots, U_n, V_1,\ldots, V_m$ such that
\begin{itemize}
\item $f=\cup_{i=1}^n f_i$ and $g=\cup_{j=1}^m g_j$;
\item $\overline{f_i} \subset U_i$, $\overline{g_j} \subset V_j$ for each $i,j$;
\item $\{\so(f_i)\}_{i=1}^n, \{\ra(f_i)\}_{i=1}^n, \{\so(g_j)\}_{j=1}^m, \{\ra(g_j)\}_{j=1}^m \subset \cB_\partial(\sH^0,\nu)$ are each pairwise disjoint.
\end{itemize}
Using that $f^{-1} = \cup_{i=1}^n f_i^{-1}$, it is straightforward to check that $f^{-1} \in \lb \sH \rb_{top}$.  Note that $fg = \cup_{i=1}^n \cup_{j=1}^m f_ig_j$. It is straightforward to check that the sources and ranges of $\{f_ig_j:~1\le i\le n, 1\le j\le m\}$ are pairwise disjoint. We claim that $\overline{f_ig_j} \subset U_iV_j$. Indeed, let $u_n \in f_i, v_n \in g_j$ and suppose $\lim_{n\to\infty} u_nv_n = w \in \sH$ exists. Observe that $\so(u_nv_n)=\so(v_n)$ so $\lim_{n\to\infty} \so(v_n)=\so(w)$. Because  $\so(\overline{g_j}) \subset \so(V_j)$, $\so(w) \in \so(V_j)$. In particular, there exists a unique $v_\infty \in V_j$ such that $\so(w)=\so(v_\infty)$. Because $\lim_{n\to\infty} \so(v_n) = \so(v_\infty)$ and $V_j$ is a bisection, $\lim_{n\to\infty} v_n = v_\infty$. Similarly, there is a unique element $u_\infty \in U_i$ with $\ra(u_\infty) = \ra(w)$ such that $\lim_{n\to\infty} u_n  = u_\infty$. It follows that $\lim_{n\to\infty} u_nv_n = u_\infty v_\infty =w \in U_iV_j$. Because $w$ is arbitrary, $\overline{f_ig_j} \subset U_iV_j$ as claimed.

It remains to show that $\so(f_ig_j), \ra(f_ig_j) \in \cB_\partial(\sH^0,\nu)$. Observe that 
\begin{eqnarray}\label{eqn:figj}
\so(f_ig_j) = \so(g_j) \cap g_j^{-1}\cdot \so(f_i) = \so(g_j) \cap V_j^{-1}\cdot (\so(f_i) \cap \ra(g_j)).
\end{eqnarray}
By Lemma  \ref{lem:algebra}, $\so(f_i) \cap \ra(g_j) \in \cB_\partial(\sH^0,\nu)$. Note
$$\partial (\so(f_i) \cap \ra(g_j)) \subset \overline{\ra(g_j)} \subset \ra(V_j).$$
Since the source and range maps restricted to $V_j^{-1}$ are measure-preserving homeomorphisms onto their images,
$$\nu( \partial V_j^{-1}\cdot (\so(f_i) \cap \ra(g_j) ) = \nu( V_j^{-1}\cdot (\partial (\so(f_i) \cap \ra(g_j) )) =  \nu( \partial (\so(f_i) \cap \ra(g_j)) = 0.$$
So $V_j^{-1}\cdot (\so(f_i) \cap \ra(g_j) \in \cB_\partial(\sH^0,\nu)$. Lemma  \ref{lem:algebra} and (\ref{eqn:figj}) now implies $\so(f_ig_j) \in \cB_\partial(\sH^0,\nu)$. The proof that $\ra(f_ig_j) \in \cB_\partial(\sH^0,\nu)$ is similar.
\end{proof}

\subsection{Examples}

\begin{example}\label{example:group}
A countable group $G$ can be thought of as a discrete pmp groupoid in which the set of objects $G^0=\{e\}$. The measure $\nu$ is simply counting measure.
% To be precise, we let $\sH=G$, $\sH^0=\{e\}$ and $\nu$ be the Dirac probability measure on $\sH^0$. 
\end{example}

\begin{example}\label{example:group-action}
Let $G$ be a countable group with a probability measure-preserving action $G \cc (X,\lambda)$. The groupoid associated to this action is $\sH=\{ (g,x):~x\in X,g \in G\}$ where $\sH^0=\{ (e,x):~x\in X\}$. The measure $\nu$ on $\sH$ is defined to be $c\times \lambda$ where $c$ is counting measure on $G$. The structure maps are defined by $\so(g,x)= (e,x), \ra(g,x)=(e,gx)$, $(g,x)^{-1}= (g^{-1},gx)$ and  $(h,gx)(g,x)=(hg,x)$.
\end{example}

\begin{example}
Let $G$ be a countable discrete group acting by homeomorphisms on a compact metric space $X$. The topological groupoid associated to this action is $\sH = G \times X$ with the product topology. The structure maps are defined as in the previous example. Note that $\sH$ is \'etale.
\end{example}

\begin{example}
Recall that a discrete pmp (probability measure-preserving) equivalence relation consists of a standard probability space $(X,\lambda)$ together with Borel equivalence relation $E \subset X \times X$ such that every $E$-class is at most countable and, if $c$ denotes the counting measure on $X$ then $\lambda \times c|_E = c\times \lambda|_E$. This can be represented as a discrete pmp groupoid by setting $\sH=E$, $\sH^0 = \{(x,x) \in E:~ x \in X\}$ and $\nu=\lambda \times c|_E$. The structure maps are defined by $\so(x,y)=(y,y), \ra(x,y)=(x,x), (x,y)^{-1}=(y,x)$ and $(x,y)(y,z)=(x,z)$.
\end{example}

\begin{example}\label{example:d}
Let $d \in \N$. The {\em full groupoid} on $\{1,\ldots,d\}$ is $\Delta_d:=\{1,\ldots,d\}^2$. The unit space is $\Delta^0_d:=\{ (i,i):~ 1 \le i \le d\}$. The structure maps are defined by $\so(i,j)=(j,j), \ra(i,j)=(i,i), (i,j)^{-1}=(j,i)$ and $(i,j)(j,k)=(i,k)$. Let $\zeta_d(E) = |E|/d$ for every set $E \subset \Delta_d$. Thus $(\Delta_d,\zeta_d)$ is a pmp groupoid. Note that $[\Delta_d]$ is isomorphic with the symmetric group on $\{1,\ldots,d\}$ while $\lb \Delta_d\rb$ is the collection of all subsets $f \subset \Delta_d$ such that the two projection maps $\so:f \to \so(f), \ra:f \to \ra(f)$ are bijections. To make the notation simpler, we set $[d]:=[\Delta_d]$, $\lb d \rb:=\lb \Delta_d \rb$, $\tr_d:=\tr_{\Delta_d}, | \cdot |_d := |\cdot|_{\Delta_d}$. So $\tr_d(f)=|f \cap \Delta^0_d|/d$ and $|f|_d = |f|/d$. 
\end{example}

\section{Sofic approximations}\label{sec:sofic-approximations}

Let $(\sH,\nu)$ be a pmp discrete groupoid. We use notation as in Example \ref{example:d}. For $d>0$, let $\Map(\lb\sH\rb,\lb d \rb)$ be the set of all functions from $\lb\sH\rb$ to $\lb d \rb$. This set carries a natural Borel structure as follows. Given a finite set $F \subset \lb \sH \rb$ and $\sigma:\lb \sH \rb \to \lb d \rb$, let $N(\sigma,F)=\{\sigma' \in \Map(\lb\sH\rb,\lb d \rb):~ \sigma'(f)=\sigma(f)~\forall f \in F\}$. We consider $\Map(\lb \sH \rb, \lb d \rb)$ with the Borel structure generated by all such $N(\sigma,F)$. 

%For any positive integer $d$, let $\lb d \rb$ be the set of all partial automorphisms of the finite set $\{1,\ldots, d\}$. More precisely, every $f \in \lb d \rb$ is a bijection $f: \dom(f) \to \ran(f)$ with $\dom(f), \ran(f) \subset \{1,\ldots, d\}$. 

%symmetric group on $[d]:=\{1,\ldots,d\}$.
 %For each $f,g\in \lb d \rb$, we define 
%\begin{eqnarray*}
%\tr_d(f)&:=& d^{-1} \# \{p \in \dom(f):~ f(p)=p\}\\
%|f-g|_d &:=& d^{-1} \# \left(\textrm{graph}(f) \vartriangle \textrm{graph}(g) \right)
%\end{eqnarray*}
%where e.g., $\textrm{graph}(f)=\{ (x,y) \in \{1,\ldots,d\} \times\{1,\ldots,d\}:~ fx=y\}$. 

\begin{notation}
We write $X \subset_f Y$ to mean ``$X$ is a finite subset of $Y$''.
\end{notation}
%Note that if $s_1, \ldots, s_k$ are elements of $\llbracket \sH \rrbracket$ with pairwise disjoint domains and pairwise disjoint ranges (in which case we say that the $s_i$ are {\em pairwise orthogonal}), then  $\bigcup_{i=1}^k s_i\in \llbracket \sH \rrbracket$. 
%For any $F \subset_f \lb\sH\rb$, let $F_{\pm}$ for the finite subset of $\llbracket\sH\rrbracket$ defined by $F_{\pm}:=F\cup \{s^{-1},\, s\in F\}\cup \{\sH^0\}$, and $F^n$ for the finite subset of $\llbracket\sH\rrbracket$ consisting of all products of $n$ elements of $F$, with the convention that $F_{\pm}^n$ refers to the subset $(F_{\pm})^n$. Also let $\bs F$ denote the collection containing $\bigcup_{i=1}^k s_i$ whenever $s_1,\ldots, s_k \in F$ are pairwise orthogonal and $\bs F_\pm^n$ refers to $\bs ((F_\pm)^n)$.
Let $F \subset_f \lb \sH \rb,  \delta>0$. We say that a map $\sigma : \llbracket\sH\rrbracket \to \llbracket d \rrbracket$ is 
{\it $(F ,\delta )$-multiplicative} if 
\[
|  \sigma (s  t ) \vartriangle \sigma (s )\sigma (t ) |_d < \delta 
\]
for all $s, t \in F$ and $(F,\delta)$\emph{-trace-preserving} if
\[
| \tr_{d}(\sigma (s)) - \tr_{\sH}(s) | < \delta
\]
for all $s\in F$.

\begin{defn}[Sofic approximation]\label{defn:sofic-approximation}
Let $J$ be a directed set. For each $j \in J$, let $d_j \in \Nb$ and $\P_j$ be a Borel probability measure on $\Map(\lb\sH\rb,\lb d_j \rb)$. We say that the family $\P=\{\P_j\}_{j\in J}$ is a {\em sofic approximation} to $(\sH,\nu)$ if
\begin{enumerate}
\item for every $F \subset_f \lb\sH\rb$ and $\delta>0$,
$$\lim_{j\to J}  \P_j(\{\sigma \in \Map(\lb \sH \rb,\lb d_j \rb):~ \sigma \textrm{ is } (F,\delta)\textrm{-trace-preserving} \}) = 1.$$
\item for every $F \subset_f \lb \sH \rb,  \delta>0$, there exists $j \in J$ such that $j' \ge j$ implies $\P_{j'}$-almost every $\sigma$ is $(F,\delta)$-multiplicative.
%$$\P_{j'}(\{\sigma \in \Map(\lb \sH \rb,\lb d_j \rb):~ \sigma \textrm{ is } (F,n,\delta)-\textrm{multiplicative} \}) = 1.$$
%for every $f, f' \in \lb\sH\rb$ and $\epsilon>0$,
%$$\lim_{j \to J} \P_j( \{\sigma \in \Map(\lb \sH \rb,\lb d_j \rb):~ \tr_d(\sigma(ff')\sigma(f')^{-1}\sigma(f)^{-1}) >1- \epsilon) = 1.$$
\item $\lim_{j\to J} d_j = +\infty$,
\item for every $f, f' \in \lb \sH \rb$ with $\nu(f \vartriangle f')=0$, $\sigma(f)=\sigma(f')$ for $\P_j$-a.e. $\sigma$. 
\end{enumerate}

\end{defn}

The groupoid $(\sH,\nu)$ is {\em sofic} if it admits a sofic approximation. The next two lemmas are of basic general use.

\begin{lem}\label{lem:basic-formulas1}
Let $(\sH,\nu)$ be a pmp groupoid. For any $s,t,t',u \in \lb\sH\rb$, 
$$| s t u \vartriangle s t' u |_\sH \le |t \vartriangle t'|_\sH.$$
\end{lem}
\begin{proof}
This is an exercise.
\end{proof}

\begin{lem}\label{lem:basic-formulas2}
Let $F \subset_f \lb \sH\rb$ and $\delta>0$. Suppose $\sigma:\lb\sH\rb\to\lb d \rb$ is $(F,\delta)$-multiplicative. Then for any Borel set $P \subset \sH^0$ with $P \in F$ and any $f$ with $f,f^{-1},\so(f),\ra(f) \in F,$ %and $\epsilon_1,\ldots, \epsilon_n \in \{-1,1\}$ such that all products $f_1^{\epsilon_1} \cdots  f_k^{\epsilon_k} \in F$ for $1\le k \le n$ we have
\begin{enumerate}
%\item $| P \vartriangle Q|_\sH = |P - Q|_\sH$. 
\item $|\sigma(P) \vartriangle  (\sigma(P) \cap \Delta^0_d)|_d \le \delta$;
\item  $|\so(\sigma(f)) \vartriangle \sigma(\so(f))|_d \le 10\delta$;
\item $|\ra(\sigma(f)) \vartriangle \sigma(\ra(f))|_d \le 10\delta$;
\item $|\sigma(f^{-1}) \vartriangle \sigma(f)^{-1} |_d \le 15\delta$.
%\item $| \sigma(f_1^{\epsilon_1}\cdots f_n^{\epsilon_n}) \vartriangle \sigma(f_1)^{\epsilon_1}\cdots \sigma(f_n)^{\epsilon_n}|_d \le 3n \delta$.
\end{enumerate}

\end{lem}

\begin{proof}
%The first item is trivial. 
%By definition, for any $f,g \in \lb d \rb$, $|f-g|_d = d^{-1} | f \vartriangle g |$. 
Let $P \subset \sH^0$ with $P \in F$. Note
$$\sigma(P) \setminus (\sigma(P) \cap \Delta^0_d)  \subset \sigma(P) \setminus \sigma(P) \sigma(P).$$
To see this, observe that if $(j,i) \in \sigma(P) \setminus (\sigma(P) \cap \Delta^0_d)$ then $i \ne j$. If $(j,i) \in \sigma(P)\sigma(P)$ then there exists $k$ such that $(j,k),(k,i) \in \sigma(P)$. Because $(k,i),(j,i) \in \sigma(P) \in \lb d \rb$, we must have $j=k$. So $(j,j) \in \sigma(P)$ which implies (because $(j,i) \in \sigma(P)$) that $i=j$, a contradiction. 

Because $\sigma$ is $(F,\delta)$-multiplicative,
\begin{eqnarray*}
|\sigma(P) \vartriangle (\sigma(P) \cap \Delta^0_d)|_d \le | \sigma(P) \vartriangle \sigma(P) \sigma(P)|_d =| \sigma(PP) \vartriangle \sigma(P) \sigma(P)|_d \le \delta.
\end{eqnarray*}
This proves the first item.

Now let $f\in F$ be such that $f^{-1},\so(f),\ra(f)\in F$. Because $\sigma$ is $(F,\delta)$-multiplicative,
\begin{eqnarray*}
2\delta &\ge& |\sigma(f)\sigma(f^{-1})\sigma(f) \vartriangle \sigma(f)|_d\\
2\delta &\ge& |\sigma(f^{-1})\sigma(f)\sigma(f^{-1}) \vartriangle \sigma(f^{-1})|_d.
\end{eqnarray*}
The first inequality above implies the range of $\sigma(f)$ is contained in the source of $\sigma(f^{-1})$ up to a $2\delta$-measure subset. The second inequality implies the range of $\sigma(f)$ contains the source of $\sigma(f^{-1})$ up to a $2\delta$-measure subset. Therefore,
%$\sigma(f)\sigma(f^{-1})(x)=x$ for all $x$ in the range of $\sigma(f)$ except for a $2\delta$-measure set. In particular, $|\ra(\sigma(f)) \setminus \so(\sigma(f^{-1})) |_d \le 2\delta$. The second inequality implies $\sigma(f^{-1})\sigma(f) (x)=x$ for every $x$ in the source of $\sigma(f^{-1})$ except for a $2\delta$-measure set. Therefore, $|\so(\sigma(f^{-1})) \setminus \ra(\sigma(f))  |_d \le 2\delta$. This shows
$$|\so(\sigma(f^{-1})) \vartriangle \ra(\sigma(f))  |_d \le 4\delta.$$
Similar considerations imply $|\so(\sigma(f)) \vartriangle \ra(\sigma(f^{-1}))  |_d \le 4\delta$.

Because of $(F,\delta)$-multiplicativity, $|\sigma(f)\sigma(f^{-1}) \vartriangle \sigma(\ra(f))|_d \le \delta$. Therefore, we have 
$$|\so(\sigma(f)\sigma(f^{-1})) \vartriangle \so(\sigma(\ra(f)))|_d \le \delta.$$
Because $\ra(\sigma(f^{-1}))$ is $4\delta$-close to $\so(\sigma(f))$, it follows that $\so(\sigma(f)\sigma(f^{-1}))$ is $4\delta$-close to $\so(\sigma(f^{-1}))$ which is $4\delta$-close to $\ra(\sigma(f))$. Thus
$$| \ra(\sigma(f)) \vartriangle \so(\sigma(\ra(f)))|_d \le 9\delta.$$
From item (1) it follows that $|\so(\sigma(\ra(f))) \vartriangle \sigma(\ra(f))|_d \le \delta$. So we obtain $|\ra(\sigma(f)) \vartriangle \sigma(\ra(f))|_d \le 10\delta.$ This proves item (3). Item (2) is similar.

%We showed above that $\sigma(f)\sigma(f^{-1})(x)=x$ for all $x$ in the range of $\sigma(f)$ except for a $2\delta$-measure set. So it follows that $\sigma(\ra(f))(x)=x$ for all $x$ in the range of $\sigma(f)$ except for a $3\delta$-measure set. Thus
%$$|\ra(\sigma(f)) \vartriangle \sigma(\ra(f))|_d \le 5 \delta.$$
%This proves item (3). Item (2) is similar.

We now have
\begin{eqnarray*}
| \sigma(f^{-1}) \vartriangle \sigma(f)^{-1}|_d &\le&  |\sigma(f^{-1})  \vartriangle \sigma(f)^{-1}\sigma(f)\sigma(f^{-1})|_d+ |\sigma(f)^{-1}\sigma(f)\sigma(f^{-1}) \vartriangle \sigma(f)^{-1}|_d\\
&\le&\delta + |\sigma(f^{-1})  \vartriangle \so(\sigma(f))\sigma(f^{-1})|_d+ |\sigma(f)^{-1}\sigma(ff^{-1}) \vartriangle \sigma(f)^{-1}|_d\\
&\le&5\delta + |\sigma(f^{-1})  \vartriangle \ra(\sigma(f^{-1}))\sigma(f^{-1})|_d+ |\sigma(f)^{-1}\sigma(\ra(f)) \vartriangle \sigma(f)^{-1}|_d\\
&\le&15\delta + |\sigma(f)^{-1}\ra(\sigma(f)) \vartriangle \sigma(f)^{-1}|_d =15\delta.
\end{eqnarray*}
\end{proof}

%{\bf Comment}: We really should check to see if both conditions are necessary to prove the results. If not, then we should invent `weak sofic approximations' etc.

\section{Actions, extensions and factors}\label{sec:actions}

Let $\sG, \sH$ be measurable groupoids. A map $\pi:\sG \to \sH$ is a {\em groupoid morphism} if $\pi(fg)=\pi(f)\pi(g)$ for every $(f,g)\in \sG^2$, $\pi(f)^{-1}=\pi(f^{-1})$ for every $f\in \sG$ and $\pi(\sG^0) \subset \sH^0$. It is {\em class-bijective} if for every $a \in \sG^0$, the restriction of $\pi$ to $\so^{-1}(a)$ is a bijection onto $\so^{-1}(\pi(a))$ and the restriction of $\pi$ to $\ra^{-1}(a)$ is also a bijection onto $\ra^{-1}(\pi(a))$. If $\pi$ is also surjective then we say $\sG$ is a {\em class-bijective extension} of $\sH$ or, equivalently, $\sH$ is a {\em class-bijective factor} of $\sG$. %If $(\sG,\mu)$ is a measured groupoid then we say $\pi$ is {\em class-bijective almost everywhere} if for $\mu$-a.e. $a\in \sG^0$, the restriction of $\pi$ to $\so^{-1}(a)$ is a bijection onto $\so^{-1}(\pi(a))$ and the restriction of $\pi$ to $\ra^{-1}(a)$ is also a bijection onto $\ra^{-1}(\pi(a))$.

We say that $\pi$ is {\em pmp (probability-measure-preserving)} if $\pi_*\mu=\nu$ and $(\sG,\mu), (\sH,\nu)$ are pmp groupoids. If $\pi: (\sG,\mu) \to (\sH,\nu)$ is class-bijective then $\pi^{-1}: \lb\sH\rb \to \lb\sG\rb$ is a homomorphism. In the topological category we have the following similar result: % If $\sG, \sH$ are topological groupoids and $\pi$ is continuous and $\sH^0$ is compact, then $\pi^{-1}([\sH]_{top}) \subset [\sG]_{top}$. %If, in addition, $\sH$ is \'etale then $[\sH]_{top} \subset \lb\sH\rb_{top}$.
\begin{lem}\label{lem:top-inclusion}
Let $\sG$ and $\sH$ be topological groupoids and $\pi:\sG\to \sH$ be continuous and class-bijective. Assume $\sG$ is \'etale. Then $\pi^{-1}([\sH]_{top}) \subset [\sG]_{top}$. Moreover, if $(\sH,\nu)$, $(\sG,\mu)$ are pmp groupoids and $\pi$ is measure-preserving then $\pi^{-1}(\lb \sH \rb_{top}) \subset \lb \sG \rb_{top}$.
\end{lem}
\begin{proof}
%Let $f \in [\sH]_{top}$. It is straightforward to check that the source and range maps restricted to $\pi^{-1}(f)$ and continuous bijections onto their images. It remains to show that $(\so|_{\pi^{-1}(f)} )^{-1}:\sG^0 \to \pi^{-1}(f)$ is continuous. For this, let $g \in \pi^{-1}(f)$. Because $\sG^0$ is \'etale, there is a bisection $U$ with $g \in U$. After replacing $U$ with $U \cap \pi^{-1}(f)$ 

We claim that if $U \subset \sH$ is a bisection then $\pi^{-1}(U) \subset \sG$ is also a bisection. The only nontrivial part of this statement is showing that the map $x \in \so(\pi^{-1}(U)) \mapsto \pi^{-1}(U)x \in \pi^{-1}(U)$ is continuous (and similarly with the range map replacing the source map). Let $x \in \so(\pi^{-1}(U)) \subset \sG^0$. Because $\sG$ is \'etale, there exists a bisection $O \subset \sG$ with $\pi^{-1}(U)x \in O$. After replacing $O$ with $O \cap \pi^{-1}(U)$ if necessary, we may assume $O \subset \pi^{-1}(U)$.

Let $N$ be an open neighborhood of $x$ in $\sG^0$. Then $N \cap \so(O)$ is an open neighborhood of $x$ in $\sG^0$ and the map $y \in N \cap \so(O) \mapsto Oy$ is continuous since $O$ is a bisection. Moreover, $Oy = \pi^{-1}(U)y$ since $O \subset \pi^{-1}(U)$. This shows that $y \mapsto \pi^{-1}(U)y$ is continuous in a neighborhood of $x$. Since $x$ is arbitrary, this map is continuous as required. The range map is similar. So $\pi^{-1}(U) \subset \sG$ is a bisection.

It is now straightforward to verify the claims of the lemma.
\end{proof}

If $\pi$ is understood then given $x \in \sG$ and $f \in \lb\sH\rb$ we let $f x$ denote $\pi^{-1}(f)x$. If $x \in \sG^0$ then we let $f\cdot x$ denote  $\ra(\pi^{-1}(f)x)$ ($=\ra(fx)$).

%If $(\sG,\mu)$ and $(\sH,\nu)$ are pmp groupoids and $\pi_*\mu=\nu$ then we say that $\pi$ is {\em probability-measure-preserving} which we abbreviate by pmp. If $\sG$ and $\sH$ are topological groupoids then we will usually require $\pi$ to be continuous.

%If $\sG$ and $\sH$ are topological groupoids, then we require $\pi$ to be a continuous map. If $(\sG,\mu)$ and $(\sH,\nu)$ are discrete pmp groupoids then we require $\pi_*\mu=\nu$ and the condition that $\pi$ be class-bijective is relaxed to $\pi$ being class-bijective almost everywhere.

%Let $(\sG ,\mu )$ and $(\sH ,\nu )$ be discrete pmp groupoids. A groupoid morphism $\pi:(\sG,\mu) \to (\sH,\nu)$ (defined, possibly, only on a set of full measure) is a factor map if $\pi_*\mu = \nu$. It is {\em class-bijective} if for a.e. $a \in \sG^0$, the restriction of $\pi$ to $\so^{-1}(a)$ is a bijection onto $\so^{-1}(\pi(a))$ and the restriction of $\pi$ to $\ra^{-1}(a)$ is also a bijection onto $\ra^{-1}(\pi(a))$. In this case, we say $(\sG,\mu)$ is a {\em class-bijective extension} of $(\sH,\nu)$ or, equivalently, $(\sH,\nu)$ is a {\em class-bijective factor} of $(\sG,\mu)$. %In this case, it is natural to think of $(\sG,\mu)$ as an {\em action} of $(\sH,\nu)$. To be precise, if $g \in \sH$ then 

%If $\pi:\sG \to \sH$ is class-bijective, $x \in \sG^0$ and $f \in \sH$ is such that $\so(f) = \pi(x)$ then we define $fx \in \sG^0$ by $fx=f'x$ where $f' \in \sG$ is the unique element satisfying $\so(f')=x, \pi(f')=f$. 

\begin{example}\label{E-action2}
Let $G$ be a countable discrete group with pmp actions $G \cc (X,\mu)$ and $G \cc (Y,\nu)$. Suppose $\pi:X \to Y$ is a $G$-equivariant factor map (so $\pi_*\mu=\nu$). 
 Then the groupoid associated to $G \cc (X,\mu)$ (as in Example \ref{example:group-action}) is a pmp class-bijective extension of the groupoid associated to $G \cc (Y,\nu)$. The case when $Y$ is a single point is especially interesting because then the groupoid associated to $G \cc (Y,\nu)$ is identified with $G$ itself. Therefore, class-bijective extensions of groupoids generalize group actions. 
 \end{example}
 
 \begin{example}
The previous example can be generalized as follows. Let $(\sH,\nu)$ be a discrete pmp groupoid. Let $\alpha:\sH \to \textrm{Aut}(X,\lambda)$ be a measurable cocycle into the group of automorphisms of a standard probability space. This means that $\alpha(f)\alpha(g)=\alpha(fg)$ for any $(f,g) \in \sH^2$. Associated to such a cocycle is a pmp class-bijective extension $\pi:(\sG,\mu) \to (\sH,\nu)$ defined as follows. Let $\sG=\sH \times X, \sG^0 = \sH^0\times X, \mu = \nu \times \lambda$ with the structure maps defined by $\so(h,x)=(\so(h),x), \ra(h,x)=(\ra(h),\alpha(h)x), (h,x)^{-1}=(h^{-1},\alpha(h)x), (g,\alpha(h)x)(h,x)=(gh,x)$ for $h,g \in \sH, x\in X$. Let $\pi:\sG\to \sH$ be projection onto the first coordinate. An exercise shows this is a pmp class-bijective extension.
\end{example}

 We say two class-bijective extensions $\pi_i:(\sG_i,\mu_i) \to (\sH_i,\nu_i)$ ($i=1,2$) are {\em measure-isomorphic} if there exist measure-preserving isomorphisms $\Phi:(\sG_1,\mu_1) \to (\sG_2,\mu_2)$ and $\Psi:(\sH_1,\nu_1) \to (\sH_2,\nu_2)$ such that $\pi_2 \Phi = \Psi \pi_1$ almost everywhere. Similarly, two continuous class-bijective extensions $\pi_i:\sG_i \to \sH_i$ ($i=1,2$) are {\em isomorphic} if there exist continuous isomorphisms $\Phi:\sG_1 \to \sG_2$ and $\Psi:\sH_1 \to \sH_2$ such that $\pi_2 \Phi = \Psi \pi_1$.

\section{Spanning and separated sets}\label{sec:sss}

We use spanning and separated sets as a tool to define the topological entropy of a continuous class-bijective groupoid extension in the next section. Here we set notation and obtain a well-known result. Recall that a pseudo-metric $\rho$ on a set $X$ possesses all the properties of a metric except nondegeneracy: it can happen that $x\ne y \in X$ but $\rho(x,y)=0$. 

\begin{defn}
Given a pseudo-metric space $(Z,\rho)$ and $\epsilon>0$, a subset $Y \subset Z$ is {\em $(\rho,\epsilon)$-separated} if for every $y_1\ne y_2 \in Y$ $\rho(y_1,y_2) > \epsilon$. For $X \subset Z$, let $N_\epsilon(X,\rho)$ denote the maximum cardinality of a $(\rho,\epsilon)$-separated subset $Y \subset X$.

Let $X,Y \subset Z$. We say $Y$ {\em $(\rho,\epsilon)$-spans} $X$ if for every $x \in X$ there exists $y \in Y$ with $\rho(x,y) < \epsilon$. Let $N'_\epsilon(X,\rho)$ denote the minimum cardinality of a set $Y \subset Z$ which $(\rho,\epsilon)$-spans $X$. This number implicitly depends on $Z$.
\end{defn}

\begin{lem}\label{lem:span-sep}
For any pseudo-metric space $(Z,\rho)$, $X \subset Z$, $\epsilon>0$,
$$N'_{2\epsilon}(X,\rho) \le N_\epsilon(X,\rho) \le N_{\epsilon/2}'(X,\rho).$$
%$$N_\epsilon^{sparse}(Y,\rho) \le N_\epsilon^{sep}(Y,\rho) \le N_{\epsilon/2}^{sparse}(Y,\rho).$$
\end{lem}

\begin{proof}
Let $Y_1 \subset X$ be a maximal $(\rho,\epsilon)$-separated subset. Then $Y_1$ $(\rho,2\epsilon)$-spans $X$. Therefore $N_{2\epsilon}'(X,\rho) \le N_\epsilon(X,\rho)$.

Let $Y_2 \subset Z$ be a minimal $(\rho,\epsilon/2)$-spanning subset for $X$. Then the $\epsilon/2$-neighborhood of any point $y \in Y_2$ contains at most $1$ point of $Y_1$. Moreover, every point of $Y_1$ is contained in the $\epsilon/2$-neighborhood of some point of $Y_2$. Therefore $|Y_2| \ge |Y_1|$ which implies $N_\epsilon(X,\rho) \le N'_{\epsilon/2}(X,\rho)$.

%Let $Z_3 \subset Y$ be a maximal $(\rho,\epsilon)$-sparse subset. Then $Z_3$ is $\epsilon$-separating. Therefore $N_\epsilon^{sparse}(Y,\rho) \le N_\epsilon^{sep}(Y,\rho)$.
\end{proof}

\section{Topological entropy}\label{sec:top}
%{\bf The input}. 

We assume as given: two discrete separable topological groupoids $\sG,\sH$ such that $\sG^0$ and $\sH^0$ are compact metrizable spaces, a continuous class-bijective factor map $\pi:\sG \to \sH$, a Borel probability measure $\nu$ on $\sH$ making $(\sH,\nu)$ a pmp groupoid, a sofic approximation $\P=\{\P_j\}_{j\in J}$ to $(\sH,\nu)$, a bias $\beta$ (defined below) and a number $p\in [1,\infty]$. From this and a choice of generating pseudo-metric, we will define the {\em sofic topological entropy of $\pi$ with respect to $(\P,p,\beta)$}.  

%Given $x \in \sG^0$ and $f \in \lb\sH\rb$, if $\pi(x) \notin \so(f)$ then $fx$ is not defined. In this case, we set $fx:=*$ where $*$ is a special symbol. Thus $fx \in \sG^0 \cup \{*\}$. 

Given an integer $d>0$, we will write $x \in (\sG^0)^d$ as $x=(x_1,\ldots,x_d)$. Given $f \in \lb\sH\rb$, we let $f\cdot x:=(f\cdot x_1,\ldots,f \cdot x_d)$. If $\pi(x_i) \notin \so(f)$ then $f\cdot x_i$ is not defined. In this case, we set $f\cdot x_i:=*$ where $*$ is a special symbol. Thus $f\cdot x \in (\sG^0 \cup \{*\})^d$. 

For ease of notation, we will identify $\Delta^0_d$ with $\{1,\ldots, d\}$ and, for $\sigma:\lb \sH\rb \to \lb d\rb$, $f \in \lb\sH\rb$ and $i\in \{1,\ldots, d\}$ we will write $\sigma(f)i \in \{1,\ldots, d\}$ instead of $\sigma(f)\cdot i$. With $x$ as above, we define $x \circ \sigma(f) := (x_{\sigma(f)1}, \ldots, x_{\sigma(f)d} )$. If $i \notin \so(\sigma(f))$ then $x_{\sigma(f)i}$ is not well-defined. In this case, we set $x_{\sigma(f)i}:=*$. So $x \circ \sigma(f) \in (\sG^0 \cup \{*\})^d$.

Let $\rho$ be a continuous pseudo-metric on $\sG^0$. We extend $\rho$ to $\sG^0\cup\{*\}$ by setting $\rho(*,*) = 0$ and $\rho(*,x)=\max \{\rho(y,z):~y,z\in \sG^0\}$ for any $x\in \sG^0$. This induces pseudo-metrics on the $d$-fold Cartesian product of $\sG^0 \cup \{*\}$ by 
$$\rho_2(x,x') : = \left( \frac{1}{d} \sum_{i=1}^d \rho(x_i,x_i')^2 \right)^{1/2}, \quad \rho_\infty(x,x'):= \max_{1\le i \le d} \rho(x_i,x_i').$$

\begin{defn}[Approximate partial orbits]
Let $C(\sH^0)$ denote the space of continuous complex-valued functions on $\sH^0$. Given a map $\sigma:\lb\sH\rb \to \lb d \rb$, finite sets $F \subset \lb\sH\rb$, $K \subset C(\sH^0)$ and $\delta>0$, we let $Orb_\nu(\pi,\sigma,F,K,\delta,\rho)$ be the set of all $d$-tuples $(x_1,\ldots, x_d)$ (with $x_i \in \sG^0$) such that
\begin{eqnarray*}
\delta &>& \rho_2( f\cdot x,  x \circ \sigma(f) ) \quad \forall f\in F,\\
\delta &>& \left| d^{-1}\sum_{i=1}^d k(\pi(x_i)) - \int k~d\nu \right|  \quad \forall k \in K.
\end{eqnarray*}
\end{defn}

%A {\em biased sofic approximation} is a sofic approximation $\{\P_j\}_{j\in J}$ together with a bias $\beta \in \{-1,+1\}$. We write $\P$ to denote a biased sofic approximation, leaving the choice of $\beta$ implicit. 
\begin{defn}
A {\em bias} $\beta$ for $J$ is either an element of $\{-,+\}$ or an ultrafilter on $J$ with the property that for every $j\in J$, the set $\{j' \in J:~j' \ge j\} \in \beta$. Given a function $\Phi: J \to \R$, if $\beta$ is an ultrafilter then the ultralimit $\lim_{j\to \beta} \Phi(j)$ is well-defined. Otherwise, define
\begin{displaymath}
\lim_{j \to \beta} \Phi(j) := \left\{ \begin{array}{cc}
\liminf_{j\to J} \Phi(j) & \textrm{ if } \beta = - \\
\limsup_{j\to J} \Phi(j) & \textrm{ if } \beta = +
%\lim_{j \to \beta} \Phi(j) & \textrm{ if } \beta \textrm{ is an ultrafilter on $J$}
\end{array}\right.
\end{displaymath}
%If $\beta$ is an ultrafilter on $J$ then we define $\lim^\beta_{j\to J} \Phi(j) = \lim_{j \to\beta} \Phi(j)$ to be the ultralimit. We call $\beta$ a {\em bias}.
\end{defn}

\begin{notation}
Given a function $\phi$ on $\Map(\lb\sH\rb, \lb d_j \rb)$, we let $\|\phi \|_{p,\P_j}$ denote the $L^p$ norm of $\phi$ with respect to $\P_j$. For example, if $1\le p <\infty$ then
$$\| N_\epsilon( Orb_\nu(\pi,\cdot, F,K,\delta,\rho), \rho_2)\|_{p, \P_j} = \left( \int  N_\epsilon( Orb_\nu(\pi,\sigma,F,K,\delta,\rho), \rho_2)^p ~d\P_j(\sigma)\right)^{1/p}.$$
%\begin{notation}
%We write $X \subset_f Y$ to mean that $X$ is a finite subset of $Y$.
%\end{notation}
\end{notation}
\begin{defn}
Recall that we write $X \subset_f Y$ to mean that $X$ is a finite subset of $Y$. Define
$$h^\beta_{\P,p}(\pi,\rho,2) := \sup_{\epsilon>0} \inf_{\delta>0} \inf_{F \subset_f \lb\sH\rb_{top} }\inf_{K \subset_f C(\sH^0)} \lim_{j \to \beta} \frac{1}{d_j} \log\| N_\epsilon( Orb_\nu(\pi,\cdot,F,K,\delta,\rho), \rho_2)\|_{p, \P_j};$$
$$h^\beta_{\P,p}(\pi,\rho,\infty) := \sup_{\epsilon>0} \inf_{\delta>0} \inf_{F \subset_f \lb\sH\rb_{top} }\inf_{K \subset_f C(\sH^0)} \lim_{j \to \beta} \frac{1}{d_j} \log\| N_\epsilon( Orb_\nu(\pi,\cdot,F,K,\delta,\rho), \rho_\infty)\|_{p, \P_j}.$$
\end{defn}
%$$h_{{\mathfrak P}}(\pi,\rho,\infty) := \sup_{\epsilon>0} \inf_{\delta>0} \inf_{F \subset_f \lb \sH \rb_{top} }\inf_{K \subset_f C(\sH^0)} \lim_{j \to \beta} \frac{1}{d_j} \log\| N_\epsilon( Orb_\nu(\pi,\cdot,F,K,\delta,\rho), \rho_\infty)\|_{p, \P_j}.$$

\begin{remark} For most of the paper, the choices of $p$ and $\beta$ are irrelevant. Therefore, we will write $h_{\P}(\pi,\rho,2)$ instead of $h^\beta_{\P,p}(\pi,\rho,2)$ and $h_{\P}(\pi,\rho,\infty)$ instead of $h^\beta_{\P,p}(\pi,\rho,\infty)$, leaving $p$ and $\beta$ implicit. The order of the supremums, infimums and limits above is important with the exception that one can permute the three infimums without affecting the definition. 
\end{remark}

\begin{remark} There is a certain useful monotonicity phenomenon in the formulas above: the quantity
 $$\frac{1}{d_j} \log\| N_\epsilon( Orb_\nu(\pi,\cdot,F,K,\delta,\rho), \rho_2)\|_{p, \P_j}$$
 is monotone increasing in $\delta$ and monotone decreasing in $\epsilon,F,K$ (subsets are ordered by inclusion). Therefore, the infimums and the supremum can be replaced by the appropriate (directed) limits. In the sequel, we will use these facts without explicit reference. Similar statements hold true if $\rho_2$ is replaced with $\rho_\infty$ or $N_\epsilon$ is replaced with $N'_\epsilon$.
\end{remark}

%the infimum over $\delta$ could be changed to a limit as $\delta \to 0$, the infimum over $F \subset_f \lb\sH\rb_{top}$ is monotone in the 

%Note that the choice of $\beta$ is suppressed from the notation on the right-hand side. If we wish to emphasize the choice of $\beta$, we will write $h^\beta_{\P,\nu}(\pi,\rho,2)$ etc. 

\begin{lem}\label{lem:span}
If we replace $N_\epsilon(\cdot)$ in the definitions above with $N'_\epsilon(\cdot)$ then we obtain equivalent definitions. More precisely,
$$h_{\P}(\pi,\rho,2) = \sup_{\epsilon>0} \inf_{\delta>0} \inf_{F \subset_f \lb\sH\rb_{top} }\inf_{K \subset_f C(\sH^0)} \lim_{j \to \beta} \frac{1}{d_j} \log\| N'_\epsilon( Orb_\nu(\pi,\cdot,F,K,\delta,\rho), \rho_2)\|_{p, \P_j};$$
$$h_{\P}(\pi,\rho,\infty) = \sup_{\epsilon>0} \inf_{\delta>0} \inf_{F \subset_f \lb\sH\rb_{top} }\inf_{K \subset_f C(\sH^0)} \lim_{j \to \beta} \frac{1}{d_j} \log\| N'_\epsilon( Orb_\nu(\pi,\cdot,F,K,\delta,\rho), \rho_\infty)\|_{p, \P_j}.$$
\end{lem}
\begin{proof}
This is immediate from Lemma \ref{lem:span-sep}.
\end{proof}

\begin{lem}\label{lem:2=infinity}
In general, $h_{\P}(\pi,\rho,2)= h_{\P}(\pi,\rho,\infty)$. 
\end{lem}

\begin{proof}
First note that $\rho_2 \le \rho_\infty$. Therefore, any $(\rho_2,\epsilon)$-separated subset is $(\rho_\infty,\epsilon)$-separated which implies
$$N_\epsilon( Orb_\nu(\pi,\sigma,F,K,\delta,\rho), \rho_2) \le N_\epsilon( Orb_\nu(\pi,\sigma,F,K,\delta,\rho), \rho_\infty)$$
for any $\sigma,F,K,\delta$. Thus $h_{\P}(\pi,\rho,2)\le h_{\P}(\pi,\rho,\infty)$. 

To prove the other direction, let $\epsilon, \kappa>0$ be such that $\epsilon/\kappa<1/10$ and let $M$ be a $(\rho,\kappa)$-spanning subset of $\sG^0$ of minimum cardinality. Let  $\delta>0, F \subset_f \lb\sH\rb$, $K \subset_f C(\sH^0)$, $\sigma:\lb\sH\rb \to \lb d \rb$ and $Y \subset (\sG^0)^d$ be a $(\rho_2,\epsilon)$-spanning set for $Orb_\nu(\pi,\sigma,F,K,\delta,\rho)$ of minimum cardinality. For ease of notation, let $\eta = \lceil \frac{\epsilon^2 d}{\kappa^2} \rceil$. Define $Y' \subset (\sG^0)^d$ as follows. For $y \in Y$, every set $\Lambda \subset [d]$ of cardinality $\eta$ and every map $\phi:\Lambda \to M$ define $y^\phi \in (\sG^0)^d$ by
\begin{displaymath}
y^\phi_i = \left\{ \begin{array}{cc}
y_i & \textrm{ if } i \notin \Lambda\\
\phi(i) & \textrm{ if } i \in \Lambda.
\end{array}\right.
\end{displaymath}
Let $Y'$ be the set of all $y^\phi$ over all such $y \in Y$ and $\phi:\Lambda \to M$. Observe that
$$|Y'| \le |Y| {d \choose \eta} |M|^\eta.$$
We claim that $Y'$ is $(\rho_\infty, \kappa)$-spanning for $Orb_\nu(\pi,\sigma,F,K,\delta,\rho)$. To see this, let $z \in Orb_\nu(\pi,\sigma,F,K,\delta,\rho)$. Because $Y$ is $(\rho_2,\epsilon)$-spanning for $Orb_\nu(\pi,\sigma,F,K,\delta,\rho)$, there is a $y \in Y$ such that $\rho_2(y,z) < \epsilon$. I.e., 
$$\epsilon^2 > \frac{1}{d} \sum_{i=1}^d \rho(y_i,z_i)^2.$$
If we let $\Lambda' = \{i:~ \rho(y_i,z_i) \ge \kappa\}$ then $\epsilon^2 > \frac{1}{d} \sum_{i\in \Lambda'} \kappa^2$
implies $|\Lambda'| \le \eta$. Therefore, there exists a set $\Lambda \subset [d]$ such that for $i \notin \Lambda$, $\rho(y_i,z_i) <\kappa$ and $|\Lambda|=\eta$. By definition of $M$, for every $i \in \Lambda$ there is a point $\phi(i) \in M$ such that $\rho(\phi(i),z_i) < \kappa$. Therefore, $\rho_\infty(y^\phi,z) < \kappa$. This proves the claim: $Y'$ is $(\rho_\infty, \kappa)$-spanning for $Orb_\nu(\pi,\sigma,F,K,\delta,\rho)$.

It follows that
\begin{eqnarray*}
N'_\kappa( Orb_\nu(\pi,\sigma,F,K,\delta,\rho), \rho_\infty) &\le& |Y'| \le  |Y| {d \choose \eta} |M|^\eta\\
&=& N'_\epsilon( Orb_\nu(\pi,\sigma,F,K,\delta,\rho), \rho_2)  {d \choose \lceil \epsilon^2 \kappa^{-2} d \rceil } |M|^{\lceil \epsilon^2\kappa^{-2} d \rceil}.
\end{eqnarray*}
It follows from Stirling's approximation that
\begin{eqnarray*}
&&\lim_{j\to\beta} \frac{1}{d_j} \log \| N'_\kappa( Orb_\nu(\pi,\cdot,F,K,\delta,\rho), \rho_\infty)\|_{p,\P_j}\\
&\le& \lim_{j\to\beta} \frac{1}{d_j} \log \|  N'_\epsilon( Orb_\nu(\pi,\cdot,F,K,\delta,\rho), \rho_2) \|_{p,\P_j}\\
&&+ 2\epsilon^2\kappa^{-2} \log(|M|) - 2\epsilon^2\kappa^{-2}\log(2\epsilon^2\kappa^{-2}) -(1-2\epsilon^2\kappa^{-2})\log(1-2\epsilon^2\kappa^{-2}).
\end{eqnarray*}
Next we take the infimum over $F,K,\delta$ and then the limit supremum as $\epsilon \searrow 0$ to obtain
\begin{eqnarray*}
&& \inf_{\delta>0} \inf_{F \subset_f \lb\sH\rb_{top} }\inf_{K \subset_f C(\sH^0)} \lim_{j\to\beta} \frac{1}{d_j} \log \| N'_\kappa( Orb_\nu(\pi,\cdot,F,K,\delta,\rho), \rho_\infty)\|_{p,\P_j}\\
&\le& \limsup_{\epsilon \searrow 0} \inf_{\delta>0} \inf_{F \subset_f \lb\sH\rb_{top} }\inf_{K \subset_f C(\sH^0)} \lim_{j\to\beta} \frac{1}{d_j} \log \|  N'_\epsilon( Orb_\nu(\pi,\cdot,F,K,\delta,\rho), \rho_2) \|_{p,\P_j}.
\end{eqnarray*}
Because  $|N'_\epsilon( Orb_\nu(\pi,\sigma,F,K,\delta,\rho), \rho_2)|$ is monotone decreasing in $\epsilon$ we can replace $\limsup_{\epsilon \searrow 0}$ above with $\sup_{\epsilon>0}$. It now follows from Lemma \ref{lem:span} that $h_{\P}(\pi,\rho,2)\ge h_{\P}(\pi,\rho,\infty)$. 
\end{proof}

Because of the lemma above, we will write $h_{\P}(\pi,\rho)$ to denote either $h_{\P}(\pi,\rho,2)$ or $h_{\P}(\pi,\rho,\infty)$. If we need to specify $\beta$ and $p$ then we denote this quantity by $h^\beta_{\P,p}(\pi,\rho)$.

%Recall that $\lb \sH \rb_{top}$ denotes the set of all sets $B \subset \sH$ such that $\so:B \to \sH^0$ and $\ra:B \to \sH^0$ are homeomorphisms. 

\begin{defn}
A pseudo-metric $\rho$ on $\sG^0$ is {\em dynamically generating} for $\pi:\sG \to \sH$ if for every distinct $x,y \in \sG^0$ there exists $f \in [\sH]_{top}$ such that $\rho(f\cdot x,f\cdot y)>0$. Note that we are using $[\sH]_{top}$ instead of $\lb \sH\rb_{top}$ to define this property. %For example, if $\pi(x)\notin \so(f)$ then, by convention, $fx=*$. If also $\pi(y) \in \so(f)$ then $fy \in \sG^0$ so  $\rho(fx,fy)>0$ in this case.
\end{defn}

The main result of this section is:
\begin{thm}\label{thm:K-top}
If $\rho_1,\rho_2$ are dynamically generating continuous pseudo-metrics on $\sG^0$ and $\sG$ is \`etale then $h_\P(\pi,\rho_1) = h_\P(\pi,\rho_2)$.
\end{thm}

\begin{remark}
The proof of Theorem \ref{thm:K-top} uses only properties (2) and (3) of the definition of sofic approximation (Definition \ref{defn:sofic-approximation}). Moreover, it does not use the full definition of $\lb \sH \rb_{top}$. We only need to use the fact that $\lb \sH \rb_{top}$ is closed under composition (by Lemma \ref{lem:composition}) and $[\sH]_{top} \subset \lb \sH \rb_{top}$. However, the proof of the variational principle (Theorem \ref{thm:var-principle}) makes use of the full definition of $\lb \sH \rb_{top}$.
%We could replace $\lb \sH \rb_{top}$ in the definition of $h_\P(\pi,\rho)$ with any subset $\cF \subset [\sH]$ 
%All that is needed is that $\lb \sH \rb_{top}$ is an inverse semigroup. However, the properties of $\lb \sH \rb_{top}$ are used later in the second definition of measure entropy and therefore, in the proof of the variational principle.
\end{remark}

\begin{defn}
Given Theorem \ref{thm:K-top}, we define the {\em sofic topological entropy of $\pi$ with respect to $(\P,p,\beta)$} by $h_\P(\pi):=h_\P(\pi,\rho)=h_{\P,p}^\beta(\pi,\rho)$ where $\rho$ is any dynamically generating continuous pseudo-metric on $\sG^0$. Intuitively, this is the relative entropy with respect to the measure $\nu$. Because the sofic approximation $\P$ determines $\nu$, $\nu$ is implicitly referenced in the notation. Indeed, for any Borel subset $P \subset \sH^0$,
$$\nu(P)=\tr_{\sH}(P) = \lim_{j\to\beta} \int \tr_d(\sigma(P))~d\P_j(\sigma).$$
This is implied by the asymptotic trace-preserving property of $\P$. %because for every $\delta>0$, $\sigma$ is $(\{P\},\delta)$-trace-preserving for asymptotically all $\sigma$ (with respect to $\P$).
\end{defn}

\begin{lem}\label{lem:K-top1}
If $\rho, \rho'$ are continuous metrics on $\sG^0$ then $h_\P(\pi, \rho) = h_\P(\pi, \rho')$.
\end{lem}

\begin{proof}
Because $\rho$ and $\rho'$ are continuous metrics and $\sG^0$ is compact, for every $\delta>0$ and sufficiently large integer $n\gg 0$ there exist  $\delta_0, \epsilon_n>0$ with $\delta_0<1$ such that 
\begin{enumerate}
\item $\rho'(x,y) \le \sqrt{\delta_0} \Rightarrow \rho(x,y) < \delta$,
\item $\delta_0 \le \delta^2$,
\item $\rho'(x,y)> \epsilon_n \Rightarrow \rho(x,y) > 1/n$,
\item $\lim_{n\to\infty} \epsilon_n = 0$.
\end{enumerate}
Let $M= \max \{\rho(x,y):~ x,y \in \sG^0\}$ be the diameter of $\rho$. 

\noindent {\bf Claim 1}. For any $\sigma:\lb \sH \rb \to \lb d \rb$, $F \subset_f \lb \sH \rb_{top}$ and $K \subset_f C(\sH^0)$,
$$Orb_\nu(\pi,\sigma,F,K,\delta_0,\rho')  \subset Orb_\nu(\pi,\sigma,F,K, \delta(M^2 +1)^{1/2}, \rho).$$ 

\begin{proof}[Proof of Claim 1]
Let $z \in Orb_\nu(\pi,\sigma,F,K,\delta_0,\rho')$ and $f\in F$. By definition,
$$\rho'_2( f\cdot z, z \circ \sigma(f) ) = \left(\frac{1}{d} \sum_{i=1}^d \rho'(f\cdot z_i, z_{\sigma(f) i})^2\right)^{1/2}< \delta_0.$$ 
So there exists a set $\Lambda=\Lambda(z,f,\sigma) \subset \{1,\ldots, d\}$ such that
\begin{enumerate}
\item $|\Lambda| \le \delta_0 d$;
\item for every $i \notin \Lambda$, $\rho'(f\cdot z_i, z_{\sigma(f) i}) < \sqrt{\delta_0}$, which implies $\rho(f\cdot z_i, z_{\sigma(f) i})  < \delta$. 
\end{enumerate}
Therefore,
\begin{eqnarray*}
\rho_2(f\cdot z , z \circ \sigma_f)^2 &=&  \frac{1}{d}\sum_{i=1}^d \rho( f\cdot z_i, z_{\sigma(f)i})^2  \le \frac{|\Lambda|}{d}M^2 + \frac{1}{d}\sum_{i \notin \Lambda}  \rho(f\cdot z_i, z_{\sigma(f)i})^2  \\
&<& \delta_0 M^2 + \frac{d-|\Lambda|}{d}\delta^2 \le \delta^2(M^2+1).
\end{eqnarray*}
Because $f\in F$ is arbitrary, $z \in Orb_\nu(\pi,\sigma,F, K,\delta(M^2 +1)^{1/2}, \rho)$. Because $z$ is arbitrary, this implies the claim.
\end{proof}
By choice of $\epsilon_n$,
$$N_{\epsilon_n}(Orb_\nu(\pi,\sigma,F,K,\delta_0,\rho'), \rho'_\infty) \le N_{1/n}(Orb_\nu(\pi,\sigma,F,K,\delta(M^2+1)^{1/2},\rho), \rho_\infty).$$
Thus we obtain
\begin{eqnarray*}
&&\lim_{j\to\beta} \frac{1}{d_j} \log \| N_{\epsilon_n}(Orb_\nu(\pi,\cdot,F,K,\delta_0,\rho'), \rho'_\infty)\|_{p,\P_j}\\
 &\le& \lim_{j\to\beta}\frac{1}{d_j} \log  \| N_{1/n}(Orb_\nu(\pi,\cdot,F,K,\delta(M^2+1)^{1/2},\rho), \rho_\infty)\|_{p,\P_j}.
 \end{eqnarray*}

%$$h^{\epsilon_n}_\P(\pi, F,K,\delta_0,\rho'_\infty) \le h^{1/n}_\P(\pi, F, \delta(M^2+1)^{1/2}, \rho_\infty).$$
Taking the infimum over $\delta_0>0$, then over all $\delta>0$ then over all $F \subset_f \lb \sH \rb_{top}$ and $K \subset_f C(\sH^0)$, then the supremum over all $n$  (and using that $\epsilon_n \to 0$ as $n\to \infty$) we obtain $h_\P(\pi,\rho') \le h_\P(\pi,\rho).$ Because $\rho'$ and $\rho$ are arbitrary, this implies the lemma.
\end{proof}

\begin{defn}
Given a continuous pseudo-metric $\rho$ on $\sG^0$ and a sequence $\{\phi_i\}_{i=1}^\infty$ with $\phi_i\in \lb \sH \rb_{top}$, define a pseudo-metric $\rho^\phi$ on $\sG^0$ by
$$\rho^\phi(x,y) = \left(\sum_{i=1}^\infty 2^{-i} \rho(\phi_i \cdot x, \phi_i \cdot y)^2\right)^{1/2}.$$
%where $\rho(\phi_i x, \phi_i y)=0$ if either $x \notin \so(\phi_i)$ or $y \notin \so(\phi_i)$.
\end{defn}

%\begin{defn}\label{defn:separable}
%Let $m$ be a continuous metric on $\sG^0$ and define a metric $\overline{m}$ on $[\sH]_{top}$ by $\overline{m}(f,g) = \sup_{x\in \sH^0} m(f\cdot x, g\cdot x)$. The topology that this induces on $[\sH]_{top}$ is independent of the choice of metric $m$. 
%\end{defn}

\begin{lem}\label{lem:K-top2}
Assume $\sG$ is \'etale. Consider the homeomorphism group of $\sG^0$, $Homeo(\sG^0)$, with the topology of pointwise convergence. Consider the homomorphism $\varpi:[\sH]_{top} \to Homeo(\sG^0)$ given by $\varpi(\phi)(x) = \pi^{-1}(\phi) \cdot x$.  If $\rho$ is a continuous dynamically generating pseudo-metric and $\{\phi_i\}_{i=1}^\infty \subset [\sH]_{top}$ is such that $\varpi(\{\phi_i\}_{i=1}^\infty)$ is dense in $\varpi([\sH]_{top})$  then $\rho^\phi$ is a continuous metric.
\end{lem}

\begin{remark}
The fact that $\varpi$ is a homomorphism uses Lemma \ref{lem:top-inclusion}. In particular, this lemma uses the hypothesis that $\sG$ is \`etale.\end{remark}

\begin{proof}
It is clear that $\rho^\phi$ is a continuous pseudo-metric. So it suffices to show that for any $x,y \in \sG^0$ with $x\ne y$, $\rho^\phi(x,y)>0$. Because $\rho$ is dynamically generating, there is an $f \in [\sH]_{top}$ such that $\rho(f\cdot x,f\cdot y)>0$. Because $\varpi(\{\phi_i\}_{i=1}^\infty)$ is dense in $\varpi([\sH]_{top})$ and $\rho$ is continuous, there exists an $i$ such that 
$$\max(\rho(\phi_i\cdot x, f\cdot x), \rho(\phi_i\cdot y,f\cdot y)) <\rho(f\cdot x,f\cdot y)/3.$$
Therefore,
$$\rho(\phi_i\cdot x,\phi_i\cdot y) \ge \rho(f\cdot x,f\cdot y) - \rho(\phi_i\cdot x, f\cdot x) - \rho(\phi_i\cdot y,f\cdot y) \ge \rho(f\cdot x,f\cdot y)/3 >0.$$
Thus $\rho^\phi(x,y) \ge 2^{-i/2}\rho(f\cdot x,f\cdot y)/3>0$. Because $x,y$ are arbitrary, this establishes that $\rho^\phi$ is a metric as claimed.
\end{proof}

\begin{remark}
The reason our definition of generating pseudo-metric uses $[\sH]_{top}$ instead of $\lb \sH \rb_{top}$ (or other possible choices) is that, if $\{\phi_i\}_{i=1}^\infty \subset \lb \sH \rb_{top}$ then $\rho^\phi$ is not necessarily continuous, but if $\{\phi_i\}_{i=1}^\infty \subset [\sH]_{top}$ then $\rho^\phi$ is continuous.
\end{remark}

%Note: because $\sH$ is a locally compact metrizable groupoid, it is separable and therefore, $\lb \sH \rb_{top}$ is also separable (in the Hausdorff topology). 

\begin{lem}\label{lem:K-top3}
Let $\rho$ be a continuous dynamically generating pseudo-metric and let $\{\phi_i\}_{i=1}^\infty$ be a subset of $[\sH]_{top}$ with $\phi_1=\sH^0$. Then
$$h_\P(\pi,\rho) = h_\P(\pi, \rho^\phi).$$
\end{lem}

\begin{proof}
{\bf Claim 1}. For any $x,y \in (\sG^0)^d$,
$$\rho_2^\phi(x,y)^2 = \sum_{j=1}^\infty 2^{-j} \rho_2(\phi_j \cdot x, \phi_j \cdot y)^2.$$
\begin{proof}[Proof of Claim 1]
This is a straightforward computation:
\begin{eqnarray*}
\rho_2^\phi(x,y)^2 &=& \frac{1}{d}\sum_{i=1}^d \rho^\phi(x_i,y_i)^2 = \frac{1}{d} \sum_{i=1}^d \sum_{j=1}^\infty 2^{-j} \rho(\phi_j \cdot x_i, \phi_j \cdot y_i)^2\\
&=&  \sum_{j=1}^\infty 2^{-j}  \frac{1}{d} \sum_{i=1}^d \rho(\phi_j \cdot x_i, \phi_j \cdot y_i)^2 = \sum_{j=1}^\infty 2^{-j} \rho_2(\phi_j \cdot x, \phi_j \cdot y)^2.
\end{eqnarray*}
\end{proof}

Because $\phi_1=\sH^0$, $\rho \le 2 \rho^\phi$. So
$$Orb_\nu(\pi,\sigma,F,K,\delta,\rho^\phi) \subset Orb_\nu(\pi,\sigma,F,K,2\delta,\rho)$$
for any $\sigma, F,K,\delta$. 

We are going to use spanning sets, but there is one technical issue. A spanning set for a given set $Y$ is not required to be contained in $Y$. To remedy this we show:

{\bf Claim 2}. There exists a $(\rho_2,2\epsilon)$-spanning set $Y$ for $Orb_\nu(\pi,\sigma,F,K,2\delta,\rho)$ which is contained in $Orb_\nu(\pi,\sigma,F,K,2\delta,\rho)$ and satisfies $|Y| \le N'_{\epsilon}(Orb_\nu(\pi,\sigma,F,K,2\delta,\rho))$.

\begin{proof}[Proof of Claim 2]
Let $Y'$ be a minimal $(\rho_2,\epsilon)$-spanning set for $Orb_\nu(\pi,\sigma,F,K,2\delta,\rho)$. Because $Y'$ is minimal, for each $y' \in Y'$ there exists an element $y \in Orb_\nu(\pi,\sigma,F,K,2\delta,\rho)$ with $\rho_2(y,y')<\epsilon$. The collection $Y$ of all of these elements satisfies the claim.
\end{proof}

%{\bf Claim 3}. For any $\epsilon>0$ there exists $F \in \cF(\lb \sH \rb)$ and $\eta>0$ such that
%$$N'_{\epsilon}(Orb_\nu(\pi,\sigma,F,2\delta,\rho),\rho_2) \ge N'_\eta((Orb_\nu(\pi,\sigma,F,\delta,\rho^\phi), \rho^\phi_2).$$

%\begin{proof}[Proof of Claim 3]
Let $M$ be the diameter of $(\sG^0,\rho)$. Let $F_n$ be any finite subset of $\lb \sH \rb_{top}$ containing $\{\phi_1,\ldots, \phi_n\}$. If $1\le i \le n$ and $x,y \in Orb_\nu(\pi,\sigma,F_n,K,2\delta,\rho)$ satisfy $\rho_2(x,y)<2\epsilon$ then
\begin{eqnarray*}
\rho_2(\phi_i \cdot x, \phi_i \cdot y) &\le & \rho_2(x \circ \sigma(\phi_i), y \circ \sigma(\phi_i)) + \rho_2(x \circ \sigma(\phi_i), \phi_i \cdot x) + \rho_2(y \circ \sigma(\phi_i), \phi_i \cdot y)<2\epsilon + 4\delta.
\end{eqnarray*}
We have used there that $\rho_2(x \circ \sigma(\phi_i), y \circ \sigma(\phi_i)) \le \rho_2(x,y)$. So,
\begin{eqnarray*}
\rho_2^\phi(x,y)^2 &=& \sum_{j=1}^\infty 2^{-j}  \rho_2(\phi_j x, \phi_j y)^2 < 2^{-n}M^2 + (2\epsilon + 4\delta)^2.
\end{eqnarray*}
By Claim 2, there exists a $(\rho_2,2\epsilon)$-spanning set $Y$ for $Orb_\nu(\pi,\sigma,F_n,K,2\delta,\rho)$ which is contained in $Orb_\nu(\pi,\sigma,F_n,K,2\delta,\rho)$ and satisfies $|Y| \le N'_{\epsilon}(Orb_\nu(\pi,\sigma,F_n,K,2\delta,\rho))$. So for any $x \in Orb_\nu(\pi,\sigma,F_n,K,2\delta,\rho)$ there exists $y \in Y$ with $\rho_2(x,y)<2\epsilon$ which implies $\rho_2^\phi(x,y)^2 < 2^{-n}M^2 + (2\epsilon + 4\delta)^2$. Thus $Y$ is $(\rho_2^\phi, \sqrt{2^{-n}M^2 + (2\epsilon + 4\delta)^2})$-spanning. Letting $\eta =\sqrt{2^{-n}M^2 + (2\epsilon + 4\delta)^2}$, we have
\begin{eqnarray*}
&&N'_{\epsilon}(Orb_\nu(\pi,\sigma,F_n,K,2\delta,\rho),\rho_2) \ge N'_\eta((Orb_\nu(\pi,\sigma,F_n,K,2\delta,\rho), \rho^\phi_2)\\
 &\ge& N'_\eta((Orb_\nu(\pi,\sigma,F_n,K,\delta,\rho^\phi), \rho^\phi_2)
 \end{eqnarray*}
where the last inequality follows from the inclusion $Orb_\nu(\pi,\sigma,F_n,K,\delta,\rho^\phi) \subset Orb_\nu(\pi,\sigma,F_n,K,2\delta,\rho)$. By monotonicity, if $n$ is large enough and $\delta$ is small enough then $3\epsilon>\eta$ which implies
$$N'_{\epsilon}(Orb_\nu(\pi,\sigma,F_n,K,2\delta,\rho),\rho_2) \ge N'_{3\epsilon}((Orb_\nu(\pi,\sigma,F_n,K,\delta,\rho^\phi), \rho^\phi_2).$$

%\end{proof}
Because $F_n$ is any finite subset of $\lb \sH \rb_{top}$ containing $\{\phi_1,\ldots, \phi_n\}$ we have 
\begin{eqnarray*}
&&\inf_{\delta>0} \inf_{F  \subset_f \lb \sH \rb_{top}} \inf_{K \subset_f C(\sH^0)} \lim_{j\to\beta} \frac{1}{d_j} \log \|  N'_\epsilon( Orb_\nu(\pi,\cdot, F, K,2\delta,\rho), \rho_2) \|_{p,\P_j}\\
 &\ge&\inf_{\delta>0} \inf_{F  \subset_f \lb \sH \rb_{top}} \inf_{K \subset_f C(\sH^0)} \lim_{j\to\beta} \frac{1}{d_j} \log \| N'_{3\epsilon}( Orb_\nu(\pi,\cdot, F,K, \delta,\rho^\phi), \rho^\phi_2) \|_{p,\P_j}.
\end{eqnarray*}
Since this is true for every $\epsilon>0$, we have $h_\P(\pi, \rho, 2) \ge h_\P(\pi, \rho^\phi, 2)$ which implies $h_\P(\pi, \rho) \ge h_\P(\pi, \rho^\phi).$

{\bf Claim 3}. Given any finite $F \subset \lb \sH \rb_{top}$ with $\sH^0 \in F$ and $\delta>0$, if $n$ is sufficiently large, $F'=\{\phi_j  f :~ f\in F , 1\le j \le n\}$ and $\sigma$ is $(F',\delta^2/M^2)$-multiplicative then
$$Orb_\nu(\pi,\sigma,F,K,2\delta,\rho^\phi) \supset Orb_\nu(\pi,\sigma,F',K,\delta,\rho).$$

\begin{proof}[Proof of Claim 3]
Let $n$ be large enough so that $(3\delta)^2+ 2^{-n}M^2 \le (4\delta)^2$. Then for any $f \in F$ and $j$ with $1\le j \le n$, if $\sigma$ is $(F',\delta^2/M^2)$-multiplicative and $x \in Orb_\nu(\pi,\sigma,F',K,\delta,\rho)$ then 
\begin{eqnarray*}
\rho_2(\phi_j  f\cdot x, \phi_j (x \circ \sigma(f))  ) &\le&  \rho_2(\phi_j f\cdot x,x \circ \sigma(\phi_j f)) + \rho_2(x \circ \sigma(\phi_j  f),x \circ \sigma(\phi_j) \sigma( f))\\
&& + \rho_2(x \circ \sigma(\phi_j) \sigma( f),  (\phi_j \cdot x) \circ \sigma(f) ) \le 3\delta.
\end{eqnarray*}
This calculation relies on two easily verified facts: $\phi_j (x \circ \sigma(f)) =  (\phi_j x) \circ \sigma(f)$ and $\rho_2(x \circ \sigma(\phi_j) \sigma( f),  (\phi_j x) \circ \sigma(f) ) \le \rho_2(x \circ \sigma(\phi_j),  (\phi_j x) )$.
 Thus 
\begin{eqnarray*}
\rho_2^\phi( f\cdot x, x \circ \sigma(f))^2 &=&  \sum_{j=1}^\infty 2^{-j} \rho_2(\phi_j f \cdot x, \phi_j (x \circ \sigma(f))  )^2\le (3\delta)^2+ 2^{-n}M^2 < 4\delta^2.
\end{eqnarray*}
This implies the claim.
\end{proof}
Recall that $\rho \le 2\rho^\phi$. So when Claim 3 holds,
$$N'_\epsilon( Orb_\nu(\pi,\sigma,F,K,2\delta,\rho^\phi), \rho^\phi_\infty) \ge N'_{2\epsilon}( Orb_\nu(\pi,\sigma,F',K,\delta,\rho), \rho_\infty).$$
Thus
%\begin{eqnarray*}
%&&
$$\lim_{j\to\beta} \frac{1}{d_j} \log \| N'_\epsilon( Orb_\nu(\pi,\cdot, F, K,2\delta,\rho^\phi), \rho^\phi_\infty) \|_{p,\P_j}\ge\lim_{j\to\beta} \frac{1}{d_j} \log \|  N'_{2\epsilon}( Orb_\nu(\pi,\cdot, F', K,\delta,\rho), \rho_\infty) \|_{p,\P_j}$$
%\end{eqnarray*}
which implies $h_\P(\pi,\rho^\phi) \ge h_\P(\pi,\rho).$ Because we obtained the reverse inequality above, this proves the lemma.
\end{proof}

\begin{proof}[Proof of Theorem \ref{thm:K-top}]
Because $\sG^0$ is compact and metrizable, it is second countable; i.e., there is a countable base $\mathcal{U}$ of open sets of $\sG^0$. Let $\Phi \subset [\sH]_{top}$ be any countable set such that for any $U_1,\ldots, U_n, V_1,\ldots, V_n \in \mathcal{U}$, if there exists $\psi \in [\sH]_{top}$ such that $\varpi(\psi)(U_i) \subset V_i$ (for all $i$) then there exists $\phi \in \Phi$ such that $\varpi(\phi) (U_i) \subset V_i$ (for all $i$). We claim that $\varpi(\Phi)$ is dense in $\varpi([\sH]_{top})$. It suffices to show that for every $\psi \in [\sH]_{top}$, $x_1,\ldots, x_n \in \sG^0$ and open sets $V'_1,\ldots, V'_n$ with $\varpi(\psi)(x_i) \in V'_i$, there exists $\phi \in \Phi$ such that $\varpi(\phi)(x_i) \in V'_i$. Because $\mathcal{U}$ is a basis and $\varpi(\psi)$ is continuous, there are sets $U_1,\ldots, U_n, V_1,\ldots, V_n \in \mathcal{U}$ such that $x_i \in U_i$ and $\varpi(\psi)(U_i) \subset V_i \subset V'_i$ for all $i$. By definition, $\Phi$ contains an element $\phi$ such that  $\varpi(\phi)(U_i) \subset V_i $ for all $i$ and therefore $\varpi(\phi)(x_i) \in V'_i$ proving the claim.

So there exists a sequence $\phi=\{\phi_i\}_{i=1}^\infty \subset [\sH]_{top}$ with $\phi_1=\sH^0$ such that $\varpi(\{\phi_i\}_{i=1}^\infty)$ is dense in $\varpi([\sH]_{top})$. By Lemmas \ref{lem:K-top3}, \ref{lem:K-top2}, \ref{lem:K-top1}, $h_\P(\pi,\rho_1) = h_\P(\pi,\rho_1^\phi) = h_\P(\pi,\rho_2^\phi) = h_\P(\pi,\rho_2)$.
\end{proof}

\section{Measure entropy via partitions}\label{sec:meas1}

In this section, we define measure sofic entropy for groupoid extensions in a manner analogous to \cite{Ke12}. Let $\pi:(\sG,\mu) \to (\sH,\nu)$ be a pmp class-bijective extension of discrete pmp groupoids and $\P:=\{\P_j\}_{j\in J}$ be a sofic approximation  to $(\sH,\nu)$.

% Let $(\sG,\mu), (\sH,\nu)$ be discrete pmp groupoids and $\pi:\sG \to \sH$ a pmp groupoid morphism. We assume there are conull subsets $\sG' \subset \sG, \sH' \subset \sH$ such that $\pi$ restricted to $\sG'$ is a class-bijective extension of $\sH'$. Also let $\P:=\{\P_j\}_{j\in J}$ be a sofic approximation  to $(\sH,\nu)$. %We require that $\mu$ is $\lb \sH \rb$-invariant.

Given a finite partition $\cP$ of $\sG^0$ and a finite set  $F \subset \lb \sH \rb$, let $\cP^F$ be the coarsest partition of $\sG^0$ containing $\{f\cdot P:~ f \in F, P \in \cP\}$. Also let $\Sigma(\cP)$ be the smallest sigma-algebra of $\sG^0$ containing $\cP$. Let $\cB(\Delta^0_d)$ be the set of all subsets of $\Delta^0_d$. Of course, $\cB(\Delta^0_d)$ is a sigma-algebra.

A map $\phi:\Sigma(\cP) \to \cB(\Delta^0_d)$ is a {\em homomorphism} if for every $A,B \in \Sigma(\cP)$, $\phi(A\cup B) = \phi(A) \cup \phi(B)$, $\phi(A\cap B) = \phi(A) \cap \phi(B)$, $\phi(\emptyset)=\emptyset$ and $\phi(\sG^0) = \Delta^0_d$. 

\begin{defn}[Good homomorphisms]
Given $\sigma:\lb \sH \rb \to \lb d \rb$ and $f \in \lb \sH \rb$, we let $\sigma_f$ denote $\sigma(f)$. Given $\delta>0$ and $F \subset \lb\sH\rb$ with $\sH^0 \in F$, let $\Hom(\pi,\sigma, \cP, F, \delta)$ be the set of all homomorphisms $\phi:\Sigma(\cP^F) \to \cB(\Delta^0_d)$ such that
\begin{enumerate}
\item $ \sum_{P \in \cP} |\sigma_f \cdot \phi(P) \vartriangle \phi( f\cdot P) |d^{-1} < \delta \quad \forall f\in F$;
\item $ \sum_{P\in \cP^F} | |\phi(P)|d^{-1} - \mu(P) | < \delta$.
\end{enumerate}
\end{defn}

\begin{defn} Given a partition $\cQ$ of $\sG^0$ with $\cQ \le \cP$, let $|\Hom(\pi,\sigma, \cP, F, \delta)|_\cQ$ be the cardinality of the set of homomorphisms $\phi:\Sigma(\cQ) \to \cB(\Delta^0_d)$ such that there exists a $\phi' \in \Hom(\pi,\sigma,\cP,F,\delta)$ so that $\phi$ is the restriction of $\phi'$ to $\Sigma(\cQ)$.
\end{defn}

\begin{defn} For the definitions below, recall the definitions of $\lim_{j\to\beta}, \| \cdot \|_{p,\P_j}$ and $X \subset_f Y$ from the beginning of \S \ref{sec:top}. In particular, choose a bias $\beta$ and $p \in [1,\infty]$. Let $\cB(\sG^0)$ denote the Borel sigma-algebra of $\sG^0$. Given finite Borel partitions $\cQ\le \cP$ and a sub-algebra $\cF \subset \cB(\sG^0)$ define
\begin{eqnarray*}
h_{\P,\mu}(\pi,  \cQ, \cP,  F, \delta)&:=&  \lim_{j\to\beta} \frac{1}{d_j} \log \|  |\Hom(\pi, \cdot, \cP, F,\delta) |_\cQ \|_{p,\P_j}\\
h_{\P,\mu}(\pi,  \cQ, \cP)&:=&\inf_{F \subset_f \lb \sH \rb} \inf_{\delta>0} h_{\P,\mu}(\pi,  \cQ, \cP, F, \delta)\\
h_{\P,\mu}(\pi, \cQ, \cF)&:=& \inf_{\cQ \le \cP \subset \cF}  h_{\P,\mu}(\pi,  \cQ, \cP)\\
h_{\P,\mu}(\pi, \cF)&:=& \sup_{\cQ \subset \cF} h_{\P,\mu}(\pi,  \cQ,\cF).
\end{eqnarray*}
In the second line we require that $\sH^0 \in F$. This condition is necessary in order that $\Hom(\pi,\sigma,\cP,F,\delta)$ be well-defined. We will, as a rule, leave this condition implicit in the notation. The infimum in the second-to-last line is over all finite Borel partitions $\cP$ with $\cQ \le \cP \subset \cF$ and the supremum in the last line is over all finite partitions $\cQ \subset \cF$. The {\em sofic measure entropy} of $\pi$ (with respect to $\P,p,\beta$) is $h_{\P,\mu}(\pi):=h_{\P,\mu}(\pi, \cB(\sG^0))$. 
\end{defn}

\begin{remark}
Of course, $h_{\P,\mu}(\pi)$ depends implicity on $1\le p \le \infty$ and a bias $\beta$. Whenever we want to the emphasize this dependence, we will write $h^\beta_{\P,p,\mu}(\pi)$ instead of $h_{\P,\mu}(\pi)$ and similarly for the other quantities above.
\end{remark}

\begin{remark}
The order of the supremums, infimums and limits above is important with the exception that one can permute the three infimums without affecting the definition of $h_{\P,\mu}(\pi, \cQ, \cF)$.\end{remark}

\begin{remark} There is a certain useful monotonicity phenomenon in the formulas above: the quantity
 $$\frac{1}{d_j} \log \|  |\Hom(\pi, \cdot, \cP, F,\delta) |_\cQ \|_{p,\P_j}$$
 is monotone increasing in $\delta, \cQ$ and monotone decreasing in $F,\cP$ (subsets are partially ordered by inclusion and partitions are partially ordered by refinement). Therefore, the infimums and the supremum can be replaced by the appropriate (directed) limits. In the sequel, we will use these facts without explicit reference.
\end{remark}

\begin{defn} Given a sub-algebra $\cF$ of $\sG^0$, let $\Sigma_\pi(\cF)$ be the smallest sigma-algebra such that for every $P \in \cF$ and $f \in \lb \sH \rb$, $f\cdot P \in \Sigma_\pi(\cF)$. We say that $\cF$ is {\em $\pi$-generating} if $\Sigma_\pi(\cF)$ is the full Borel sigma-algebra $\cB(\sG^0)$ (up to sets of measure zero). 
\end{defn}

The remainder of this section is devoted to proving:

\begin{thm}\label{thm:K}
If $\cF \subset \cB(\sG^0)$ is $\pi$-generating, then $h_{\P,\mu}( \pi) = h_{\P,\mu}( \pi, \cF)$.
\end{thm}

\begin{remark}
The proof of Theorem \ref{thm:K} uses only properties (2) and (3) in the definition of sofic approximation (Definition \ref{defn:sofic-approximation}).
\end{remark}

%\section{Proof attempt for idea \#2}

\begin{lem}\label{lem:4}
Let $\cP$ be a finite partition of $\sG^0$ and $\epsilon>0$. Then there is a $\delta>0$ such that for every sub-algebra $\cS \subset \cB(\sG^0)$ with $\max_{P\in \cP} \inf_{B \in  \cS} \mu(P \vartriangle B) <\delta$ there exists a homomorphism $\phi: \Sigma(\cP) \to \cS$ satisfying $\mu(\phi(P)  \vartriangle P) < \epsilon$ for all $P \in \Sigma(\cP)$.
\end{lem}

\begin{proof}
Let $\cP=\{P_1,\ldots,P_n\}$ and choose $\delta>0$ so that $3 n^3 \delta < \epsilon$. Suppose there is a map $\psi:\cP \to \cS$ such that $\mu(P \vartriangle \psi(P)) < \delta$ for every $P \in \cP$. For $1\le i <n$ define
$$\phi(P_i) = \psi(P_i) \setminus \bigcup_{j<i} \psi(P_j).$$
Set $\phi(P_n) = \sG^0 \setminus \bigcup_{i=1}^{n-1} \psi(P_i)$. Note that $\phi(\cP)$ is a partition of $\sG^0$ so there is a unique way to extend $\phi$ to $\Sigma(\cP)$ so that it is a homomorphism. For any $i \ne j$, 
$$\mu(\psi(P_i) \cap \psi(P_j)) \le \mu(P_i \cap P_j) + 2 \delta = 2\delta.$$
So for $1\le i < n$, 
$$\mu(\phi(P_i) \vartriangle P_i) \le \mu(\psi(P_i) \vartriangle P_i) + 2n \delta \le 3n\delta.$$
Also
$$\mu( \phi(P_n) \vartriangle P_n) \le \sum_{i=1}^{n-1} \mu( \phi(P_i) \vartriangle P_i) \le 3n^2\delta.$$
Because any $P \in \Sigma(\cP)$ is a union of at most $n$ elements of $\cP$, we have 
$$\mu(\phi(P) \vartriangle P) \le 3 n^3 \delta < \epsilon \quad \forall P \in \Sigma(\cP).$$
%for every $P \in \cP$.
 \end{proof}

\begin{defn}
Let $\cQ$ be a finite partition of $\sG^0$. On the set of all homomorphisms from some sub-algebra containing $\cQ$ of $\cB(\sG^0)$ to $\cB(\Delta^0_d)$ we define the pseudo-metric
$$\rho_\cQ(\phi,\psi) = \max_{Q \in \cQ} d^{-1} | \phi(Q) \vartriangle \psi(Q) |.$$
Given a set $K$ of homomorphisms, let $N_\epsilon(K, \rho_\cQ)$ be the maximum cardinality of a $(\rho_\cQ,\epsilon)$-separated subset. For $\epsilon>0$, define
\begin{eqnarray*}
h_{\P,\mu}^\epsilon(\pi, \cQ, \cP,  \delta, F)&:=&  \lim_{j\to\beta} \frac{1}{d_j} \log \| N_\epsilon(\Hom(\pi, \cdot, \cP, F,\delta) , \rho_\cQ) \|_{p,\P_j}\\
h_{\P,\mu}^\epsilon(\pi,  \cQ, \cP,\delta)&:=& \inf_{F \subset_f \lb \sH \rb} h_{\P,\mu}^\epsilon(\pi,  \cQ, \cP, \delta,F)\\
h_{\P,\mu}^\epsilon(\pi,  \cQ, \cP)&:=& \inf_{\delta>0} h_{\P,\mu}^\epsilon(\pi,  \cQ, \cP, \delta)\\
h_{\P,\mu}^\epsilon(\pi,  \cQ, \cF)&:=& \inf_{\cQ \le \cP \subset \cF}  h_{\P,\mu}^\epsilon(\pi,  \cQ, \cP)\\
h_{\P,\mu}^\epsilon(\pi, \cF)&:=& \sup_{\cQ \subset \cF} h_{\P,\mu}^\epsilon(\pi, \cQ,\cF).
\end{eqnarray*}
The infimum in the second-to-last line is over all finite Borel partitions $\cP$ with $\cQ \le \cP \subset \cF$ and the supremum in the last line is over all finite partitions $\cQ \subset \cF$. 
\end{defn}

\begin{lem}\label{lem:5}
Let $\cQ$ be a finite measurable partition of $\sG^0$ and let $\kappa>0$. Then there is an $\epsilon>0$ such that $h_{\P,\mu}(\pi, \cQ, \cP) \le h_{\P,\mu}^\epsilon(\pi,\cQ,\cP)+\kappa$ for all finite measurable partitions $\cP$ refining $\cQ$.
\end{lem}

\begin{proof}
Let $0<\epsilon<1/2$. The key observation is that for any $A \subset \Delta^0_d$ the number of sets $B \subset \Delta^0_d$ with $|A \vartriangle B|_d \le \epsilon$ equals
$$\sum_{i=0}^{\lfloor \epsilon d\rfloor} {d \choose i} \le \left( \lfloor \epsilon d \rfloor +1 \right){d \choose  \lfloor \epsilon d \rfloor }.$$
By Stirling's approximation this is at most $e^{\kappa d/|\cQ|}$ if $\epsilon>0$ is sufficiently small. Therefore, if $\phi:\Sigma(\cQ) \to \cB(\Delta^0_d)$ is any homomorphism, then the set of all homomorphisms $\psi:\Sigma(\cQ) \to \cB(\Delta^0_d)$ such that $\rho_{\cQ}(\phi,\psi) \le \epsilon$ has cardinality at most $\exp( \kappa d)$ (if $\epsilon$ is sufficiently small). In this case,
$$|\Hom(\pi, \sigma, \cP, F,\delta)|_{\cQ} \le N_\epsilon(\Hom(\pi, \sigma, \cP, F,\delta), \rho_{\cQ})\exp( \kappa d)$$
for every $\sigma:\lb\sH\rb\to\lb d \rb$, finite $\cP\ge \cQ, F \subset_f \lb\sH\rb, \delta>0$ (with $\sH^0 \in F$). This implies the lemma.
\end{proof}

%The proof of Theorem \ref{thm:K} begins by chosing any two $\pi$-generating sub-algebras $\cS, \bar{\cS}$ and $\cQ \subset \cS$. 
The next lemma is a key part of the proof of Theorem \ref{thm:K}: it enables us to define a map from $\Hom(\pi,\sigma,\cP,VW,\delta)$ to $\Hom(\pi,\sigma,\bar{\cP},V,\bar{\delta})$ for appropriate $\cP,V,W,\delta,\bar{\cP},V,\bar{\delta}$.

\begin{lem}\label{lem:5'}
Let $\cS,\bar{\cS}$ be any two $\pi$-generating sub-algebras of $\cB(\sG^0)$, $\kappa>0$ and $\cQ\subset\cS$ be a finite partition. Then there exists $\epsilon,\delta,\bar{\delta}>0$; finite partitions $\cP,\bar{\cP},\bar{\cQ}$; finite subsets $U,V,W \subset \lb \sH \rb$ and a homomorphism $\theta:\Sigma(\bar{\cP}^V) \to \Sigma(\cP^{VW})$ satisfying:
\begin{enumerate}
\item $\cQ \le \cP \subset \cS$;
\item $\bar{\cQ} \le \bar{\cP} \subset \bar{\cS}$;
\item $\sH^0 \in U \subset V$ and $\sH^0 \in W$;
\item for each $u \in U$, $\ra(u) \in U$ and $\sH^0 \setminus \ra(u) \in U$;
\item for each $v \in V$, $\ra(v) \in V$ and $\sH^0 \setminus \ra(v) \in V$;
\item for each $w \in W$, $\ra(w) \in W$ and $\sH^0 \setminus \ra(w) \in W$;
\item $\bar{\delta} < \epsilon/(8|\bar{\cQ}^{U}||U|)$;
\item $\delta < \min(\bar{\delta}/(9|\bar{\cP}^{V}||\cP^{W}| |W|), \epsilon/8)$;
\item $h_{\P,\mu}(\pi,\cQ,\cR) \le h_{\P,\mu}^{\epsilon}(\pi,\cQ,\cR) + \kappa$ for every finite partition $\cR$ which refines $\cQ$;
\item $h_{\P,\mu}(\pi,  \bar{\cQ}, \bar{\cP}, V,\bar{\delta}) \le h_{\P,\mu}(\pi,\bar{\cQ},\bar{\cS}) + \kappa;$
\item for every $Q \in \cQ$ there is a $\bar{Q} \in \Sigma(\bar{\cQ}^U)$ such that $\mu(Q \vartriangle \bar{Q}) < {\epsilon}/16$;
%$\Lambda_Q \subset \bar{\cQ}^U$ such that if $\bar{Q}:=\bigcup_{Y \in \Lambda_Q} Y $ then $\mu(Q \cap \bar{Q}) < {\epsilon}/16$;
\item for every $\bar{P} \in \Sigma(\bar{\cP}^V)$ there is a $P \in \Sigma(\cP^W)$ such that $\mu(P\vartriangle \bar{P}) < \bar{\delta}/(12|\bar{\cP}^{V}|)$;
% $\Lambda_P \subset \cP^W$ such that if $P':=\bigcup_{Y \in \Lambda_P} Y$ then  $\mu(P\vartriangle P') \le \bar{\delta}/(12|\cP^{V}|)$. 
\item $\mu(\theta(\bar{P}) \vartriangle \bar{P})< \min(\bar{\delta}/(12|\bar{\cP}^{V}|), {\epsilon}/(16|\bar{\cQ}^{V}|))$ for all $\bar{P} \in \Sigma(\bar{\cP}^V)$.
%\item for every $\sigma:\lb \sH \rb \to \lb d \rb$ and every $\phi \in \Hom(\pi,\sigma,\cP,VW,\delta)$, $|\phi(P)|d^{-1} < 2 \mu(P)$ for all $P \in \Sigma(\cP^{VW})$ with $\mu(P)>0$. If $\mu(P)=0$ then $|\phi(P)|d^{-1}<\delta$.
\end{enumerate}
\end{lem}

\begin{proof}
  By Lemma \ref{lem:5} there is an ${\epsilon}>0$ such that $h_{\P,\mu}(\pi,\cQ,\cR) \le h_{\P,\mu}^{\epsilon}(\pi,\cQ,\cR) + \kappa$ for every finite partition $\cR$ which refines $\cQ$.

Because $\bar{\cS}$ is $\pi$-generating there are a finite partition $\bar{\cQ} \subset \bar{\cS}$ and a nonempty finite set $U \subset \lb \sH \rb$ such that for every $Q \in \cQ$ there is a $\bar{Q} \in\Sigma(\bar{\cQ}^U)$  such that $\mu(Q \vartriangle \bar{Q}) < {\epsilon}/16$. By choosing $U$ larger if necessary we may assume $\sH^0 \in U$ and (4) is satisfied. 

Take a finite partition $\bar{\cP} \le \bar{\cS}$ with $\bar{\cP} \ge \bar{\cQ}$, a finite set $V \subset \lb \sH \rb$ containing $U\cup\{\sH^0\}$ and a $\bar{\delta}>0$ such that
$$h_{\P,\mu}(\pi,\bar{\cQ},\bar{\cS}) + \kappa \ge h_{\P,\mu}(\pi,  \bar{\cQ}, \bar{\cP},V,\bar{\delta}).$$
By shrinking $\bar{\delta}$ if necessary we may assume it is less than ${\epsilon}/(8|\bar{\cQ}|^{|U|}|U|)$. By choosing $V$ larger if necessary, we may assume (5) is satisfied.

Since $\cS$ is $\pi$-generating, there are a finite partition $\cP \subset \cS$ refining $\cQ$ and a nonempty finite set $W \subset \lb \sH \rb$ such that for every $\bar{P} \in \Sigma(\bar{\cP}^V)$ there is a $P \in \Sigma(\cP^W)$ such that  $\mu(P\vartriangle \bar{P}) < \bar{\delta}/(12|\bar{\cP}^{V}|)$. By choosing $\cP$ finer and $W$ larger if necessary, we may assume that $\sH^0 \in W$, (6) is satisfied and by Lemma \ref{lem:4}, that there is a homomorphism $\theta:\Sigma(\bar{\cP}^V) \to \Sigma(\cP^{VW})$ such that $\mu(\theta(\bar{P}) \vartriangle \bar{P})$ is less than both $\bar{\delta}/(12|\bar{\cP}^{V}|)$ and ${\epsilon}/(16|\bar{\cQ}^{V}|)$ for all $\bar{P} \in \Sigma(\bar{\cP}^V)$.

To finish choose $\delta>0$  smaller than $\min(\bar{\delta}/(9|\bar{\cP}^{V}||\cP^{W}| |W|),\epsilon/8)$.% and also small enough so that for every $\sigma:\lb \sH \rb \to \lb d \rb$ and every $\phi \in \Hom(\pi,\sigma,\cP,VW,\delta)$, $|\phi(P)|d^{-1} < 2 \mu(P)$ for all $P \in \Sigma(\cP^{VW})$ with $\mu(P)>0$.
\end{proof}

To motivate the next lemma, observe that if $F \subset_f [\sH]$ and $\cP$ is a finite partition of $\sG^0$ then every atom of $\cP^F$ has the form $\bigcap_{f\in F} f \cdot Y_f$ for some choice of $Y_f \in \cP$. This simple fact no longer holds if $F \subset_f \lb \sH\rb$ instead. The next lemma obtains a slightly weaker conclusion under an additional hypothesis.

\begin{lem}\label{lem:partition-hell}
Let $\cP$ be a finite partition of $\sG^0$ and $F \subset_f \lb\sH\rb$. Suppose that for every $f\in F$, $\ra(f) \in F$ and $\sH^0 \setminus \ra(f) \in F$. Then for every $Y \in \cP^F$ there exists a subset $F_Y \subset F$ and for every $f\in F_Y$ a set $Y_f \in \cP \cup \{\sG^0\}$ such that
$$Y = \bigcap_{f\in F_Y} f \cdot Y_f.$$
Moreover, we may choose $Y_f$ so that if $Y_f=\sG^0$ then $f=\sH^0 \setminus \ra(g)$ for some $g\in F$.
\end{lem}

\begin{proof}
Let $F_Y \subset F$ be the set of all $f\in F$ such that there exists a set $Y_f \in  \cP \cup \{\sG^0\}$ such that $Y \subset f \cdot Y_f$. We choose $Y_f$ so that it is the smallest set in $\cP \cup \{\sG^0\}$ with $Y \subset f\cdot Y_f$. It is obvious that $Y \subset \bigcap_{f\in F_Y} f \cdot Y_f.$

Given a subset $Z \subset \sG^0$, let $Z^+ = Z$ and $Z^- = \sG^0 \setminus Z$. By definition of $\cP^F$, there exists a map $\gamma:F \times \cP \to \{-,+\}$ such that
$$Y = \bigcap_{f\in F} \bigcap_{P\in \cP} (f\cdot P)^{\gamma(f,P)}.$$
So it suffices to show that if $f \in F$ and $P \in \cP$ are such that $Y \subset \sG^0 \setminus (f\cdot P)$ then there is a $g \in F_Y$ such that $g  \cdot Y_g \subset \sG^0 \setminus f\cdot P$. 

Because $F$ contains $\{\ra(f), \sH^0\setminus \ra(f):~f\in F\}$, it follows that $\pi^{-1}(\ra(f)) \in \Sigma(\cP^F)$ for all $f \in F$. Therefore, either $Y \subset \pi^{-1}(\ra(f))$ or $Y \cap \pi^{-1}(\ra(f))=\emptyset$. In the first case, $Y \subset  f \cdot (\sG^0 \setminus P)$ so there is a set $Q \in \cP$ with $Q \ne P$ such that 
$$Y \subset f\cdot Q  \subset f\cdot (\sG^0 \setminus P) \subset \sG^0\setminus f\cdot P.$$
So set $g=f, Y_g = Q$. In the second case, $Y \subset \sG^0 \setminus \pi^{-1}(\ra(f)) \subset \sG^0 \setminus f\cdot P$. So set $g=\sH^0\setminus \ra(f)$ and $Y_g = \sG^0$.
\end{proof}

\begin{proof}[Proof of Theorem \ref{thm:K}]
By symmetry it suffices to show that if $\cS,\bar{\cS}$ are any two $\pi$-generating sub-algebras of $\cB(\sG^0)$ then $h_{\P,\mu}(\pi,\cS) \le h_{\P,\mu}(\pi,\bar{\cS})$. Let $\kappa>0$, $\cQ\subset \cS$ be a finite partition and let $\epsilon,\delta,\bar{\delta},\cQ,\cP,\bar{\cP},\bar{\cQ},U,V,W,\theta$ be as in Lemma \ref{lem:5'}. It suffices to show that $h_{\P,\mu}(\pi,\cS,\cQ) \le h_{\P,\mu}(\pi,\bar{\cS}) + 2\kappa$.

%By Lemma \ref{lem:4} there is a homomorphism $\theta:\Sigma(\bar{\cP}^V) \to \Sigma(\cP^{VW})$ such that $\mu(\theta(P) \vartriangle P)$ is less than both $\bar{\delta}/(12|\bar{\cP}^{U}|)$ and ${\epsilon}/(16|\cQ^{U}|)$ for all $P \in \Sigma(\bar{\cP}^V)$.

Let $\sigma:\lb \sH \rb \to \lb d \rb$ for some $d \in \Nb$. Let $\phi \in \Hom(\pi,\sigma,\cP,VW,\delta)$. Set $\phi^\sharp = \phi\circ \theta$. The purpose of the next three claims is to show that $\phi^\sharp \in \Hom(\pi,\sigma,\bar{\cP},V,\bar{\delta})$ when $\sigma$ is sufficiently multiplicative. 

%We use the notation $P',\bar{Q}$ as in the previous lemma.

\noindent {\bf Claim 1}. If $\sigma:\lb \sH \rb \to \lb d \rb$ is $(VW,\delta)$-multiplicative then for every $v\in V$ and $P \in \Sigma(\cP^W)$, 
$$d^{-1} |\phi(v\cdot P) \vartriangle \sigma_v\cdot\phi(P)| < \frac{\bar{\delta}}{3|\bar{\cP}|}.$$
\begin{proof}[Proof of Claim 1]
%Vor $Y \in \bar{\cP}^V$, $V_Y \subset V$ and $\{Y_s\}_{s \in V}$ be a collection of partition elements $Y_s \in \bar{\cP}$ such that $Y = \bigcap_{s\in V_Y} s Y_s$. So $P' = \bigcup_{Y \in \Lambda_P} \bigcap_{s\in V_Y} s Y_s$ for some collection $\Lambda_P \subset\cP^W$ and

%By the previous lemma, there exists $P' \in \Sigma(\cP^W)$ such that $\mu(P\vartriangle P') \le \bar{\delta}/(12|\bar{\cP}^{V}|)$. 
Because $P \in \Sigma(\cP^W)$, there exists a collection $\Lambda_P  \subset\cP^W$ such that $P = \bigcup_{Y \in \Lambda_P} Y$. By Lemma \ref{lem:partition-hell}, for each such $Y$, there is a set $W_Y \subset W$ and for each $w\in W_Y$ a set $Y_w \in \Sigma(\cP)$ such that $Y = \bigcap_{w\in W_Y} w\cdot Y_w$. So %With notation 
\begin{eqnarray*}
&&d^{-1} |\phi(v\cdot P) \vartriangle \sigma_v\cdot\phi(P)|\\
 &\le &d^{-1} \sum_{Y \in \Lambda_P} \sum_{w\in W_Y}  |\phi(vw\cdot Y_w) \vartriangle \sigma_v\cdot\phi(w\cdot Y_w) | \\
&\le& d^{-1} \sum_{Y \in \Lambda_P} \sum_{w\in W_Y}  |\phi(vw\cdot Y_w)\vartriangle \sigma_{vw}\cdot\phi(Y_w)| + |\sigma_{vw}\cdot\phi(Y_w)\vartriangle \sigma_v \sigma_w\cdot\phi(Y_w)|\\
&\quad & \quad \quad \quad + |\sigma_v (\sigma_w\cdot\phi(Y_w) \vartriangle \phi(w\cdot Y_w))|\\
&\le& 3|\cP^{W}||W|\delta < \frac{\bar{\delta}}{3|\bar{\cP}|}.
\end{eqnarray*}
To see this, note that since $\phi \in \Hom(\pi,\sigma,\cP,VW,\delta)$, $ d^{-1}|\phi(vw\cdot Y_w)\vartriangle \sigma_{vw}\cdot\phi(Y_w)| < \delta$ (for all $w\in W, Y_w \in \Sigma(\cP)$). Similarly, $d^{-1}|\sigma_w\cdot\phi(Y_w) \vartriangle \phi(w\cdot Y_w)| < \delta$.  Because $\sigma$ is $(VW,\delta)$-multiplicative, $d^{-1}|\sigma_{vw}\cdot\phi(Y_w)\vartriangle \sigma_v \sigma_w\cdot\phi(Y_w)| < \delta$ as well. 
\end{proof}

As in Lemma \ref{lem:5'} for $\bar{P} \in \Sigma(\bar{\cP}^V)$, let $P \in \Sigma(\cP^W)$ be such that $\mu(P\vartriangle \bar{P}) < \bar{\delta}/(12|\bar{\cP}^{V}|)$.

\noindent {\bf Claim 2}. If $\sigma:\lb \sH \rb \to \lb d \rb$ is $(VW,\delta)$-multiplicative  then for every $v\in V$,
$$\frac{1}{d}\sum_{\bar{P}\in \bar{\cP}}  |\phi^\sharp(v\cdot\bar{P}) \vartriangle \sigma_v\cdot\phi^\sharp(\bar{P})| < \bar{\delta}.$$
\begin{proof}[Proof of Claim 2]
Note:
\begin{eqnarray*}
\frac{1}{d}\sum_{\bar{P}\in \bar{\cP}}  |\phi^\sharp(v\cdot\bar{P}) \vartriangle \sigma_v\cdot\phi^\sharp(\bar{P})|  &\le&\frac{1}{d} \sum_{\bar{P} \in \bar{\cP}} |\phi(\theta(v\cdot\bar{P}) \vartriangle v\cdot P)| + |\phi(v\cdot P) \vartriangle \sigma_v\cdot\phi(P)| + |\sigma_v\cdot\phi (P \vartriangle \theta(\bar{P}))|.
\end{eqnarray*}
Because $\phi \in \Hom(\pi,\sigma,\cP,VW,\delta)$, $d^{-1}|\phi(\theta(v\cdot\bar{P}) \vartriangle v\cdot P)| < \delta +  \mu( \theta(v\cdot\bar{P})\vartriangle v\cdot P)$. By items (12) and (13), 
$$\mu(\theta(v\cdot\bar{P}) \vartriangle v\cdot P) \le \mu( \theta(v\cdot\bar{P}) \vartriangle v\cdot\bar{P}) + \mu(P \vartriangle \bar{P}) < \bar{\delta}/(6|\bar{\cP}|).$$
So
$$d^{-1}|\phi(\theta(v\cdot\bar{P}) \vartriangle v\cdot P)| < \delta+\bar{\delta}/(6|\bar{\cP}|) <\bar{\delta}/(3|\bar{\cP}|).$$
Claim 1 implies $ d^{-1}|\phi(v\cdot P) \vartriangle \sigma_v\cdot\phi(P)| < \frac{\bar{\delta}}{3|\bar{\cP}|}$. Also $d^{-1}|\sigma_v\cdot\phi (P \vartriangle \theta(\bar{P}))| \le d^{-1}|\phi (P \vartriangle \theta(\bar{P}))| < \delta + \mu(P \vartriangle \theta(\bar{P}))$ and $\mu(P \vartriangle \theta(\bar{P})) \le \mu(P \vartriangle \bar{P}) + \mu(\bar{P} \vartriangle \theta(\bar{P})) < \bar{\delta}/(6|\bar{\cP}|)$ by items (12) and (13) of Lemma \ref{lem:5'}. Thus 
$$d^{-1}|\sigma_v\cdot\phi (P \vartriangle \theta(\bar{P}))|    < \delta + \bar{\delta}/(6|\bar{\cP}|)<\bar{\delta}/(3|\bar{\cP}|).$$
Putting all of these estimates together yields
\begin{eqnarray*}
\frac{1}{d}\sum_{\bar{P}\in \bar{\cP}}  |\phi^\sharp(v\cdot\bar{P}) \vartriangle \sigma_v\cdot\phi^\sharp(\bar{P})|  &<& \sum_{\bar{P} \in \bar{\cP}}\frac{\bar{\delta}}{3|\bar{\cP}|}+ \frac{\bar{\delta}}{3|\bar{\cP}|}+\frac{\bar{\delta}}{3|\bar{\cP}|} =  \bar{\delta}.
\end{eqnarray*}
\end{proof}

\noindent {\bf Claim 3}. If $\sigma$ is $(VW,\delta)$-multiplicative then $\phi^\sharp \in \Hom(\pi,\sigma,\bar{\cP},V,\bar{\delta})$.
\begin{proof}[Proof of Claim 3]
Let $\bar{P} \in \bar{\cP}^V$. By (12,13) of Lemma \ref{lem:5'} (and the fact that $\phi \in \Hom(\pi,\sigma,\cP,VW,\delta)$),
$$d^{-1}|\phi(\theta(\bar{P}) \vartriangle P)|  < \delta+\mu(\theta(\bar{P}) \vartriangle P) \le \delta +\mu(\theta(\bar{P}) \vartriangle \bar{P}) + \mu(\bar{P} \vartriangle P) < \bar{\delta}/(3|\bar{\cP}^V|).$$
Because $\phi \in \Hom(\pi,\sigma,\cP,VW,\delta)$, for every $\bar{P} \in \bar{\cP}^V$, we have $|d^{-1}|\phi(P)| - \mu(P)| < \delta < \bar{\delta}/(3|\bar{\cP}^V|)$. By (12) of Lemma \ref{lem:5'} again, we have $\mu(P \vartriangle \bar{P}) < \bar{\delta}/(6|\bar{\cP}^V|)$. So
\begin{eqnarray*}
\sum_{\bar{P} \in \bar{\cP}^V} | d^{-1}|\phi^\sharp(\bar{P})| - \mu(\bar{P}) | &\le& \sum_{\bar{P}\in \bar{\cP}^V} d^{-1}|\phi(\theta(\bar{P}) \vartriangle P)| + |d^{-1}|\phi(P)| - \mu(P)| + \mu(P \vartriangle \bar{P})\\
&<& \sum_{\bar{P} \in \bar{\cP}^V} \bar{\delta}/(3|\bar{\cP}^V|) + \bar{\delta}/(3|\bar{\cP}^V|) + \bar{\delta}/(6|\bar{\cP}^V|) \le \bar{\delta}.
\end{eqnarray*}
Together with Claim 2, this implies Claim 3.
\end{proof}

\noindent {\bf Claim 4}. Let ${\bar{\epsilon}}>0$ be such that ${\bar{\epsilon}}< {\epsilon}/(8|\bar{\cQ}^U||U|)$. Let $\phi, \psi \in \Hom(\pi,\sigma,\cP,VW,\delta) $ be two elements with $\rho_{\bar{\cQ}}(\phi^\sharp,\psi^\sharp) < 2{\bar{\epsilon}}$. If $\sigma$ is $(VW,\delta)$-multiplicative then $\rho_{\cQ}(\phi,\psi)<{\epsilon}$. 
\begin{proof}[Proof of Claim 4]
Let $Q \in \cQ$. Recall from Lemma \ref{lem:5'} that there is a $\bar{Q} \in \Sigma(\bar{\cQ}^U)$ such that $\mu(Q \vartriangle \bar{Q}) < {\epsilon}/16$. By Lemma \ref{lem:partition-hell} for $Y \in \bar{\cQ}^U$ there exists $U_Y \subset U$ and for each $u\in U_Y$ a set $Y_u \in \bar{\cQ} \cup \{\sG^0\}$ such that $Y = \bigcap_{u\in U_Y} u\cdot Y_u$. So $\bar{Q} = \bigcup_{Y \in \Lambda_Q} \bigcap_{u\in U_Y} u\cdot Y_u$ for some collection $\Lambda_Q \subset \bar{\cQ}^U$.

Because $\bar{\cQ}\le \bar{\cP}$ and $U \subset V$, Claim 2 implies
$$d^{-1}| \phi^\sharp(u\cdot Y_u) \vartriangle \sigma_u\cdot\phi^\sharp(Y_u) | < \bar{\delta}, \quad d^{-1} | \sigma_u\cdot\psi^\sharp(Y_u) \vartriangle \psi^\sharp(u\cdot Y_u) |  < \bar{\delta} \quad \forall u\in U.$$
Because $\rho_{\bar{\cQ}}(\phi^\sharp,\psi^\sharp) < 2{\bar{\epsilon}}$, we also have 
$$d^{-1}| \sigma_u\cdot(\phi^\sharp(Y_u) \vartriangle \psi^\sharp(Y_u)) | \le d^{-1}| \phi^\sharp(Y_u) \vartriangle \psi^\sharp(Y_u) | < 2{\bar{\epsilon}}.$$
Therefore,
\begin{eqnarray*}
&&d^{-1}| \phi^\sharp(\bar{Q}) \vartriangle \psi^\sharp(\bar{Q}) |\\
 &\le& d^{-1}\sum_{Y \in \Lambda_Q} \sum_{u\in U_Y} | \phi^\sharp(u\cdot Y_u) \vartriangle \psi^\sharp(u\cdot Y_u) | \\ 
&\le& d^{-1}\sum_{Y \in \Lambda_Q} \sum_{u\in U_Y} | \phi^\sharp(u\cdot Y_u) \vartriangle \sigma_u\cdot\phi^\sharp(Y_u) | + | \sigma_u\cdot(\phi^\sharp(Y_u) \vartriangle \psi^\sharp(Y_u)) | + | \sigma_u\cdot\psi^\sharp(Y_u) \vartriangle \psi^\sharp(u\cdot Y_u) | \\ 
&<& \sum_{Y \in \Lambda_Q} \sum_{u\in U} 2\bar{\delta} + 2{\bar{\epsilon}} \le 2(\bar{\delta}+{\bar{\epsilon}})|{\bar{\cQ}}^U||U| \le {\epsilon}/2.
\end{eqnarray*}
Also
\begin{eqnarray*}
\mu(Q \vartriangle \theta(\bar{Q})) &\le& \mu(Q \vartriangle \bar{Q}) + \mu(\bar{Q} \vartriangle \theta(\bar{Q}))< {\epsilon}/16 + {\epsilon}/(16|{\bar{\cQ}}^U|) \le {\epsilon}/8.
\end{eqnarray*}
The second to last inequality above follows from items (2,3,13) of Lemma \ref{lem:5'}. Therefore,
\begin{eqnarray*}
\rho_{\cQ}(\phi,\psi) &=& \max_{Q \in \cQ} \frac{1}{d} |\phi(Q) \vartriangle \psi(Q)| \\
&\le& \max_{Q \in \cQ} \frac{1}{d} |\phi(Q  \vartriangle \theta(\bar{Q}))| + \frac{1}{d} |\phi^\sharp(\bar{Q}) \vartriangle \psi^\sharp(\bar{Q})| + \frac{1}{d}  |\psi(\theta(\bar{Q})  \vartriangle Q)|\\
&\le& {\epsilon}/2 + 2\delta+2\max_{Q \in \cQ} \mu(Q  \vartriangle \theta(\bar{Q}))   < {\epsilon}.
\end{eqnarray*}
The second inequality above uses that $\phi, \psi \in \Hom(\pi,\sigma,\cP,VW,\delta)$. This proves the claim.
\end{proof}

Let $\Gamma:\Hom(\pi,\sigma,\cP,VW,\delta) \to \Hom(\pi,\sigma,\bar{\cP},V,\bar{\delta})$ be the map $\Gamma(\phi)=\phi^\sharp$. If $\sigma$ is $(VW,\delta)$-multiplicative then by Claim 3, $\Gamma$ really does map into $\Hom(\pi,\sigma,\bar{\cP},V,\bar{\delta})$ as required. It follows from Claim 4 that for every $(\rho_{\cQ},{\epsilon})$-separated subset $Z \subset \Hom(\pi,\sigma,\cP,VW,\delta)$, the image $\Gamma(Z)$ is $(\rho_{\bar{\cQ}},2{\bar{\epsilon}})$-separated. So,
$$N_{\epsilon}( \Hom(\pi,\sigma,\cP,VW,\delta), \rho_{\cQ}) \le N_{2{\bar{\epsilon}}}(  \Hom(\pi,\sigma,\bar{\cP},V,\bar{\delta}), \rho_{\bar{\cQ}}) \le | \Hom(\pi,\sigma,\bar{\cP},V,\bar{\delta})|_{\bar{\cQ}}. $$
By items (9,10) of Lemma \ref{lem:5'}, 
\begin{eqnarray*}
h_{\P,\mu}(\pi,\cQ,\cS) &\le& h_{\P,\mu}(\pi,\cQ, \cP) \stackrel{(9)}{\le} h_{\P,\mu}^{\epsilon}(\pi,\cQ, \cP)  + \kappa\\
&\le& \lim_{j\to\beta} \frac{1}{d_j}  \log \left(\| N_{\epsilon}( \Hom(\pi,\cdot,\cP,VW,\delta), \rho_{\cQ}) \|_{p,\P_j}\right)+ \kappa  \\
&\le& \lim_{j\to\beta} \frac{1}{d_j}  \log \left(\| | \Hom(\pi,\cdot,\bar{\cP},V,\bar{\delta})|_{\bar{\cQ}}  \|_{p,\P_j}\right)+ \kappa=h_{\P,\mu}(\pi,  \bar{\cQ}, \bar{\cP}, V,\bar{\delta}) + \kappa  \\
&\stackrel{(10)}{\le}& h_{\P,\mu}(\pi,{\bar{\cQ}},\bar{\cS}) + 2\kappa \le h_{\P,\mu}(\pi,\bar{\cS}) + 2\kappa.
\end{eqnarray*}
Because $\cQ \le \cS$ is an arbitrary finite partition and $\kappa>0$ is also arbitrary, we conclude that $h_{\P,\mu}(\pi,\cS) \le h_{\P,\mu}(\pi,\bar{\cS}) $ as required.
\end{proof}

\section{Replacing $\lb\sH\rb$ with $\lb\sH\rb_{top}$}\label{sec:semi}

The purpose of this section is to show that $\lb \sH \rb$ can be replaced with $\lb \sH \rb_{top}$ in the definition of measure entropy under mild conditions, explained next.

\begin{defn}\label{defn:ac}
Let $(\sH,\nu)$ be a discrete pmp groupoid, $\sigma:\lb\sH\rb \to \lb d \rb$, $F \subset_f \lb \sH \rb$ and $\delta>0$. We say $\sigma$ is {\em $(F,\delta)$-continuous} if $|\sigma_f  \vartriangle \sigma_g|_d<\delta + \nu(f \vartriangle g)~\forall f,g \in F$. If $\P=\{\P_j\}_{j\in J}$ is a sofic approximation to $(\sH,\nu)$, then we say $\P$ is  {\em asymptotically continuous} if for every $F \subset_f \lb \sH \rb,  \delta>0$, there exists $j \in J$ such that $j' \ge j$ implies $\P_{j'}$-almost every $\sigma$ is $(F,\delta)$-continuous.
\end{defn}

To justify our claim that the condition above is mild, first note that if $\sH$ is a group (as in example \ref{example:group}) then every map $\sigma:\lb \sH \rb \to \lb d \rb$ is $(F,\delta)$-continuous and therefore every sofic approximation is asymptotically continuous. This is because if $f \ne g$ then $\nu(f \vartriangle g) = 2 \ge |\sigma_f \vartriangle \sigma_g|_d$. For further justification, the next lemma implies that if $\sigma$ is sufficiently multiplicative and trace-preserving then it is $(F,\delta)$-continuous. We will not need it in the rest of the paper. 
%We say that it is a {\em strong sofic approximation} if, in addition, 
\begin{lem}\label{lem:basic-formulas4}
Let $f,g \in \lb \sH \rb$, $F=\{f,g,f^{-1}g,f^{-1},\so(f),\ra(f),\so(g),\ra(g)\}$ and $\delta>0$. If $\sigma:\lb \sH \rb \to \lb d \rb$ is $(F,\delta)$-multiplicative and $(F,\delta)$-trace-preserving then $\sigma$ is $(\{f,g\},58\delta)$-continuous.
%$$|\sigma_f \vartriangle \sigma_g|_d \le |f \vartriangle g|_d + ...$$
\end{lem}

\begin{remark}
This lemma implies that if each $\P_j$ is concentrated on a single map $\sigma_j$ then $\P=\{\P_j\}_{j\in J}$ is asymptotically continuous. Therefore, any sofic groupoid admits an asymptotically continuous sofic approximation.
\end{remark}

\begin{proof}
Recall that $\tr_d(\sigma_{f^{-1}g}) = |\sigma_{f^{-1}g} \cap \Delta^0_d|_d$. Because $\sigma$ is $(F,\delta)$-multiplicative, $|\sigma_{f^{-1}g} \vartriangle \sigma_{f^{-1}}\sigma_g|_d \le \delta$. By Lemma \ref{lem:basic-formulas2}, $|\sigma_{f^{-1}} \vartriangle \sigma_f^{-1}|_d \le 15 \delta$. So
$$\tr_d(\sigma_{f^{-1}g}) - |\sigma_f^{-1}\sigma_g \cap \Delta^0_d|_d \le 16 \delta.$$
Observe that $\so(\sigma_f \cap \sigma_g) = \sigma_f^{-1}\sigma_g \cap \Delta^0_d$. By Lemma \ref{lem:basic-formulas2},
\begin{eqnarray*}
\tr_d(\sigma_{f^{-1}g})  &\le& 16 \delta + |\so(\sigma_f \cap \sigma_g)|_d = 16 \delta + |\sigma_f \cap \sigma_g|_d= 16\delta + \frac{1}{2} \left( |\sigma_f|_d + |\sigma_g|_d - |\sigma_f \vartriangle \sigma_g|_d \right)\\
&=& 16\delta + \frac{1}{2} \left( |\so(\sigma_f)|_d + |\so(\sigma_g)|_d - |\sigma_f \vartriangle \sigma_g|_d \right)\le 26\delta + \frac{1}{2} \left( |\sigma_{\so(f)}|_d + |\sigma_{\so(g)}|_d - |\sigma_f \vartriangle \sigma_g|_d \right)\\
&\le& 27\delta + \frac{1}{2} \left( |\sigma_{\so(f)} \cap \Delta^0_d|_d + |\sigma_{\so(g)} \cap \Delta^0_d|_d - |\sigma_f \vartriangle \sigma_g|_d \right)\\
&=& 27\delta + \frac{1}{2} \left( \tr_d(\sigma_{\so(f)}) + \tr_d(\sigma_{\so(g)}) - |\sigma_f \vartriangle \sigma_g|_d \right).
\end{eqnarray*}
Rearranging terms we obtain:
\begin{eqnarray*}
|\sigma_f \vartriangle \sigma_g|_d &\le& 54\delta + \tr_d(\sigma_{\so(f)}) + \tr_d(\sigma_{\so(g)}) -  2\tr_d(\sigma_{f^{-1}g}).
\end{eqnarray*}
Because $\sigma$ is $(F,\delta)$-trace preserving
\begin{eqnarray*}
|\sigma_f \vartriangle \sigma_g|_d &<& 58\delta + \nu(\so(f)) + \nu(\so(g)) -  2\tr_{\sH}(f^{-1}g) = 58\delta + \nu(f)+\nu(g) - 2 \nu(f^{-1}g \cap \sH^0)\\
&=& 58\delta + \nu(f)+\nu(g) - 2 \nu(f \cap g) = 58\delta + \nu(f \vartriangle g).
\end{eqnarray*}
So $\sigma$ is $(\{f,g\},58\delta)$-continuous as required.
\end{proof}

The main result of this section is:
\begin{thm}\label{thm:top-replace}
Let $(\sH,\nu)$ be a discrete pmp \'etale topological groupoid. Assume $\nu$ is regular, $\sH^0$ is compact and metrizable and $\P$ is asymptotically continuous. Then for every pmp class-bijective extension $\pi:(\sG,\mu) \to (\sH,\nu)$ and finite Borel partitions  $\cQ \le \cP$ of $\sG^0$, we have
$$h_{\P,\mu}(\pi,\cQ,\cP) = \inf_{F \subset_f \lb\sH\rb_{top}} \inf_{\delta>0} \lim_{j\to \beta} d_j^{-1} \log \| |\Hom(\pi,\cdot, \cP,F,\delta)|_\cQ \|_{p,\P_j}.$$
In other words, we can replace $\lb \sH\rb$ in the definition of $h_{\P,\mu}(\pi,\cQ,\cP)$ with $\lb\sH\rb_{top}$.
\end{thm}
This theorem is crucial to our proof that the measure entropy defined in \S \ref{sec:meas1} agrees with the measure entropy defined in \S \ref{sec:meas2}; which itself is key to establishing the variational principle. Theorem \ref{thm:top-replace} is a consequence of the next three lemmas. 

\begin{notation}
 Let $X$ be a metrizable space and $\lambda$ a regular Borel measure on $X$. For any subset $Y \subset X$, let $\partial Y = \overline{Y} \setminus \textrm{interior}(Y)= \overline{Y} \cap \overline{X\setminus Y}$. We let $\cB_\partial(X,\lambda)$ denote the collection of all Borel subsets $Y \subset X$ with $\lambda(\partial Y)=0$.
\end{notation}

% Given any set $A \subset \sG^0$, the boundary of $A$ with respect to $\rho$ is the set $\partial_\rho A:=\{x \in A:~ \forall r>0 ~\exists y \notin A \textrm{ such that } \rho(x,y)<r\}$. We say that a partition $\cP$ of $\sG^0$ has {\em measure zero boundary} if for every $P \in \cP$, $\mu(\partial_\rho P)=0$.

%Recall that $h_{\P,\mu}(\pi) = \sup_\cQ \inf_{\cP \ge \cQ} h_{\P,\mu}(\pi,\cQ,\cP)$ where the supremum is over all finite Borel partitions $\cQ$ of $\sG^0$ and the infimum is over all finite Borel partitions $\cP\ge \cQ$ of $\sG^0$. Let $h^0_{\P,\mu}(\pi) := \sup^0_\cQ \inf^0_{\cP \ge \cQ} h_{\P,\mu}(\pi,\cQ,\cP)$ where the supremum is over all finite Borel partitions $\cQ$ of $\sG^0$ {\em with measure zero boundary}  and the infimum is over all finite Borel partitions $\cP\ge \cQ$ of $\sG^0$ {\em with measure zero boundary}. The superscript $0$ on $\sup^0$ and $\inf^0$ is intended to remind us that we require the associated partitions to have measure zero boundaries.

\begin{lem}\label{lem:boundary-zero}
%Given any metrizable space $X$ with a regular Borel measure $\lambda$, $\cB_\partial(X,\lambda)$ is an algebra. 
Let $X$ be a metrizable space with a regular Borel measure $\lambda$. Then for any measurable $C \subset X$ with $\lambda(C)<\infty$ and $\epsilon>0$ there exist $A \in \cB_\partial(X,\lambda)$ with $\lambda(A \vartriangle C)<\epsilon$. Moreover if $C$ is open then we can choose $A$ to be a closed subset of $C$.
\end{lem}

\begin{proof}
%To simply notation, let $Y^c = X\setminus Y$ for any $Y \subset X$.

%Let $A,B \in \cB_\partial(X,\lambda)$. Observe that $\partial(A \cup B) \subset \partial A \cup \partial B$. Hence $\lambda(\partial(A\cup B))=0$. Also, $\partial(A\cap B) \subset \partial A \cup \partial B$ which implies $\lambda(\partial(A\cap B))=0$. To see this, let $x \in \partial(A \cap B)$. Then there exist elements $\{y_n\}_{n=1}^\infty\subset (A \cap B)^c$ with $\lim_{n\to\infty} y_n =x$. Because $\{y_n\}_{n=1}^\infty \subset A^c \cup B^c$, either $\{y_n\}_{n=1}^\infty \cap A^c$ is infinite or $\{y_n\}_{n=1}^\infty \cap B^c$ is infinite. In the first case, $x \in \partial A$ and in the second $x \in \partial B$ which proves the claim. So $\cB_\partial(X,\lambda)$ is closed under finite unions and intersections which means it is an algebra.

%Because $(A\cup B)^c = (A^c \cap B^c)$ and $(A\cap B)^c = A^c \cup B^c$, it now follows that $A \cup B \in \cB_\partial(X,\lambda)$ and $A \cap B \in \cB_\partial(X,\lambda)$. So $\cB_\partial(X,\lambda)$ is an algebra.

Let $\rho$  be a continuous metric on $X$. For any subset $L \subset X$ and $r>0$, let $N_r(L) = \{x \in X:~ \exists y \in L ~\rho(x,y) \le r\}$. Let $C \subset X$ be a measurable set with finite measure and $\epsilon>0$. Because $\lambda$ is regular, there exists a compact set $K \subset C$ with $\lambda(C \setminus K)<\epsilon/2$. Because $K$ is closed, $K= \cap_{r>0} N_r(K)$. So there exists an $r>0$ such that $\lambda(N_r(K) \setminus K) < \epsilon/2$. If $C$ is open, we can choose $r>0$ so that $N_r(K) \subset C$.

Suppose $0<s<t$. We claim that $\partial N_s(K) \cap \partial N_t(K) = \emptyset$. Indeed, if $x \in \partial N_s(K) \cap \partial N_t(K)$ then there exists a sequence $\{y_i\}_{i=1}^\infty \subset N_t(K)^c$ with $\lim_{i\to\infty} y_i = x$. Since $\rho(y_i,K)>t$ for every $i$, $\liminf_{i\to\infty} \rho(y_i,K) \ge t$ which implies $\rho(x,K) \ge t$. Since $x \in \overline{N_t(K)}=N_t(K)$, we must have $\rho(x,K)=t$. However, $x \in N_s(K)$ as well, so $\rho(x,K) \le s < t$. This contradiction proves the claim.

Because any uncountable sum of positive numbers equals positive infinity, for any $x \in X$ there is at most a countable number of numbers $t>0
$ such that $\lambda(\partial N_t(x))>0$ (where $N_t(x)=N_t(\{x\})$). By compactness there exist a finite set $x_1,\ldots, x_n \in K$ and numbers $r_1,\ldots, r_n>0$ such that
\begin{itemize}
\item $r_i < r$ for all $i$;
\item $\lambda(\partial N_{r_i}(x_i) )= 0$ for all $i$;
\item $K \subset \cup_{i=1}^n N_{r_i}(x_i)$.
\end{itemize}
Let $A = \cup_{i=1}^n N_{r_i}(x_i)$. Because $\cB_\partial(X,\lambda)$ is an algebra (by Lemma \ref{lem:algebra}), $A \in \cB_\partial(X,\lambda)$. By construction, $K \subset A \subset N_r(K)$ which implies 
$$A \vartriangle C \subset (C \setminus K) \cup (N_r(K) \setminus K) \Rightarrow \lambda(A \vartriangle C) <\epsilon$$
as required. Moreover if $C$ is open then $A \subset N_r(K) \subset C$. %Moreover $A$ is closed because each $N_{r_i}(x_i)$ is closed.
\end{proof}

%The lemma above implies that $\cB_\partial(\sG^0,\mu)$ is $\pi$-generating. So Theorem \ref{thm:K} implies $h_{\P,\mu}(\pi) = h_{\P,\mu}(\pi,\cB_\partial(\sG^0,\mu))$. 

\begin{lem}\label{lem:top-dense}
If $\sH$ is \'etale, $\sH^0$ is compact and metrizable and $\nu$ is regular then $\lb \sH \rb_{top}$ is dense in $\lb \sH \rb$ in the measure-algebra sense. This means that for every $f \in \lb \sH \rb$ and every $\epsilon>0$ there exists $f' \in \lb \sH \rb_{top}$ such that $\nu( f \vartriangle f') <\epsilon$.
\end{lem}

\begin{proof}
Because $\nu$ is regular, there exists a compact set $K \subset f$ and an open set $O \supset f$ such that $\nu(O \setminus K) < \epsilon$. Because $K$ is compact and $\sH$ is \'etale, there exists a finite collection $U_1,\ldots, U_n$ of bisections with $K \subset \cup_{i=1}^n U_i \subset O$. 

By Lemma \ref{lem:boundary-zero} there exist closed sets $T_i, V_i \in \cB_\partial(\sH^0,\nu)$ such that $T_i \subset \ra(U_i), V_i \subset \so(U_i)$ and
$$\sum_{i=1}^n \nu(T_i \vartriangle \ra(U_i)) + \nu(V_i \vartriangle \so(U_i)) < \epsilon.$$
So
$$\nu\left(\bigcup_{i=1}^nT_iU_iV_i \vartriangle \bigcup_{i=1}^nU_i\right) \le \sum_{i=1}^n \nu(T_iU_iV_i \vartriangle U_i) \le \epsilon.$$ 
We define $U'_i$ by $U'_1=T_1U_1V_1$ and
$$U'_{i} = T_iU_{i}V_i \setminus \left( \bigcup_{j<i} T_jU_jV_j \right)$$
for $i>1$. %Finally, let $f' = \cup_{i=1}^n U'_i$. 

For $x \in \sH^0$, let $m(x)$ be the number of indices $i$ such that $x \in \so(U'_i)$. Let $Y_s = \{x\in \sH^0:~m(x)\ge 2\}$. Then
\begin{eqnarray*}
\nu(K) + 2\epsilon &\ge& \nu(O) + \epsilon \ge \nu\left( \bigcup_{i=1}^n U_i\right)  + \epsilon\ge \nu\left( \bigcup_{i=1}^n T_iU_iV_i \right)= \nu\left( \bigcup_{i=1}^n U'_i \right)\\
 &=&  \int_{\sH^0} m(x)~d\nu(x)\ge  \nu\left( \so\left( \bigcup_{i=1}^n U'_i \right)\right) + \nu(Y_s) \ge \nu\left( \so\left( \bigcup_{i=1}^n U_i\right)\right) + \nu(Y_s) - \epsilon\\
&\ge& \nu( \so(K)) + \nu(Y_s) - \epsilon = \nu(K) + \nu(Y_s) - \epsilon.
\end{eqnarray*}
Therefore, $3\epsilon \ge \nu(Y_s)$. Similarly, if $Y_r$ is the set of all $x \in \sH^0$ such that there exist $i\ne j$ such that $x \in \ra(U'_i) \cap \ra(U'_j)$ then $\nu(Y_r) \le 3\epsilon$.

%Because $\cB_\partial(\sH^0,\nu)$ is an algebra and $\so(U'_i),\ra(U'_i) \in \cB_\partial(\sH^0,\nu)$, it follows that $Y_s,Y_r \in \cB_\partial(\sH^0,\nu)$. 
Finally, let 
$$U''_i =(\sH^0 \setminus Y_r) U'_i (\sH^0 \setminus Y_s)$$
and $f'=\bigcup_{i=1}^nU''_i$. 

We need to show that $f' \in \lb \sH \rb_{top}$. First we claim that $\so(T_iU_iV_i) \in \cB_\partial(\sH^0,\nu)$ for each $i$. Indeed, $\so(T_iU_iV_i) = V_i \cap U_i^{-1}\cdot T_i$. Because $U_i$ is a measure-preserving homeomorphism and $T_i\subset \ra(U_i)$ is closed, 
$$\nu(\partial (U_i^{-1}\cdot T_i)) = \nu(U_i^{-1}\cdot \partial T_i) = \nu(\partial T_i)=0.$$
So $U_i^{-1}\cdot T_i \in \cB_\partial(\sH^0,\nu)$ which, by Lemma \ref{lem:algebra}, implies $\so(T_iU_iV_i) \in \cB_\partial(\sH^0,\nu)$. 

Next we claim that $T_iU_iV_i \in \cB_\partial(\sH,\nu)$. Indeed, $T_iU_iV_i  = U_i\so(T_iU_iV_i)$. Because $\so(T_iU_iV_i) \subset \so(U_i)$ is closed and $\so$ restricted to $U_i$ is a measure-preserving homeomorphism, 
$$\nu(\partial T_iU_iV_i) = \nu(\partial U_i \so(T_iU_iV_i)) = \nu(U_i(\partial \so(T_iU_iV_i))) = \nu(\partial \so(T_iU_iV_i))=0.$$
So $T_iU_iV_i \in \cB_\partial(\sH,\nu)$ as claimed.

Lemma \ref{lem:algebra} now implies $U'_i \in \cB_\partial(\sH,\nu)$. Because $U'_i \subset T_iU_iV_i$ and $T_iU_iV_i$ is closed it follows that $\overline{U'_i} \subset U_i$. Similarly, $\so(U'_i) \subset V_i$ and $V_i$ is closed implies $\overline{\so(U'_i)} \subset V_i \subset \so(U_i)$. Therefore $\partial \so(U'_i) = \so(\partial U'_i)$ (since $U_i$ is a bisection). Because the source map restricted to $U_i$ is measure-preserving,
$$\nu(\partial \so(U'_i)) = \nu(\so(\partial U'_i)) = \nu(\partial U'_i) = 0.$$
So $\so(U'_i) \in \cB_\partial(\sH^0,\nu)$. Similarly, $\ra(U'_i) \in \cB_\partial(\sH^0,\nu)$. 

Next observe that $Y_s = \bigcup_{i\ne j} \so(U'_i) \cap \so(U'_j)$ implies $Y_s \in \cB_\partial(\sH^0,\nu)$. Similarly, $Y_r \in \cB_\partial(\sH^0,\nu)$. Now we claim that $\so(U''_i) \in \cB_\partial(\sH^0,\nu)$. Note that
$$\so(U''_i) = \so(U'_i) \cap (\sH^0 \setminus Y_s) \cap (U'_i)^{-1}\cdot (\sH^0 \setminus Y_r)=  \so(U'_i) \cap (\sH^0 \setminus Y_s) \cap (U_i)^{-1}\cdot (\ra(U'_i) \setminus Y_r).$$%= \so(U'_i) \cap (\sH^0 \setminus Y_s) \cap (U_i)^{-1}\cdot (\sH^0 \setminus Y_r).$$
The last equality above occurs because $(U'_i)^{-1}\cdot X = U_i^{-1}\cdot (\ra(U'_i) \cap X)$ for any set $X \subset \sH^0$ since $U'_i \subset U_i$.

In order to show that $\so(U''_i) \in \cB_\partial(\sH^0,\nu)$, it suffices to show that $(U_i)^{-1}\cdot (\ra(U'_i) \setminus Y_r)\in \cB_\partial(\sH^0,\nu)$ (by Lemma \ref{lem:algebra}). Because $\overline{\ra(U'_i)} \subset \ra(U_i)$ and $\overline{\so(U'_i)}\subset \so(U_i)$,
$$\overline{\ra(U'_i) \setminus Y_r} \subset \ra(U_i), \quad \overline{(U_i)^{-1}\cdot (\ra(U'_i) \setminus Y_r)} \subset \so(U_i).$$
Since $U_i^{-1}$ is a bisection this implies
$$\partial U_i^{-1} \cdot (\ra(U'_i)  \setminus Y_r)  = U_i^{-1}\cdot (\partial(\ra(U'_i) \setminus Y_r)).$$
Since $U_i^{-1}\cdot $ is measure-preserving, this implies $U_i^{-1}\cdot (\ra(U'_i) \setminus Y_r) \in \cB_\partial(\sH^0,\nu)$ as required. So $\so(U''_i) \in \cB_\partial(\sH^0,\nu)$. 
%This follows from the computation:
%$$\nu(\partial (U_i^{-1}\cdot Y_r)) = \nu(U_i^{-1}\cdot(\partial Y_r)) = \nu(\partial Y_r)=0$$
%which occurs because the map $x \mapsto U_i^{-1}\cdot x$ is a measure-preserving homeomorphism from $\ra(U_i)$ to $\so(U_i)$ and $Y_r \subset \ra(U_i)$ is closed.

By similar reasoning, $\ra(U''_i) \in \cB_\partial(\sH^0,\nu)$. Because 
$$U''_i \subset U'_i \subset \overline{U'_i} \subset U_i,\quad \overline{\so(U''_i)} \subset \overline{\so(U'_i)} \subset \so(U_i),\quad\overline{\ra(U''_i)} \subset \overline{\ra(U'_i)} \subset \ra(U_i),$$
$U_i$ is a bisection and the $U''_i$ have pairwise disjoint sources and ranges, it follows that $f' = \bigcup_{i=1}^nU''_i \in \lb \sH \rb_{top}$.

%Because $f,f' \in \lb \sH \rb_{top}$, the source map is at most 2-1 on $f \vartriangle f'$. So
Observe that
\begin{eqnarray*}
\bigcup_{i=1}^nU''_i \vartriangle \bigcup_{i=1}^nU'_i &=& \bigcup_{i=1}^nU'_i \setminus \bigcup_{i=1}^nU''_i \\
&\subset& (O \setminus K) \cup (\so^{-1}(\so(K) \cap Y_s) \cap K) \cup (\ra^{-1}(\ra(K) \cap Y_r) \cap K).
\end{eqnarray*}
Because $K \subset f \in \lb\sH\rb$, the source and range maps restricted to $K$ are measure-preserving. For example, $\nu(\so^{-1}(\so(K) \cap Y_s) \cap K) = \nu(\so(K)\cap Y_s) \le 3\epsilon$. Thus 
\begin{eqnarray*}
\nu(f' \vartriangle f) &\le& \nu\left( \bigcup_{i=1}^nU''_i \vartriangle \bigcup_{i=1}^nU'_i\right) + \nu\left(\bigcup_{i=1}^nU'_i \vartriangle \bigcup_{i=1}^n U_i\right) + \nu\left( \bigcup_{i=1}^n U_i \vartriangle f\right)\\
&\le& \epsilon+\nu(Y_s) + \nu(Y_r) + \epsilon + \epsilon \le 9\epsilon.
\end{eqnarray*}

\end{proof}

\begin{lem}
Let $(\sH,\nu)$ be a discrete pmp groupoid. Suppose $\P$ is asymptotically continuous and $\cF \subset \lb \sH \rb$ is dense in the sense that for every $f \in \lb \sH \rb$ and $\epsilon>0$ there exists $f' \in \cF$ such that $\nu(f \vartriangle f') < \epsilon$. We also require $\sH^0 \in \cF$. Then for every pmp class-bijective extension $\pi:(\sG,\mu) \to (\sH,\nu)$ and finite Borel partitions  $\cQ \le \cP$ of $\sG^0$, we have
$$h_{\P,\mu}(\pi,\cQ,\cP) = \inf_{F \subset_f \cF} \inf_{\delta>0} \lim_{j\to \beta} d_j^{-1} \log \| |\Hom(\pi,\cdot, \cP,F,\delta)|_\cQ \|_{p,\P_j}.$$
In other words, we can replace $\lb \sH\rb$ in the definition of $h_{\P,\mu}(\pi,\cQ,\cP)$ with $\cF$.
\end{lem}

\begin{proof}
It is immediate that
$$h_{\P,\mu}(\pi,\cQ,\cP) \le \inf_{F \subset_f \cF} \inf_{\delta>0} \lim_{j\to \beta} d_j^{-1} \log \| |\Hom(\pi,\cdot, \cP,F,\delta)|_\cQ \|_{p,\P_j}$$
so we need only prove the opposite inequality.

Let $F \subset_f \lb \sH \rb$ be such that $\sH^0 \in F$ and for every $f\in F$, $\ra(f),\sH^0\setminus \ra(f) \in F$. Let $\delta>0$ and choose $\delta'$ so that $0<\delta' < \delta/(10|F| |\cP^F|^2)$.  By hypothesis there exist $F' \subset_f \cF$ and a map $\theta:F \to F'$ such that $\nu(f \vartriangle \theta(f)) < \delta'$ for all $f\in F$. We require that $\sH^0 \in F'$ and $\theta(\sH^0) = \sH^0$. We will show that if $\sigma$ is $(F \cup F', \delta')$-continuous then $| \Hom(\pi,\sigma,\cP,F',\delta')|_\cQ \le | \Hom(\pi,\sigma,\cP,F,\delta)|_\cQ$.

Define $\Psi: \cP^F \to \cP^{F'}$ as follows. By Lemma \ref{lem:partition-hell}, for every $P \in \cP^F$, there exists a finite set $\Lambda_P \subset F$ and for each $f \in \Lambda_P$ a set $Y_f \in \cP \cup \{\sG^0\}$ such that $P = \cap_{f\in \Lambda_P} f \cdot Y_f$. Moreover, by choosing $\Lambda_P$ to be as large as possible and each $Y_f$ to be as small as possible, this representation is uniquely determined by $P$. We define
$$\Psi( P) = \Psi\left( \bigcap_{f\in \Lambda_P} f \cdot Y_f\right) := \bigcap_{f\in \Lambda_P} \theta(f) \cdot Y_f.$$
Because $\nu(f \vartriangle \theta(f))< \delta'$, it follows that $\mu( \Psi(P) \vartriangle P) < \delta'|F|$ for every $P \in \cP^F$.

 We extend $\Psi$ to a map from $\Sigma(\cP^F) \to \Sigma(\cP^{F'})$ by requiring $\Psi(P \cup Q) = \Psi(P) \cup \Psi(Q)$ for any $P,Q \in \cP^F$. We claim that for any $Y \in \cP$, $\Psi(Y) \subset Y$. To see this, suppose $P\in \cP^F, P \subset Y$. As above, we represent $P$ by $P= \cap_{f\in \Lambda_{P}} f Y_f$. We must have $Y_{\sH^0}=Y$. Thus $\Psi(P) = \cap_{f \in \Lambda_{P}} \theta(f)Y_f$ implies, because $\theta(\sH^0)=\sH^0$, that $\Psi(P) \subset Y$. Because $P$ is arbitrary, $\Psi(Y) \subset Y$ as claimed.

The map $\Psi$ might not be a homomorphism from $\Sigma(\cP^F)$ to $\Sigma(\cP^{F'})$. To correct for this possibility, for any $Y \in \cP$, we enumerate the atoms of $\cP^F$ contained in $Y$ by $\{P^Y_1,\ldots,P^Y_{n(Y)}\}$. If $n(Y)=1$ then we define $\Psi'(Y)=\Psi'(P^Y_1) = Y$. Otherwise define $\Psi':\cP^F \to \cP^{F'}$ by 
\begin{displaymath}
\Psi'(P^Y_i) =\left\{ \begin{array}{ll}
\Psi(P^Y_1) & \textrm{ if $i=1$}\\
\Psi(P^Y_i) \setminus \cup_{j<i} \Psi(P^Y_j) & \textrm{ if $1 < i < n(Y)$}\\
Y \setminus \cup_{j<n(Y)} \Psi'(P^Y_j) & \textrm{ if $i=n(Y)$}.
\end{array}\right.\end{displaymath}
Because $\Psi(Y) \subset Y$, $\Psi'$ extends to a unique homomorphism from $\Sigma(\cP^F)$ to $\Sigma(\cP^{F'})$ which we also denote by $\Psi'$. Moreover, $\Psi'$ fixes $\cP$ pointwise.

Next, we estimate how far $\Psi'$ is from the identity map. We claim that for any $P \in \cP^F$,
$$\Psi'(P) \vartriangle P \subset Z:=\bigcup_{Q \in \cP^F} \Psi(Q) \vartriangle Q.$$
Since $\Psi'(P) \vartriangle P \subset ( \Psi'(P) \vartriangle \Psi(P)) \cup (\Psi(P) \vartriangle P)$, it suffices to show that $\Psi'(P) \vartriangle \Psi(P) \subset Z$. As above, we assume $P$ is of the form $P=P^Y_i$ for some $i,Y$. The claim is obvious if $i=1$. If $1<i<n(Y)$ then 
\begin{eqnarray*}
\Psi'(P^Y_i) \vartriangle \Psi(P_i^Y) &=& \bigcup_{j<i} \Psi(P_i^Y) \cap \Psi(P_j^Y)\\
 &\subset& \bigcup_{j<i} [\Psi(P_i^Y) \setminus P_i^Y] \cup [\Psi(P_j^Y) \setminus P_j^Y] \subset Z.
 \end{eqnarray*}
 The first inclusion above occurs because $\cP^F$ is a partition. If $i=n(Y)>1$ then
 \begin{eqnarray*}
\Psi'(P^Y_{n(Y)}) \vartriangle \Psi(P_{n(Y)}^Y) &=& \left[Y \setminus \bigcup_{j<n(Y)} \Psi(P_j^Y)\right] \vartriangle \Psi(P_{n(Y)}^Y)\\
 &=& \left[Y \setminus \bigcup_{j\le n(Y)} \Psi(P^Y_j)\right] \cup \left[\Psi(P_{n(Y)}^Y) \cap \bigcup_{j<n(Y)} \Psi(P_j^Y)\right] \subset Z.
 \end{eqnarray*}
This proves the claim. So if $R \in \Sigma(\cP^F)$ then 
$$\Psi'(R) \vartriangle R \subset (\Psi'(R) \vartriangle \Psi(R)) \cup (\Psi(R) \vartriangle R) \subset Z=\bigcup_{Q \in \cP^F} \Psi(Q) \vartriangle Q$$
which implies
\begin{eqnarray}\label{eqn:replace1}
\mu(\Psi'(R) \vartriangle R) \le \sum_{Q \in \cP^F} \mu(\Psi(Q) \vartriangle Q) < \delta' |F| |\cP^F|.
\end{eqnarray}
%Therefore, $\bigcup_{P \in \cP^F} \Psi'(P) \vartriangle \Psi(P)  \subset \bigcup_{P \in \cP^F} \Psi(P) \setminus P$ which implies
%$$\mu\left( \bigcup_{P \in \cP^F}\Psi'(P) \vartriangle \Psi(P) \right) \le \sum_{P \in \cP^F} \mu\left( \Psi'(P) \vartriangle \Psi(P) \right) \le  \delta' |F| |\cP^F|.$$
%Therefore,
%$$\mu\left( \bigcup_{P \in \cP^F} \Psi'(P) \vartriangle P \right) \le \sum_{P \in \cP^F}\mu\left( \Psi'(P) \vartriangle P \right) \le 2\delta' |F| |\cP^F|.$$
%In particular,
%$$\sum_{P \in \cP^F} |\mu(\Psi'(P)) - \mu(P)| \le \sum_{P \in \cP^F} \mu\left( \Psi'(P) \vartriangle P \right) \le 2\delta' |F||\cP^F|.$$

Also note that if $f \in F, P \in \cP$ then
\begin{eqnarray}
\mu(\theta(f)\cdot P \vartriangle \Psi'(f \cdot P) ) &\le& \mu(\theta(f) \cdot P \vartriangle f \cdot P) + \mu(f \cdot P \vartriangle \Psi'(f \cdot P) )\\
&<& \delta' + \delta' |F| |\cP^F| \le 2 \delta' |F| |\cP^F|.\label{eqn:replace}
\end{eqnarray}

Next we show that the map $\phi \mapsto \phi \circ \Psi'$ takes $\Hom(\pi,\sigma,\cP,F',\delta')$ into $\Hom(\pi,\sigma,\cP,F,\delta)$. So let $\phi \in  \Hom(\pi,\sigma,\cP,F',\delta')$. By (\ref{eqn:replace1}),
\begin{eqnarray*}
\sum_{P \in \cP^F} | |\phi \Psi'(P)|_d - \mu(P) | &\le& \sum_{P \in \cP^F} | |\phi \Psi'(P)|_d - \mu(\Psi'(P)) |  + |\mu(\Psi'(P)) - \mu(P)|\\
 &<& \delta' + \delta' |F| |\cP^F|^2 < \delta.
\end{eqnarray*}
This uses that $\phi \in \Hom(\pi,\sigma,\cP,F',\delta')$. Next we observe that for any $f\in F$,
\begin{eqnarray*}
&&\sum_{P \in \cP} | \sigma_f\cdot \phi \Psi'(P) \vartriangle \phi \Psi'(f \cdot P) |_d\\
&\le& \sum_{P \in \cP} | \sigma_f\cdot \phi \Psi'(P) \vartriangle \sigma_f\cdot \phi(P)  |_d + |\sigma_f \cdot \phi(P) \vartriangle \sigma_{\theta(f)} \cdot \phi(P)|_d\\
&& \quad \quad+ |\sigma_{\theta(f)} \cdot \phi(P) \vartriangle \phi(\theta(f) \cdot P)|_d + |\phi(\theta(f)\cdot P) \vartriangle \phi \Psi'(f \cdot P)|_d.
\end{eqnarray*}
Next we estimate each of the four terms above. The first term equals zero because $\Psi'(P)=P$ for each $P \in \cP$.
%\begin{eqnarray*}
%&& \sum_{P \in \cP} | \sigma_f\cdot \phi \Psi'(P) \vartriangle \sigma_f\cdot \phi(P)  |_d  \le \sum_{P \in \cP} |  \phi (\Psi'(P) \vartriangle P)  |_d \\
%&\le&  \sum_{P \in \cP}   | |\phi (\Psi'(P) \vartriangle P)  |_d -\mu(\Psi'(P) \vartriangle P)| + \mu(\Psi'(P) \vartriangle P)\\
%&\le& \delta' |\cP| + 2\delta' |F| n.
%\end{eqnarray*}
Because $\sigma$ is $(F\cup F', \delta')$-continuous,
$$\sum_{P \in \cP} |\sigma_f \cdot \phi(P) \vartriangle \sigma_{\theta(f)} \cdot \phi(P)|_d \le |\sigma_f \vartriangle \sigma_{\theta(f)}|_d < \delta' + \nu(f \vartriangle \theta(f)) < 2\delta'.$$
Because $\phi \in  \Hom(\pi,\sigma,\cP,F',\delta')$,
$$\sum_{P \in \cP}  |\sigma_{\theta(f)} \cdot \phi(P) \vartriangle \phi(\theta(f) \cdot P)|_d < \delta'.$$
Finally,
\begin{eqnarray*}
&&\sum_{P \in \cP}  |\phi(\theta(f)\cdot P) \vartriangle \phi \Psi'(f \cdot P)|_d = \sum_{P \in \cP} |  \phi (\theta(f)\cdot P \vartriangle \Psi'(f \cdot P))  |_d \\
&\le&  \sum_{P \in \cP}   | |  \phi (\theta(f)\cdot P \vartriangle \Psi'(f \cdot P))  |_d-\mu(\theta(f)\cdot P \vartriangle \Psi'(f \cdot P))| + \mu(\theta(f)\cdot P \vartriangle \Psi'(f \cdot P))\\
&<& \delta' |\cP| + 2\delta' |F| |\cP^F|^2.
\end{eqnarray*}
The last line above uses that $\phi \in \Hom(\pi,\cdot,\cP,F',\delta')$ (and therefore $\sum_{P \in \cP^{F'}} | |\phi(P)|_d - \mu(P)| \le \delta'$) and (\ref{eqn:replace}).

Putting this altogether we obtain:
\begin{eqnarray*}
&&\sum_{P \in \cP} | \sigma_f\cdot \phi \Psi'(P) \vartriangle \phi \Psi'(f \cdot P) |_d \le 2\delta' + \delta' + \delta' |\cP| + 2\delta' |F| |\cP^F|^2 < 10\delta'|F| |\cP^F|^2 < \delta.
\end{eqnarray*}
So $\phi \Psi' \in \Hom(\pi,\sigma,\cP,F,\delta)$. Because $\Psi'$ fixes $\cP$ pointwise, it fixes $\cQ$ pointwise. Therefore, 
$$| \Hom(\pi,\sigma,\cP,F',\delta')|_\cQ \le |\Hom(\pi,\sigma,\cP,F,\delta)|_\cQ.$$
Because $\P$ is asymptotically continuous, this implies
$$\lim_{j \to \beta} d_j^{-1} \log \| | \Hom(\pi,\cdot,\cP,F',\delta')|_\cQ \|_{p,\P_j} \le\lim_{j \to \beta} d_j^{-1} \log \| |\Hom(\pi,\cdot,\cP,F,\delta)|_\cQ\|_{p,\P_j}.$$
Now we take (in order) the infimum over $\delta'>0$, the infimum over $F' \subset_f \cF$, the infimum over $\delta>0$, the infimum over $F \subset_f \lb \sH \rb$ to obtain the lemma.
\end{proof}

The previous two lemmas imply Theorem \ref{thm:top-replace}.

\section{Measure entropy via pseudo-metrics}\label{sec:meas2}

The formulation of measure entropy in this section is closely aligned with topological entropy. We assume as given: two discrete pmp topological groupoids $\sG,\sH$ such that $\sG^0$ and $\sH^0$ are compact metrizable spaces, a class-bijective continuous factor $\pi:(\sG,\mu) \to (\sH,\nu)$, a sofic approximation $\P=\{\P_j\}_{j\in J}$ to $(\sH,\nu)$, a continuous pseudo-metric $\rho$ on $\sG^0$, a bias $\beta$ and $p\in [1,\infty]$. From this data, we will define the {\em sofic measure entropy of $(\pi,\rho)$ with respect to $(\P,p,\beta)$} and show that when $\rho$ is dynamically generating, $\sH$ is \'etale, $\nu$ is regular and $\P$ is asymptotically continuous then this entropy coincides with the definition of \S \ref{sec:meas1}.

\begin{defn}
Given a map $\sigma:\lb\sH\rb\to\lb d \rb$, finite sets $F \subset \lb \sH \rb_{top}$, $K \subset C(\sG^0)$ and $\delta>0$, we let $Orb_\mu(\pi,\sigma,F,K,\delta,\rho)$  be the set of all $d$-tuples $(x_1,\ldots, x_d) \in Orb_\nu(\pi,\sigma,F,\emptyset, \delta,\rho)$ (as defined in \S \ref{sec:top}) such that
%$$\rho_2( (fx_1,\ldots, fx_d), (x_{\sigma(f)1}, \ldots, x_{\sigma(f)d} )) <\delta.$$
$$\max_{k \in K} \left| \frac{1}{d}\sum_{i=1}^d k(x_i) - \int_{\sG^0} k ~d\mu\right| < \delta.$$
The main difference between $Orb_\nu(\pi,\sigma,F,K, \delta,\rho)$ and $Orb_\mu(\pi,\sigma,F,K, \delta,\rho)$ is that, in the first case $K \subset_f C(\sH^0)$ while in the second case $K \subset_f C(\sG^0)$. Define
\begin{eqnarray*}
h_{\P,\mu}(\pi,\rho,2) &:=& \sup_{\epsilon>0} \inf_{\delta>0} \inf_{F \subset_f \lb \sH \rb_{top}} \inf_{K \subset_f C(\sG^0)}  \lim_{j\to\beta} \frac{1}{d_j} \log \| N_\epsilon( Orb_\mu(\pi,\cdot, F,K, \delta,\rho), \rho_2) \|_{p,\P_j};\\
h_{\P,\mu}(\pi,\rho,\infty) &:=& \sup_{\epsilon>0} \inf_{\delta>0} \inf_{F \subset_f \lb \sH \rb_{top}} \inf_{K \subset_f C(\sG^0)}  \lim_{j\to\beta} \frac{1}{d_j} \log\|  N_\epsilon( Orb_\mu(\pi,\cdot, F, K, \delta,\rho), \rho_\infty) \|_{p,\P_j}.
\end{eqnarray*}
We are suppressing the choice of bias $\beta$ and parameter $p \in[1,\infty]$ from the notation.
\end{defn}
\begin{remark}
 As in the topological case, the order of the supremums, infimums and limits above is important with the exception that one can permute the three infimums without affecting the definition. There is a certain useful monotonicity phenomenon in the formulas above: the quantity
 $$\frac{1}{d_j} \log\| N_\epsilon( Orb_\mu(\pi,\cdot,F,K,\delta,\rho), \rho_2)\|_{p, \P_j}$$
 is monotone increasing in $\delta$ and monotone decreasing in $\epsilon,F,K$ (subsets are ordered by inclusion). Therefore, the infimums and the supremum can be replaced by the appropriate (directed) limits. In the sequel, we will use these facts without explicit reference. Similar statements hold true if $\rho_2$ is replaced with $\rho_\infty$ or $N_\epsilon$ is replaced with $N'_\epsilon$.
\end{remark}

\begin{lem}\label{lem:span-me}
If we replace $N_\epsilon(\cdot)$ in the definitions above with $N'_\epsilon(\cdot)$ then we obtain equivalent definitions. More precisely,
$$h_{\P,\mu}(\pi,\rho,2) := \sup_{\epsilon>0} \inf_{\delta>0} \inf_{F \subset_f \lb \sH \rb_{top}} \inf_{K \subset_f C(\sG^0)}  \lim_{j\to\beta} \frac{1}{d_j}  \log \| N'_\epsilon( Orb_\mu(\pi,\cdot, F,K, \delta,\rho), \rho_2) \|_{p,\P_j};$$
$$h_{\P,\mu}(\pi,\rho,\infty) := \sup_{\epsilon>0} \inf_{\delta>0}\inf_{F \subset_f \lb \sH \rb_{top}} \inf_{K \subset_f C(\sG^0)}  \lim_{j\to\beta} \frac{1}{d_j} \log \| N'_\epsilon( Orb_\mu(\pi,\cdot, F, K, \delta,\rho), \rho_\infty) \|_{p,\P_j}.$$
\end{lem}
\begin{proof}
This is immediate from Lemma \ref{lem:span-sep}.
\end{proof}

\begin{lem}
In general, $h_{\P,\mu}(\pi,\rho,2) = h_{\P,\mu}(\pi,\rho,\infty)$. 
\end{lem}

\begin{proof}
The proof is essentially the same as the proof of Lemma \ref{lem:2=infinity}.
\end{proof}

\begin{notation}
Because of the lemma above, we will write $h_{\P,\mu}(\pi,\rho)$ to denote either $h_{\P,\mu}(\pi,\rho,2)$ or  $h_{\P,\mu}(\pi,\rho,\infty)$.
\end{notation}

%\begin{thm}\label{thm:K-top-me}
%If $\rho_1,\rho_2$ are dynamically generating continuous pseudo-metrics on $\sG^0$ then $h_{\P,\mu}(\pi,\rho_1) = h_{\P,\mu}(\pi,\rho_2)$.
%\end{thm}

\begin{lem}\label{lem:K-top-me}
If $\rho_1,\rho_2$ are dynamically generating continuous pseudo-metrics on $\sG^0$ then $h_{\P,\mu}(\pi,\rho_1) = h_{\P,\mu}(\pi,\rho_2)$.
\end{lem}

\begin{proof}
The proof is essentially the same as the proof of Theorem \ref{thm:K-top}. %For example, note that
%$$Orb_\nu(\pi,\sigma,F,\emptyset,\delta,\rho^\phi) \supset Orb_\nu(\pi,\sigma,F,\emptyset,\delta_0,\rho)$$
%implies $Orb_\mu(\pi,\sigma,F,K,\delta,\rho^\phi) \supset Orb_\mu(\pi,\sigma,F,K,\delta_0,\rho)$ (for any finite $K \subset C(\sG^0))$.
\end{proof}
%Theorem \ref{thm:1=2} below implies $h_{\P,\mu}(\pi,\rho_1)$ does not depend on the choice of topological model for $\sG^0$ under suitable mild conditions, defined next.

The main result of this section is:
\begin{thm}\label{thm:1=2}
If $\rho$ is a dynamically generating continuous pseudo-metric on $\sG^0$, $\sH,\sG$ are \'etale, $\nu$ is regular and $\P$ is asymptotically continuous then $h_{\P,\mu}(\pi,\rho)= h_{\P,\mu}(\pi)$.
\end{thm}

We will need the next lemma which shows that good homomorphisms have to be close to $\sigma$ wherever this makes sense. 

\begin{lem}\label{lem:basic-formulas3}
Let $(\sH,\nu)$ be a discrete pmp groupoid, $\pi:(\sG,\mu) \to (\sH,\nu)$ a class-bijective pmp factor, $\cP$ a finite Borel partition of $\sG^0$, $F \subset_f \lb\sH\rb$ with $\sH^0 \in F$ and $R \subset \sH^0$ a Borel set. Suppose $R \in F$. Let $\sigma:\lb\sH\rb \to \lb d\rb$ be $(F, \delta)$-multiplicative. If $\phi \in \Hom(\pi,\sigma, \cP,F,\delta)$  then
$$| \phi(\pi^{-1}(R) ) \vartriangle \sigma(R) |_d < 3\delta.$$
Moreover, for any $f \in F$, if $fR \in F$ then
$$ | \sigma_f\cdot\phi( \pi^{-1}(R))  \vartriangle \phi(f\cdot \pi^{-1}(R)) |_d < 3\delta.$$
%In particular, if $\Hom(\pi,\sigma, \cP,F,\delta)$ is nonempty then $\sigma$ is $(\{R\},\delta)$-trace preserving.
\end{lem}

\begin{proof}
%We claim that
%$$|\sigma(R) \vartriangle \ra(\sigma(R))|_d \le 2\delta.$$
%By Lemma \ref{lem:basic-formulas2}, $|\sigma(R) \setminus \Delta^0_d| \le \delta$. If $(i,j) \in \sigma(R) \vartriangle \ra(\sigma(R))$ then either $(i,j) \in \sigma(R) \setminus \Delta^0_d$ 

Because $\phi$ is a homomorphism, $\phi(\sG^0)=\Delta^0_d$. By Lemma \ref{lem:basic-formulas2},
\begin{eqnarray*}
\delta &>& |\sigma_R \cdot \phi(\sG^0) \vartriangle \phi(R \cdot \sG^0)|_d\ge -\delta+ |(\sigma_R \cap \Delta^0_d)\cdot \phi(\sG^0) \vartriangle \phi(\pi^{-1}(R))|_d\\
 &=& -\delta+ |(\sigma_R \cap \Delta^0_d) \vartriangle \phi(\pi^{-1}(R))|_d \ge -2\delta +  |\sigma_R \vartriangle \phi(\pi^{-1}(R))|_d
 \end{eqnarray*}
This proves the first inequality. Suppose $f,fR \in F$. Then
\begin{eqnarray*}
\delta &>& |\sigma_{f  R} \cdot \phi(\sG^0) \vartriangle \phi( f  R \cdot \sG^0)|_d > |\sigma_f \sigma_R\cdot\phi(\sG^0) \vartriangle \phi( f\cdot\pi^{-1}(R) )|_d - \delta \\
&>& |\sigma_f \cdot \phi(R \cdot \sG^0) \vartriangle \phi( f\cdot\pi^{-1}(R) )|_d  -  2\delta= | \sigma_f\cdot\phi( \pi^{-1}(R))  \vartriangle \phi(f\cdot \pi^{-1}(R)) |_d - 2\delta.
\end{eqnarray*}
The first and third inequalities follow from $\phi \in \Hom(\pi,\sigma, \cP,F,\delta)$ while the second inequality uses the $(F,\delta)$-multiplicativity of $\sigma$.
\end{proof}

We will say that a partition $\cP$ of $\sG^0$ has {\em measure zero boundary} if $\mu(\partial P)=0$ for every $P \in \cP$.

\begin{proof}[Proof of Theorem \ref{thm:1=2}]
By Lemma \ref{lem:K-top-me}, we may assume $\rho$ is a metric on $\sG^0$. Let $\epsilon>0$ and $\cQ$ be a finite Borel partition of $\sG^0$ with measure zero boundary such that each atom of $\cQ$ has diameter $\le \epsilon$ with respect to $\rho$. Let $\cP\ge \cQ$ be a finite Borel partition with measure zero boundary. To simplify notation, we will identify $\Delta^0_d$ with $\{1,\ldots, d\}$ in the obvious way.

\noindent {\bf Claim 1}. Given $\delta>0$ and $F=F^{-1} \subset_f \lb \sH \rb_{top}$ with $\sH^0\in F$ and $\ra(f),\so(f) \in F$ for every $f\in F$, there exist $\delta',\eta>0$ and $K \subset_f C(\sG^0)$ such that for any $\sigma:\lb \sH \rb \to \lb d \rb$ which is $(F,\eta)$-multiplicative ,
$$|\Hom(\pi,\sigma, \cP, F, 2\delta)|_\cQ \ge N_{\epsilon}(Orb_\mu(\pi,\sigma, F,K, \delta',\rho), \rho_\infty).$$ 
\begin{proof}[Proof of Claim 1]
For $x\in \sG^0$, let $\cP(x)$ denote the atom of $\cP$ containing $x$. For $\delta'>0$, let 
$$B(\cP,\sqrt{\delta'}) := \{ x \in \sG^0:~ \exists x' \in \sG^0 \textrm{ s.t. } \cP(x) \ne \cP(x') \textrm{ and } \rho(x,x')<\sqrt{\delta'}\}.$$
Because $\cP$ has measure zero boundary and $F \subset_f \lb \sH\rb_{top}$ it follows that $\cP^F$ has measure zero boundary (this uses Lemma \ref{lem:top-inclusion}). Because $\mu$ is regular on $\sG^0$ there exists a $\delta',\eta>0$ and a finite set $K \subset C(\sG^0)$ such that for any $\sigma:\lb \sH \rb \to \lb d \rb$,
\begin{enumerate}
\item $\mu(B(\cP,\sqrt{\delta'})) < \delta$;
\item $\left(\delta'  + M \sqrt{15 \eta}\right)^2 < \delta \delta'$ where $M=\max_{x,y \in \sG^0} \rho(x,y)$ is the diameter of $\rho$;
\item if $x=(x_1,\ldots, x_d) \in (\sG^0)^d$ satisfies
$$\max_{k \in K} \left| \frac{1}{d}\sum_{i=1}^d k(x_i) - \int_{\sG^0} k ~d\mu\right| < \delta'$$
then \begin{enumerate}
\item $\sum_{P \in \cP^F} | d^{-1}\#\{ 1\le i \le d:~ x_i \in P\} - \mu(P)| < \delta$
\item $d^{-1}\sum_{f\in F}\#\{ 1\le i \le d:~ f^{-1}\cdot x_i \in B(\cP,\sqrt{\delta'}) \textrm{ or } x_{\sigma(f)^{-1}i} \in B(\cP,\sqrt{\delta'}) \}  \le \delta$.
\end{enumerate}
\end{enumerate}
Item (3a) above uses that $\cP^F$ has measure zero boundary. 

Let $\sigma:\lb\sH\rb \to \lb d\rb$ be $(F,\eta)$-multiplicative. For $y\in Orb_\mu(\pi,\sigma, F,K, \delta',\rho)$, let $\phi^y:\Sigma(\cP^F) \to \cB(\Delta^0_d)$ be the homomorphism $\phi^y(P) = \{ 1\le i \le d:~ y_i \in P\}$. We claim that $\phi^y \in \Hom(\pi,\sigma, \cP, F, 2\delta)$. Indeed, for each $P \in \cP^F$,
$$\sum_{P \in \cP^F} | d^{-1}|\phi^y(P)| - \mu(P)| < \delta$$
follows from the choice of $K,\delta'$. 

We need the following estimate which follows from Lemma \ref{lem:basic-formulas2}.
\begin{eqnarray*}
\rho_2(f^{-1}\cdot y, y \circ \sigma(f)^{-1}) &\le& \rho_2(f^{-1}\cdot y, y \circ \sigma(f^{-1})) + \rho_2(y \circ \sigma(f^{-1}), y \circ \sigma(f)^{-1})\\
 &\le& \delta'+ \left( \frac{1}{d} \sum_{i=1}^d \rho(y_{\sigma(f^{-1})i}, y_{\sigma(f)^{-1}i})^2 \right)^{1/2} \le \delta' + M (15\eta)^{1/2}.
 \end{eqnarray*}
 Suppose $1\le i \le d$, $f\in F$ and $\cP( f^{-1} \cdot y_i) \ne \cP( y_{\sigma(f)^{-1}i})$. If $f^{-1}\cdot y_i$ or $y_{\sigma(f)^{-1}i} \notin B(\cP,\sqrt{\delta'})$ then $\rho(f^{-1}\cdot y_i, y_{\sigma(f)^{-1}i}) \ge \sqrt{\delta'}$. So
\begin{eqnarray*}
(\delta' + M(15\eta)^{1/2} )^2 &\ge& \rho_2(f^{-1}\cdot y, y \circ \sigma(f)^{-1})^2\\
 &\ge & d^{-1} \delta' |\{ 1\le i \le d:~ f^{-1}\cdot y_i \textrm{ or } y_{\sigma(f)^{-1}i} \notin B(\cP,\sqrt{\delta'}), \cP( f^{-1}\cdot y_i) \ne \cP( y_{\sigma(f)^{-1}i})\}|\\
&\ge& d^{-1}\delta' |\{ 1\le i \le d:~ \cP( f^{-1}\cdot y_i) \ne \cP( y_{\sigma(f)^{-1}i})\}| - \delta' \delta.
\end{eqnarray*}
This implies
$$d^{-1} |\{ 1\le i \le d:~ \cP( f^{-1}\cdot y_i) \ne \cP( y_{\sigma(f)^{-1}i})\}| \le \frac{(\delta' + M(15\eta)^{1/2} )^2}{\delta'} + \delta < 2\delta.$$ 
Observe that if $i \in \sigma_f\cdot\phi^y(P) \vartriangle \phi^y( f\cdot P)$ (for some $P\in \cP$) then $\cP( f^{-1}\cdot y_i) \ne \cP( y_{\sigma(f)^{-1}i})$. So
\begin{eqnarray*}
d^{-1}\sum_{P \in \cP} |\sigma_f\cdot\phi^y(P) \vartriangle \phi^y( f\cdot P) | &\le&  d^{-1}|\{ 1\le i \le d:~ \cP( f^{-1}\cdot y_i) \ne \cP( y_{\sigma(f)^{-1}i})\}| \le 2\delta.
\end{eqnarray*} 
This implies $\phi^y \in \Hom(\pi,\sigma, \cP, F, 2\delta)$.

Next suppose $y,z \in Orb_\mu(\pi,\sigma, F,K,\delta',\rho)$ and $\rho_\infty(y,z) > \epsilon$. In other words, $\rho(y_i,z_i) > \epsilon$ for some $i$.  Because the diameter of each partition element of $\cQ$ is at most $\epsilon$,  $\phi^y$ restricted to $\Sigma(\cQ)$ is different from $\phi^z$ restricted to $\Sigma(\cQ)$. So the map $y \mapsto \phi^y$ takes any $(\rho_\infty,\epsilon)$-separated subset to a set of homomorphisms whose restrictions to $\Sigma(\cQ)$ are distinct. This proves $|\Hom(\pi,\sigma, \cP, F, 2\delta)|_\cQ \ge  N_{\epsilon}(Orb_\mu(\pi,\sigma, F,K, \delta',\rho), \rho_\infty)$ as claimed.
\end{proof}
Claim 1 and Theorem \ref{thm:top-replace} imply
$$h_{\P, \mu}(\pi,\cQ,\cP) \ge \inf_{\delta>0}\inf_{F \subset_f \lb \sH \rb_{top}} \inf_{K \subset_f C(\sG^0)}  \lim_{j\to\beta} \frac{1}{d_j} \log \| N_{\epsilon}( Orb_\mu(\pi,\cdot, F, K, \delta,\rho), \rho_\infty) \|_{p,\P_j}.$$
We now take the infimum over all $\cP$ with measure zero boundary, then the supremum over all $\cQ$ with measure zero boundary, then the supremum over all $\epsilon>0$ to obtain
$$h_{\P,\mu}(\pi)=h_{\P,\mu}(\pi,\cB_\partial(\sG^0)) \ge h_{\P,\mu}(\pi,\rho).$$
The equality above holds because $\cB_\partial(\sG^0)$ is $\pi$-generating by Lemma \ref{lem:boundary-zero}.

%h_{\P,\mu}(\pi,\rho)$. 

\noindent {\bf Claim 2}. Let $\cQ \subset \cB_\partial(\sG^0)$ be a finite partition and $\kappa>0$. By Lemma \ref{lem:5} there exists $\epsilon>0$ such that $h_{\P,\mu}(\pi, \cQ, \cP) \le h_{\P,\mu}^\epsilon(\pi,\cQ,\cP)+\kappa$ for all finite measurable partitions $\cP$ refining $\cQ$. Given $\delta>0$ and $K \subset_f C(\sG^0)$, there exist a finite partition $\cP \ge \cQ$ with $\cP \subset \cB_\partial(\sG^0)$ and $\eta,\delta'>0$ such that for any $F \subset_f \lb \sH \rb$ satisfying $F=F^{-1}$ and $\so(f),\ra(f), \sH^0 \in F ~(\forall f\in F)$ and any $(F,\delta')$-multiplicative $\sigma:\lb \sH \rb \to \lb d \rb$,
$$N_\epsilon(\Hom(\pi,\sigma, \cP, F, \delta'),\rho_\cQ) \le N_{\eta}(Orb_\mu(\pi,\sigma, F,K, \delta,\rho), \rho_\infty).$$ 
Moreover, $\eta$ depends only on $\cQ,\epsilon$, and $\cP,\delta'$ depend only on $\cQ,\epsilon,K,\delta$. 

\begin{proof}[Proof of Claim 2]

For $Q \subset \sG^0$ and $t>0$, let $N_t(Q)$ be the set of all $x \in \sG^0$ such that there exists $q \in Q$ with $\rho(q,x)\le t$. As shown in the proof of Lemma \ref{lem:boundary-zero}, $\partial N_t(Q) \cap \partial N_s(Q) = \emptyset$ if $t \ne s$. Since an uncountable sum of positive numbers is infinite, there exists some $\eta>0$ such that 
\begin{itemize}
\item $\mu(\partial N_\eta(\sG^0\setminus Q)) = 0$ for all $Q \in \cQ$ (i.e., $N_\eta(\sG^0\setminus Q) \in \cB_\partial(\sG^0)$),
\item $\mu(Q \cap N_\eta(\sG^0 \setminus Q) )\le \epsilon/5$ for all $Q \in \cQ$.
\end{itemize}
Let $Q^\eta = Q \cap N_\eta(\sG^0 \setminus Q) \in \cB_\partial(\sG^0)$. Let $\diam(\rho):=\max\{\rho(x,y):~x,y \in\sG^0\}$. Choose a finite partition $\cP \subset \cB_\partial(\sG^0)$ and $\delta'>0$ so that 
\begin{enumerate}
\item $\cQ \le \cP$ and every atom of $\cP$ has diameter at most $\delta/2$,
\item $Q^\eta \in \Sigma(\cP)$ for every $Q \in \cQ$,
\item $\sqrt{100\diam(\rho)^2 \delta'+ (\delta/2)^2} \le \delta$ and $\delta' < \epsilon/5$.
%\item for any homomorphism $\phi:\Sigma(\cP^F) \to \cB(\Delta^0_d)$ satisfying $ \sum_{P\in \cP^F} | |\phi(P)|d^{-1} - \mu(P) | < \delta'$ and 
\item for any point $x\in (\sG^0)^d$ such that 
$$\sum_{P \in \cP} \Big| |\{1\le i \le d:~ x_i \in P\}|d^{-1} - \mu(P)\Big| < \delta'$$
we have
$$\max_{k \in K} \left| \frac{1}{d}\sum_{i=1}^d k(x_i) - \int_{\sG^0} k ~d\mu\right| < \delta.$$
\end{enumerate}

For each $P \in \cP^F$, choose a basepoint $x_P \in P$. Given $\phi \in \Hom(\pi,\sigma, \cP, F, \delta')$, define $y^\phi \in (\sG^0)^d$ by $y^\phi_i = x_P$ if $i \in \phi(P)$. We claim that $y^\phi \in Orb_\mu(\pi,\sigma, F,K, \delta,\rho)$. Indeed the choice of $\cP,\delta'$ above implies
$$\max_{k \in K} \left| \frac{1}{d}\sum_{i=1}^d k(y^\phi_i) - \int_{\sG^0} k ~d\mu\right| < \delta.$$
Because $\phi \in \Hom(\pi,\sigma, \cP, F, \delta')$,
\begin{enumerate}
\item $ \sum_{P \in \cP} |\sigma_f\cdot\phi(P) \vartriangle \phi( f\cdot P) |d^{-1} < \delta' \quad \forall f\in F$;
\item $ \sum_{P\in \cP^F} | |\phi(P)|d^{-1} - \mu(P) | < \delta'$.
\end{enumerate}
Fix $f\in F$. For $x \in \sG^0$, let $\cP(x)$ be the element of $\cP$ containing $x$. Note that if, for some $i$, $f \cdot y^\phi_i$ and $\sigma(f) i$ are well-defined but $\cP(f \cdot y^\phi_i) = P \ne \cP(y^\phi_{\sigma(f) i})$ then $\sigma(f) i \in \sigma(f)\cdot \phi(f^{-1} \cdot P) \setminus \phi(P)$.  So
\begin{eqnarray*}
&&|\{ i \in \Delta^0_d:~ y^\phi_i \in \pi^{-1}(\so(f)), i \in \so(\sigma(f)), \cP(f\cdot y^\phi_i) \ne \cP(y^\phi_{\sigma(f)i})\}|_d\\
 &\le& \sum_{P\in \cP} |\sigma(f)\cdot\phi(f^{-1}\cdot P)\setminus\phi(P)|_d\\
&\le& \sum_{P\in \cP} |\phi(f^{-1}\cdot P)\setminus\sigma(f)^{-1}\cdot \phi(P)|_d\\
&\le& \sum_{P\in \cP} |\phi(f^{-1}\cdot P)\vartriangle \sigma(f)^{-1}\cdot \phi(P)|_d\\
&\le& \sum_{P\in \cP} |\phi(f^{-1}\cdot P)\vartriangle \sigma(f^{-1})\cdot \phi(P)|_d + | \sigma(f^{-1}) \cdot \phi(P) \vartriangle \sigma(f)^{-1}\cdot\phi(P)|_d \\
&\le& |\sigma(f^{-1}) \vartriangle\sigma(f)^{-1}|_d + \sum_{P\in \cP} |\phi(f^{-1}\cdot P)\vartriangle \sigma(f^{-1})\cdot \phi(P)|_d < 15\delta' + \delta' = 16\delta'.
\end{eqnarray*}
The last inequality holds by Lemma \ref{lem:basic-formulas2}. By Lemmas \ref{lem:basic-formulas2} and \ref{lem:basic-formulas3},
\begin{eqnarray*}
&&|\{ i \in \Delta^0_d:~ y^\phi_i \in \pi^{-1}(\so(f)), i \notin \so(\sigma(f))  \}|_d + |\{ i \in \Delta^0_d:~ y^\phi_i \notin \pi^{-1}(\so(f)), i \in \so(\sigma(f))  \}|_d\\ 
 &=& |\phi(\pi^{-1}(\so(f))) \vartriangle \so(\sigma(f))|_d \le  |\phi(\pi^{-1}(\so(f))) \vartriangle \sigma(\so(f))|_d + |\sigma(\so(f)) \vartriangle \so(\sigma(f))|_d \\
&\le& 3\delta' + 10\delta' = 13\delta'.
\end{eqnarray*}
So,
\begin{eqnarray*}
\rho_2( f y^\phi, y^\phi \circ \sigma_f)^2 &\le& d^{-1} \diam(\rho)^2|\{ 1\le i \le d:~ \cP(f y^\phi_i) \ne \cP(y^\phi_{\sigma(f)i})|\\
&&+d^{-1} \diam(\rho)^2|\{ i \in \Delta^0_d:~ y^\phi_i \in \pi^{-1}(\so(f)), i \notin \so(\sigma(f))  \}|\\
&&+d^{-1} \diam(\rho)^2|\{ i \in \Delta^0_d:~ y^\phi_i \notin \pi^{-1}(\so(f)), i \in \so(\sigma(f))  \}|\\
&& + d^{-1} (\delta/2)^2 |\{ 1\le i \le d:~ \cP(f y^\phi_i) = \cP(y^\phi_{\sigma(f)i})| \\
&<&  100\diam(\rho)^2 \delta'+ (\delta/2)^2.
\end{eqnarray*}
So 
$$\rho_2( f y^\phi, y^\phi \circ \sigma_f) < \sqrt{100\diam(\rho)^2 \delta'+ (\delta/2)^2} \le \delta.$$
This shows that $y^\phi \in  Orb_\mu(\pi,\sigma, F,K, \delta,\rho)$ as claimed.

We claim that if $\phi,\psi \in \Hom(\pi,\sigma, \cP, F, \delta')$ and $\rho_\cQ(\phi,\psi) > \epsilon$ then $\rho_\infty(y^\phi,y^\psi)> \eta$. Indeed, there exists $Q \in \cQ$ such that $|\phi(Q) \vartriangle \psi(Q)|_d >\epsilon$. Because $\phi,\psi \in \Hom(\pi,\sigma, \cP, F, \delta')$ and $Q^\eta \in \Sigma(\cP)$,
\begin{eqnarray*}
\epsilon &<& |\phi(Q) \vartriangle \psi(Q)|_d\\
 &=& |\phi(Q\setminus Q^\eta) \setminus \psi(Q)|_d + |\psi(Q\setminus Q^\eta) \setminus \phi(Q)|_d +|\phi(Q^\eta) \setminus \psi(Q)|_d + |\psi(Q^\eta)\setminus \phi(Q)|_d\\
 &\le&  |\phi(Q\setminus Q^\eta) \setminus \psi(Q)|_d + |\psi(Q\setminus Q^\eta) \setminus \phi(Q)|_d  +  2\delta' + 2\mu(Q^\eta).
 \end{eqnarray*}
 Since $ 2\delta' + 2\mu(Q^\eta) < 4\epsilon/5$, we obtain that there exists $i \in (\phi(Q\setminus Q^\eta) \setminus \psi(Q)) \cup (\psi(Q\setminus Q^\eta) \setminus \phi(Q))$. Therefore, $\rho(y^\phi_i, y^\psi_i) > \eta$ which implies the claim. 

So the map $\phi \mapsto y^\phi$ takes $(\rho_\cQ,\epsilon)$-separated subsets of $\Hom(\pi,\sigma, \cP, F, \delta')$ to $(\rho_\infty, \eta)$-separated subsets of $Orb_\mu(\pi,\sigma, F,K, \delta,\rho)$. This proves Claim 2.
  \end{proof}
Let us now assume the hypotheses of Claim 2. By choice of $\epsilon$,  $h_{\P,\mu}(\pi,\cQ,\cB_\partial(\sG^0)) -\kappa\le h^\epsilon_{\P,\mu}(\pi,\cQ,\cB_\partial(\sG^0))$. So
\begin{eqnarray*}
h_{\P,\mu}(\pi,\cQ,\cB_\partial(\sG^0)) -\kappa&\le& h^\epsilon_{\P,\mu}(\pi,\cQ,\cB_\partial(\sG^0))\le h^\epsilon_{\P,\mu}(\pi,\cQ,\cP,F,\delta')\\
&=& \lim_{j\to\beta}d_j^{-1}\log \| N_\epsilon(  Hom_\mu(\pi,\cdot, \cP,F, \delta'), \rho_\cQ ) \|_{p,\P_j}\\
&\le& \lim_{j\to\beta}d_j^{-1}\log \| N_\eta(  Orb_\mu(\pi,\cdot, F,K, \delta,\rho), \rho_\infty ) \|_{p,\P_j}.
\end{eqnarray*}
We can now take the infimum over $F,K,\delta$ and then the supremum over $\eta$  to obtain 
$$h_{\P,\mu}(\pi,\cQ,\cB_\partial(\sG^0)) -\kappa  \le h_{\P,\mu}(\pi,\rho).$$
Because $\kappa>0$ is arbitrary and  $\cB_\partial(\sG^0)$ is $\pi$-generating by Lemma \ref{lem:boundary-zero},
\begin{eqnarray*}
h_{\P,\mu}(\pi) = h_{\P,\mu}(\pi,\cB_\partial(\sG^0)) &\le& h_{\P,\mu}(\pi,\rho).
\end{eqnarray*}
As we have already obtained the opposite inequality, this proves the theorem.
   \end{proof}

\section{The variational principle}\label{sec:variational}

%The variational principle that we prove here contains a new technical condition. To be precise, for each $h\in \lb \sH \rb_{top}$, let $Z_h=\{x \in \sH^0:~ h \cap \so^{-1}(x) = id_x\}$. We say that a locally compact discrete pmp groupoid $(\sH,\nu)$ is {\em good} if there is a countable set $H'' \subset \lb \sH \rb_{top}$ such that whenever $\lambda$ is a Borel probability measure on $\sH^0$ with $\lambda(Z_h) \ge \nu(Z_h)~\forall h \in H''$ then $\lambda=\nu$. Note that if $\sH^0$ is a singleton then this condition is automatic. %This condition also holds if $\sH^0$ is totally disconnected and $\sH$ is ergodic.

\begin{thm}\label{thm:var-principle}
Let $(\sH,\nu)$ be a pmp separable \'etale topological discrete groupoid, $\sG$ be a separable \'etale topological discrete groupoid, $\pi:\sG \to \sH$ a continuous class-bijective factor, and $\P=\{\P_j\}_{j \in J}$ an asymptotically continuous sofic approximation to $(\sH,\nu)$ (definition \ref{defn:ac}). We assume both $\sH^0$ and $\sG^0$ are compact and metrizable and $\nu$ is regular. Then for any $p\in [1,\infty]$ and bias $\beta \ne -$,
$$h_{\P}(\pi) = \sup_\mu h_{\P,\mu}(\pi)$$
where the supremum is over all measures $\mu$ on $\sG$ such that $\pi_*\mu = \nu$ and $(\sG,\mu)$ is probability-measure-preserving.
\end{thm}

Before proving this, we need a few lemmas. The first is a generalization of the Feldman-Moore Theorem \cite{FM77}.

\begin{lem}\label{lem:FM}
Let $\sH$ be a discrete measurable groupoid. Then there exists a countable subgroup $H < [\sH]$ such that for every $g \in \sH$ there exists $h \in H$ with $g\in h$.
\end{lem}

\begin{proof}
%Let $\sH$ be the discrete measurable groupoid with $\sH^0=\sH^0$ so that for every $(x,y) \in E_\sH$ there is a unique morphism $h \in \sH$ with $hx=y$. In other words, $\sH$ is obtained from $\sH$ by removing all isotropy groups. By the Feldman-Moore theorem \cite{FM77}[Theorem 1] there is a countable subgroup $H' < [\sH]$ such that for every $(x,y) \in E_{\sH}$ there is an $h\in H'$ with $hx=y$. Let $\pi:\sH \to \sH$ be the obvious groupoid morphism. Then $\pi$ induces a surjective homomorphism $\pi_*$ from $[\sH]$ onto $[\sH]$. So there exists a countable subgroup $H'' < [\sH]$ such that $\pi(H'') = H'$. Thus for every $(x,y) \in E_\sH$ there is an $h \in H''$ with $hx=y$.

Let $E_\sH$ be the equivalence relation on $\sH^0$ given by $(x,y) \in E_\sH  \Leftrightarrow \exists g\in \sH$ such that $g\cdot x= y$. Note that $E_\sH$ is the image of $\sH^1$ under the map $f \mapsto (\so(f),\ra(f))$. Because $\sH$ is discrete, this map is at most countable-to-1. It follows from the Lusin-Novikov Theorem (see \cite[Theorem 18.10]{Ke95}) that $E_{\sH}$ is a Borel subset of $\sH^0\times \sH^0$.

It follows from \cite[Theorem 1]{FM77} that there is a countable subgroup $H' < [E_{\sH}]$ such that for every $(x,y) \in E_\sH$ there is an $h \in H'$ with $h\cdot x=y$. By the Lusin-Novikov Theorem again, for every $h' \in H'$ there exists an element $h'' \in [\sH]$ such that $h''$ maps to $h'$ under the map from $\sH$ to $\sH^0\times \sH^0$ given by $f \mapsto (\so(f),\ra(f))$. Let $H'' < [\sH]$ be the countable group generated by the elements $h''$ for $h' \in H'$.

Let $\sK =\{g\in \sH:~ \so(g)=\ra(g)\}$. By Kuratowski \cite[\S39, III, Corollary 5]{Ku33}, there is a countable Borel partition $\{P_i\}_{i\in I}$ of $\sK$ such that for each $i$, $\so|_{P_i}$ is injective. For each $i$, define $B_i = P_i \cup (\sH^0 \setminus \so(P_i))$. Note $B_i \in [\sH]$. We claim that the group $H$ generated by $H''$ and $\{B_i\}_{i\in I}$ satisfies the lemma. So let $g\in \sH$. If $g \in \sK$ then $g\in P_i$ for some $i$ and so $g \in B_i \in H$. Suppose $g\notin \sK$. Let $x=\so(g), y=\ra(g)$ so that $g\cdot x=y$. Then there exist $f \in \sH$ and $h \in H''$ with $f\in h$, $f\cdot x=y$. Observe that $g = f (f^{-1}g)$ and $f^{-1}g \in \sK$. So there is $B_j \in H$ with $f^{-1}g \in B_j$. So $g \in h B_j \in H$. Because $g$ is arbitrary, this proves the lemma.
\end{proof}

Next we show that it suffices to consider measures $\mu$ on $\sG^0$ that are $\lb\sH\rb_{top}$-invariant. To be precise:
\begin{prop}\label{prop:pmp}
Let $(\sH,\nu), \sG$ be as in Theorem \ref{thm:var-principle}. Let $\mu$ be a Borel probability measure on $\sG^0$ and suppose that $\mu$ is $\lb \sH \rb_{top}$-invariant in the sense that $\mu(k \circ f) = \mu(k \circ \ra(f))$ for every continuous function $k \in C(\sG^0)$ and $f \in \lb \sH \rb_{top}$ (where, for example, $\mu(k \circ f) := \int_{\pi^{-1}(\so(f))} k(f\cdot x)~d\mu(x))$. Suppose as well that $\pi_*\mu=\nu$. Then $(\sG,\mu)$ is probability-measure-preserving.
\end{prop}

\begin{proof}
By a standard argument, it suffices to show that $\mu$ is $\lb \sG \rb$-invariant. We first show that $\mu$ is $\lb \sH \rb$-invariant. By Lemma \ref{lem:top-dense}, $\lb \sH \rb_{top}$ is dense in $\lb \sH \rb$. In particular, if $f \in \lb \sH\rb$, $\epsilon>0$ then there exists $f' \in \lb \sH \rb_{top}$ with $\nu(f \vartriangle f')<\epsilon$. So if $k \in C(\sG^0)$ and 
$$R=\{x \in \pi^{-1}(\so(f) \cap \so(f')):~f\cdot x \ne f' \cdot x\} \cup \pi^{-1}(\so(f) \vartriangle \so(f'))$$
then
$$|\mu(k\circ f) - \mu(k \circ f')| \le  2\|k \|_\infty \mu(R) \le 2\epsilon \|k\|_\infty.$$
The last inequality occurs because $\pi_*\mu=\nu$. Similarly, $|\mu(k\circ \ra(f)) - \mu(k \circ \ra(f'))| \le 2\epsilon \|k\|_\infty$ which implies (by $\lb \sH\rb_{top}$-invariance) that $|\mu(k \circ f) - \mu(k \circ \ra(f))| \le 4\epsilon \|k\|_\infty$. Since $\epsilon$ is arbitrary, $\mu(k \circ f) = \mu(k\circ \ra(f))$. Since $f,k$ are arbitrary, $\mu$ is $\lb\sH\rb$-invariant.

%Now let $\psi \in \lb \sG \rb$. We will show $\psi_*\mu=\mu$. Let $H < [\sH]$ be as in Lemma \ref{lem:FM}. Recall that $\pi^{-1}:[\sH] \to [\sG]$ is a homomorphism.  Note
% $$\psi = \bigcup_{h \in H} \psi \cap \pi^{-1}(h).$$
%Because $\mu$ is $\pi^{-1}(h)$-invariant for each $h \in H$, it follows that $\mu$ is $(\psi \cap \pi^{-1}(h))$-invariant and therefore $\psi$-invariant. Because $\psi$ is arbitrary, $\mu$ is $\lb \sG \rb$-invariant which implies the lemma.

Now let $\psi \in \lb \sG \rb$. We will show $\psi_*\mu=\mu$. Let $H<[\sH]$ be as in Lemma \ref{lem:FM}. Let us enumerate $H$ by $H=\{h_i\}_{i=1}^\infty$. Let 
$$P_i=\{x \in \sG^0:~ \psi \cap \so^{-1}(x) = \pi^{-1}(h_i) \cap \so^{-1}(x) \textrm{ and  $i$ is minimal with this property}\}.$$
Because $\pi$ is class-bijective, $\{P_i\}_{i=1}^\infty$ is a Borel partition of $\sG^0$. We have shown that $\mu$ is $\pi^{-1}(H)$-invariant and therefore, each $\pi^{-1}(h_i)$ restricted to $P_i$ preserves $\mu$. Since $\psi$ is the disjoint union of $\pi^{-1}(h_i) \cdot P_i$, this shows that $\psi$ is measure-preserving. Because $\psi$ is arbitrary, $\mu$ is $\lb\sG\rb$-invariant which implies the lemma.

\end{proof}

\begin{defn}
Let $M(\sG^0)$ denote the space of Borel probability measures on $\sG^0$ with the weak* topology. To be precise, this is the weakest topology with the property that for every continuous function $k \in C(\sG^0)$, the map $\mu \in M(\sG^0) \mapsto \int k~d\mu$ is continuous.
\end{defn}
In order to show that the measures we obtain in the proof of Theorem  \ref{thm:var-principle} are $\lb \sH \rb_{top}$-invariant we need the following continuity result:

\begin{prop}\label{prop:continuity}
Suppose that $\Omega$ is a directed set and $\omega \in \Omega \mapsto \mu_\omega \in M(\sG^0)$ is a map such that $\lim_{\omega \to \Omega} \mu_\omega=\mu_\infty$. Suppose as well that $\pi_*\mu_\infty=\nu$. Then for any $k \in C(\sG^0)$ and $f \in \lb \sH \rb_{top}$, 
$$\lim_{\omega \to \Omega} \mu_\omega(k \circ f) = \mu_\infty(k \circ f).$$ 
\end{prop}

\begin{proof}
Let $K$ be the function on $\sG^0$ defined by $K(f\cdot x)=k(f\cdot x)$ if $x\in \pi^{-1}(\so(f))$ and $K(x)=0$ otherwise. Observe that $K$ is continuous at every $x \notin \pi^{-1}(\partial \so(f))$. Because $f \in \lb \sH\rb_{top}$ and $\pi_*\mu_\infty=\nu$ we have
$$\mu_\infty( \pi^{-1}(\partial \so(f))  )= \nu( \partial \so(f)) = 0.$$
Thus the set of discontinuity for $K$ has measure zero. The proposition now follows from a standard result in probability theory \cite[Theorem 2.7]{Bi99} sometimes called `the portmanteau theorem'. 
\end{proof}

\begin{proof}[Proof of Theorem \ref{thm:var-principle}]
Let $\rho$ be a continuous metric on $\sG^0$. Theorem \ref{thm:1=2} implies $h_\P(\pi) \ge \sup_\mu h_{\P, \mu}(\pi,\rho)=\sup_\mu h_{\P, \mu}(\pi)$. Without loss of generality, we may assume $h_\P(\pi) > -\infty$.  %We may also assume that the bias $\beta$ is a nonprincipal ultrafilter on $J$. This is because, if the result is true in this case for every nonprincipal ultrafilter, then it is true 

Let $\kappa>0$. Then there exists $\epsilon>0$ such that 
$$h^\epsilon_\P(\pi, \rho,2) \ge h_\P(\pi,\rho,2) - \kappa$$
where 
$$h_\P^\epsilon(\pi, \rho,2) := \inf_{\delta>0} \inf_{F \subset_f \lb \sH \rb_{top}}\inf_{K \subset_f C(\sH^0)} \lim_{j\to\beta} \frac{1}{d_j} \log \|  N_\epsilon( Orb_\nu(\pi,\cdot, F, K,\delta,\rho), \rho_2) \|_{p,\P_j}.$$
Let $\Omega=\{ (F,L,\delta):~  \sH^0 \in F \subset_f \lb \sH \rb_{top}, L \subset_f C(\sG^0), \delta>0\}$. We consider $\Omega$ as a directed set by declaring $(F,L,\delta) \le (F',L',\delta')$ if $F' \supset F, L' \supset L, \delta' \le \delta$. Given $\omega \in \Omega$, we write $\omega=(F_\omega,L_\omega,\delta_\omega)$ and we set $K_\omega = \{k \in C(\sH^0):~ k \circ \pi \in L_\omega\}$.  Let $M(\sG^0)$ denote the space of Borel probability measures on $\sG^0$.

\noindent {\bf Claim 1}. There exists a directed net $\omega \in \Omega \mapsto \mu_\omega \in M(\sG^0)$ such that 
\begin{enumerate}
\item $h_{\P, \mu_{\omega}}^\epsilon(\pi,\rho, \omega,2) \ge h_\P^\epsilon(\pi, \rho,2)$ where
\begin{eqnarray*}
h_{\P, \mu_\omega}^\epsilon(\pi,\rho, \omega,2) &:=& h_{\P, \mu_\omega}^\epsilon(\pi,\rho,F_\omega,L_\omega,\delta_\omega,2)\\
&:=&\lim_{j\to\beta} \frac{1}{d_j} \log \| N_\epsilon( Orb_{\mu_\omega}(\pi, \cdot, F_\omega,L_\omega, \delta_\omega,\rho), \rho_2) \|_{p,\P_j}.
\end{eqnarray*}

\item $\lim_{\omega \to \Omega} |\mu_\omega( k \circ f) - \mu_\omega(k \circ \ra(f))| = 0, \quad \forall f\in \lb\sH\rb_{top}, k \in C(\sG^0)$.
%\item For $h \in \sH^0$ and $\eta>0$, let $R_\eta(h) = \{x\in \sG^0:~ \rho(x,hx) \le \eta\}$. Then for any $k \in C(\sG^0)$ satisfying $0\le k \le 1$ and $k(x)=1$ for all $x\in R_\eta(h)$, we have $$\liminf_{\omega \to \Omega}  \pi_*\mu_{\omega}(k) \ge \nu(h \cap \sH^0)|.$$

\item $\lim_{\omega \to \Omega} |\mu_\omega(k\circ\pi) - \nu(k)| = 0, \quad \forall  k \in C(\sH^0)$.
\end{enumerate}

\begin{proof}[Proof of Theorem \ref{thm:var-principle} given Claim 1]
 Let $\mu$ be a weak* accumulation point of $\{\mu_\omega:~\omega \in \Omega\}$. By (3)  $\pi_*\mu=\nu$. By (2) and Proposition \ref{prop:continuity}, $\mu(k \circ f) = \mu(k \circ \ra(f))$ for every $f\in \lb\sH\rb_{top}, k \in C(\sG^0)$. By Proposition \ref{prop:pmp}, $(\sG,\mu)$ is probability-measure-preserving. 
 
Let $F \subset_f \lb \sH \rb_{top}, L \subset_f C(\sG^0)$ and $\delta>0$. Choose $\omega \in \Omega$ to satisfy 
\begin{enumerate}
\item $|\mu(k) - \mu_\omega(k)| \le \delta/2 ~\forall k \in L$;
\item $F \subset F_\omega$, $L \subset L_\omega$, $\delta_\omega \le \delta/2.$
\end{enumerate}
Then for any $\sigma:\lb \sH \rb \to \lb d \rb$, $x \in Orb_{\mu_\omega}(\pi,\sigma, F_\omega, L_\omega, \delta_\omega,\rho)$ and $k\in L$,
\begin{eqnarray*}
\left| \frac{1}{d} \sum_{i=1}^d k(x_i) - \mu(k) \right| &\le& \left| \frac{1}{d} \sum_{i=1}^d k(x_i) - \mu_\omega(k) \right| + \left|\mu_\omega(k)- \mu(k) \right| < \delta.
\end{eqnarray*}
Therefore,
$$Orb_{\mu_\omega}(\pi,\sigma, F_\omega,L_\omega, \delta_\omega,\rho) \subset Orb_{\mu}(\pi,\sigma,F,L,\delta,\rho).$$
So 
$$h^\epsilon_{\P,\mu}(\pi,\rho,F,L, \delta,2) \ge h^\epsilon_{\P, \mu_\omega}(\pi,\rho, F_\omega, L_\omega, \delta_\omega,2) \ge h^\epsilon_\P(\pi, \rho,2).$$
By taking the infimum over $F, L, \delta$ we obtain
$$h_{\P, \mu}(\pi,\rho,2) \ge h_{\P,\mu}^\epsilon(\pi,\rho,2) \ge h^\epsilon_\P(\pi,\rho,2) \ge h_\P(\pi,\rho,2) - \kappa.$$
Because $\kappa>0$ is arbitrary, this implies the Theorem.
\end{proof}

It remains to prove Claim 1. For $L \subset_f C(\sG^0)$ and $\delta>0$, let $M(L,\delta)$ be the set of all $\mu \in M(\sG^0)$ such that $|\mu(k \circ \pi) - \nu(k)|\le \delta$ for all $k \in C(\sH^0)$ with $k\circ \pi \in L$. For $F \subset_f \lb \sH\rb_{top}$ with $\sH^0 \in F$, let $D(F,L,\delta) \subset M(L,\delta)$ be a finite set such that for every $\lambda \in M(L,\delta)$ there exists a $\mu \in D(F,L,\delta)$ such that
$$| \mu(k \circ f) - \lambda(k \circ f)| < \delta\quad \forall f\in F, k\in L.$$

For every $x \in (\sG^0)^d$ let $m_x \in M(\sG^0)$ be the measure $m_x = d^{-1} \sum_{i=1}^d \delta_{x_i}$ where $\delta_{x_i}$ is the Dirac measure concentrated on $x_i$. For every $\omega \in \Omega$, choose a Borel map $x \in \{y \in (\sG^0)^d:~m_y \in M(L_\omega,\delta_\omega)\} \mapsto \mu_{x,\omega} \in D(\omega)$ satisfying
$$| \mu_{x,\omega}(k\circ f) - m_x(k\circ f) | < \delta_\omega\quad \forall f\in F_\omega, k \in L_\omega.$$

For every $\sigma:\lb \sH \rb \to \lb d \rb$, $F \subset_f \lb \sH \rb_{top}, K \subset_f C(\sH^0)$ and $\delta>0$, choose a maximum $(\rho_2,\epsilon)$-separated subset $Q(\sigma,F,K,\delta) \subset Orb_\nu(\pi,\sigma,F,K,\delta,\rho)$ so that for every $d_j$, the map $\sigma \in \Map(\lb \sH \rb, \lb d_j \rb) \mapsto Q(\sigma,F,K,\delta)$ is Borel.

By the Pigeonhole Principle for every $\omega \in \Omega$ and $j\in J$ there exists $\mu_{j,\omega} \in D(\omega)$ such that
\begin{eqnarray*}
\| \#\{x \in Q(\cdot,F_\omega,K_\omega,\delta_\omega):~ \mu_{j,\omega} = \mu_{x,\omega}\} \|_{p,\P_j}  &\ge& \frac{\| \#Q(\cdot,F_\omega,K_\omega,\delta_\omega) \|_{p,\P_j} }{ |D(\omega)|}\\
&=& \frac{ \| N_\epsilon(Orb_\nu(\pi,\cdot,F_\omega,K_\omega,\delta_\omega,\rho), \rho_2)\|_{p,\P_j}}{ |D(\omega)|}.
\end{eqnarray*}

Next we choose $\mu_\omega \in  D(\omega)$ so that if $J'=\{ j \in J:~ \mu_{j,\omega}=\mu_{\omega} \}$ then either $J' \in \beta$ (if $\beta$ is an ultrafilter on $J$) or $J'$ is cofinal. In the case $\beta=+$ we also require that
\begin{eqnarray*}
&&\limsup_{j\in J} d_j^{-1}\log \| N_\epsilon( Orb_{\mu_{j,\omega}}(\pi,\cdot, F_\omega,L_\omega,\delta_\omega,\rho), \rho_2)\|_{p,\P_j}\\
&&=\limsup_{j\in J'} d_j^{-1}\log \| N_\epsilon( Orb_{\mu_{j,\omega}}(\pi,\cdot, F_\omega,L_\omega,\delta_\omega,\rho), \rho_2)\|_{p,\P_j}.
\end{eqnarray*}
%Likewise if $\beta=-$ we require
%\begin{eqnarray*}
%&&\liminf_{j\in J} d_j^{-1}\log \| N_\epsilon( Orb_{\mu_{j,\omega}}(\pi,\cdot, F_\omega,L_\omega,\delta_\omega,\rho), \rho_2)\|_{p,\P_j}\\
%&&=\liminf_{j\in J'} d_j^{-1}\log \| N_\epsilon( Orb_{\mu_{j,\omega}}(\pi,\cdot, F_\omega,L_\omega,\delta_\omega,\rho), \rho_2)\|_{p,\P_j}.
%\end{eqnarray*}
Such a choice is possible because $D(\omega)$ is finite.

%We now choose $\mu_\omega \in  D(\omega)$ as follows. If  $\beta$ is an ultrafilter on $J$ we choose $\omega$ so that require $\{ j \in J:~ \mu_{j,\omega}=\mu_{\omega} \} \in \beta$. Otherwise we require that $\{ j \in J:~ \mu_{j,\omega}=\mu_{\omega} \}$ is cofinal. In the case $\beta=+$ 

%By the Pigeonhole Principle again, there exists $\mu_\omega \in D(\omega)$ such that for every $j\in J$ there exists $j' \ge j$ with $\mu_{j',\omega}=\mu_{\omega}$ and, in case $\beta$ is an ultrafilter on $J$, $\{ j \in J:~ \mu_{j,\omega}=\mu_{\omega} \} \in \beta$.

We will show that the measures $\{\mu_\omega:~\omega \in \Omega\}$ satisfy Claim 1. First, note that for any $j \in J$, $\omega \in \Omega$, if $x \in Orb_\nu(\pi,\sigma,F_\omega,K_\omega,\delta_\omega,\rho)$ and $\mu_{x,\omega}=\mu_{j,\omega}$ then
$$| m_x(k) - \mu_{j,\omega}(k)| < \delta_\omega\quad \forall k \in L_\omega$$
implies $x \in Orb_{\mu_{j,\omega}}(\pi,\sigma,F_\omega,L_\omega,\delta_\omega,\rho)$. Therefore
$$N_\epsilon( Orb_{\mu_{j,\omega}}(\pi,\sigma, F_\omega,L_\omega,\delta_\omega,\rho), \rho_2) \ge |\{x \in Q(\sigma,F_\omega,K_\omega,\delta_\omega):~ \mu_{j,\omega} = \mu_{x,\omega}\}|.$$
By choice of $\mu_{j,\omega}$ this implies,
$$\|N_\epsilon( Orb_{\mu_{j,\omega}}(\pi,\cdot, F_\omega,L_\omega,\delta_\omega,\rho), \rho_2)\|_{p,\P_j} \ge \frac{\| N_\epsilon(Orb_\nu(\pi,\cdot,F_\omega,K_\omega,\delta_\omega,\rho), \rho_2)\|_{p,\P_j}}{ |D(\omega)|}.$$
Now the choice of $\mu_{\omega}$ implies
\begin{eqnarray*}
h_{\P, \mu_\omega}^\epsilon(\pi,\rho, \omega,2) &\ge& \lim_{j\to\beta} d_j^{-1}\log \| N_\epsilon( Orb_{\mu_{j,\omega}}(\pi,\cdot, F_\omega,L_\omega,\delta_\omega,\rho), \rho_2)\|_{p,\P_j}\\
 &\ge& \lim_{j\to\beta} d_j^{-1} \log \| N_\epsilon(Orb_\nu(\pi,\cdot,F_\omega,K_\omega,\delta_\omega,\rho), \rho_2)\|_{p,\P_j}\\
&\ge& h_\P^\epsilon(\pi, \rho,2).
\end{eqnarray*}
In the first inequality above, we used that $\beta \ne -$. This proves the first item of Claim 1.

To prove the second item, let $k \in C(\sG^0)$, $f \in \lb\sH\rb_{top}$ and $\eta>0$ be a constant. Because $k$ is continuous, there exists a constant $\delta>0$ such that if $x,y \in \sG^0$ satisfy $\rho(x,y) < \delta$ then $|k(x)-k(y)| < \eta$. 

Let $\omega \in \Omega$ be such that  $k \in L_\omega, f, f^{-1}, \ra(f),\so(f) \in F_\omega$ and $\delta_\omega$ is small enough so that $\frac{\delta_\omega^2}{\delta^2}<\eta$. We may assume $\textrm{diam}(\rho)>0$ since otherwise the theorem is trivial. By choice of $\mu_\omega$, there exist a $(F_\omega,\delta_\omega^2/(100\diam(\rho)^2))$-multiplicative $\sigma:\lb \sH \rb \to \lb d \rb$ and $x \in Orb_\nu(\pi,\sigma,F_\omega,K_\omega,\delta_\omega,\rho)$ such that $\mu_{x,\omega} = \mu_\omega$. Therefore,
\begin{eqnarray*}
 |\mu_\omega(  k \circ f) - \mu_\omega(k \circ \ra(f))| &\le&  |\mu_\omega( k \circ f) - m_x(k \circ f)| + |m_x(k \circ f) - m_x(k \circ \ra(f)) |\\
 && + |m_x(k \circ \ra(f)) - \mu_\omega(k \circ \ra(f)) |\\
  &<& 2\delta_\omega + |m_x(k \circ f) - m_x(k \circ \ra(f)) |\\
   &\le& 2\delta_\omega + |m_x(k \circ f) - m_{x \circ \sigma(f)}(k)| + |m_{x \circ \sigma(f)}(k) - m_x(k \circ \ra(f))|.
%  &=&  2\delta_\omega + |m_x(fk) - m_{x \circ \sigma(f^{-1})}(k)| \le 2\delta_\omega + d^{-1} \sum_{i=1}^d | k(f^{-1} x_i) - k( x_{\sigma(f^{-1})i} ) |.
     \end{eqnarray*}
Next we estimate $|m_x(k \circ f) - m_{x \circ \sigma(f)}(k)|$. Because $x \in  Orb_\nu(\pi,\sigma,F_\omega,K_\omega,\delta_\omega,\rho)$,
 \begin{eqnarray*}
\delta_\omega^2 &>& \rho_2(f \cdot x, x \circ \sigma(f))^2\ge  d^{-1} | \{ 1\le i \le d:~ \rho( f\cdot x_i, x_{\sigma(f)i}) \ge \delta \} |  \delta^2.
\end{eqnarray*}
So,
$$\eta > \frac{\delta_\omega^2}{\delta^2} \ge  d^{-1} | \{ 1\le i \le d:~ \rho( f\cdot x_i, x_{\sigma(f)i}) \ge \delta \} |$$
which implies (by choice of $\delta$)
\begin{eqnarray*}
|m_x(k \circ f) - m_{x \circ \sigma(f)}(k)| \le d^{-1} \sum_{i=1}^d | k(f \cdot x_i) - k( x_{\sigma(f)i}) | \le 2\eta + 2\eta \|k\|_\infty.
\end{eqnarray*}
Next we estimate $|m_{x \circ \sigma(f)}(k) - m_x(k \circ \ra(f))|$. Observe that
\begin{eqnarray*}
|m_{x \circ \sigma(f)}(k) - m_x(k \circ \ra(f))| &=& d^{-1}\left| \sum_{i \in \so(\sigma(f))} k(x_{\sigma(f)i}) - \sum_{i:~x_i \in \ra(f)} k(x_i)\right|\\
&=& d^{-1}\left| \sum_{i \in \ra(\sigma(f))} k(x_i) - \sum_{i:~x_i \in \ra(f)} k(x_i)\right| \\
&\le& d^{-1}\|k\|_\infty | \ra(\sigma(f)) \vartriangle \{i:~x_i \in \ra(f)\} |.
\end{eqnarray*}
Because $x \in Orb_\nu(\pi,\sigma,F_\omega,K_\omega,\delta_\omega,\rho)$,
\begin{eqnarray*}
\delta_\omega^2 &>& \rho_2( x \circ \sigma(\ra(f)), \ra(f)\cdot  x)^2 = d^{-1}\sum_{i=1}^d \rho(x_{\sigma(\ra(f))i}, \ra(f)\cdot x_i)^2 \\
&\ge& d^{-1}\textrm{diam}(\rho)^2| \so(\sigma(\ra(f))) \vartriangle \{i:~ x_i \in \ra(f)\}|\\
&\ge&  d^{-1}\textrm{diam}(\rho)^2| \ra(\sigma(f)) \vartriangle \{i:~ x_i \in \ra(f)\}| - \delta_\omega^2
\end{eqnarray*}
by Lemma \ref{lem:basic-formulas2}. So
\begin{eqnarray*}
|m_{x \circ \sigma(f)}(k) - m_x(k \circ \ra(f))| &\le& d^{-1}\|k\|_\infty | \ra(\sigma(f)) \vartriangle \{i:~x_i \in \ra(f)\} |\\
&\le& d^{-1}\|k\|_\infty \frac{2\delta_\omega^2}{d^{-1}\textrm{diam}(\rho)^2} =  \frac{2\delta_\omega^2\|k\|_\infty}{\textrm{diam}(\rho)^2}.\end{eqnarray*}
The previous estimates now imply
$$ |\mu_\omega(k\circ f) - \mu_\omega(k \circ \ra(f))|  \le 2\delta_\omega +  2\eta + 2\eta \|k\|_\infty + \frac{2\delta_\omega^2\|k\|_\infty}{\textrm{diam}(\rho)^2}.$$
Because $\eta,f,k$ are arbitrary, this implies 
$$\lim_{\omega \to \Omega} |\mu_\omega(k\circ f) - \mu_\omega(k \circ \ra(f))| = 0, \quad \forall f\in \lb\sH\rb_{top}, k \in C(\sG^0)$$
as required.

To prove the third item of Claim 1, let $k \in C(\sH^0)$. Let $\omega\in \Omega$ be such that $k \in K_\omega$ (i.e., $k\circ\pi\in L_\omega$).  By choice of $\mu_\omega$, there exists $\sigma:\lb \sH \rb \to \lb d \rb$ and $x \in Orb_\nu(\pi,\sigma,F_\omega,K_\omega,\delta_\omega,\rho)$ such that $\mu_{x,\omega} = \mu_\omega$. Therefore,
\begin{eqnarray*}
|\mu_{\omega}(k \circ \pi) - \nu(k)| \le | \mu_\omega(k\circ \pi) - m_x(k \circ \pi)| + |m_x(k\circ \pi) - \nu(k) | < 2\delta_\omega.
\end{eqnarray*}
Thus $\lim_{\omega \to \Omega} |\mu_\omega( k\circ \pi ) - \nu(k)| = 0$ as required.
\end{proof}

\section{Some measure zero phenomena}\label{sec:measure-zero}

The main purpose of this section is to prove that entropy does not change upon passage to a conull Borel subgroupoid:
\begin{thm}\label{thm:conull}
Let $(\sG,\mu), (\sH,\nu)$ be discrete pmp groupoids. Also let $\pi:\sG \to \sH$ be a pmp groupoid morphism. Suppose that for $i=1,2$ there are conull Borel subgroupoids $\sG_i \subset \sG, \sH_i \subset \sH$ such that $\pi$ restricted to $\sG_i$ is a class-bijective extension of $\sH_i$. Let $\P$ be a sofic approximation to $\sH$. Let $\pi_i$ denote the restriction of $\pi$ to $\sG_i$. Then
$$h_{\P,\mu}(\pi_1) = h_{\P,\mu}(\pi_2).$$
%where by abuse of notation we use $\P$ to also denote its restriction to $\sH_i$ for $i=1,2$. 
\end{thm}

Before proving this result, let us note that it allows us to extend the notion of entropy. To be precise, let $(\sG,\mu), (\sH,\nu)$ and $\pi$ be as above. We say that $\pi$ is {\em class-bijective almost everywhere} if there exist Borel subgroupoids $\sG'\subset \sG, \sH'\subset \sH$ such that $\pi$ restricted to $\sG'$ is a class-bijective extension of $\sH'$. In this case we define $h_{\P,\mu}(\pi) = h_{\P,\mu}(\pi')$ where $\pi'$ is the restriction of $\pi$ to $\sG'$. By the result above, this does not depend on the choice of $\sG',\sH'$.

We first need a lemma stating that we can change the subset $F$ up to a measure zero set without changing the entropy. More precisely:
\begin{lem}\label{lem:measure-zero1}
Let $\pi:(\sG,\mu) \to (\sH,\nu)$ be a pmp class-bijective extension of discrete pmp groupoids. Let $\P$ be a sofic approximation to $(\sH,\nu)$. For $i=1,2$ let $F_i \subset_f \lb \sH\rb$ with $\sH^0 \in F_1\cap F_2$ and suppose there is a bijection $\beta:F_1 \to F_2$ such that $\nu(f \vartriangle \beta(f)) =0$ for all $f\in F_1$. Then for any finite Borel partitions $\cQ \le \cP$ of $\sG^0$,
$$\inf_{\delta>0} \lim_{j\to\beta}d_j^{-1}\log \| |\Hom(\pi,\cdot,\cP,F_1,\delta)|_\cQ \|_{p,\P_j} = \inf_{\delta>0} \lim_{j\to\beta}d_j^{-1}\log \| |\Hom(\pi,\cdot,\cP,F_2,\delta)|_\cQ \|_{p,\P_j}.$$
In particular, if $\cF \subset \lb \sH \rb$ is such that $\sH^0 \in \sF$ and for every $f\in \lb\sH\rb$ there exists $f' \in \cF$ such that $\nu(f\vartriangle f')=0$ then
$$h_{\P,\mu}(\pi,\cQ,\cP) = \inf_{F \subset_f \cF} \inf_{\delta>0} \lim_{j\to \beta} d_j^{-1} \log \| |\Hom(\pi,\cdot, \cP,F,\delta)|_\cQ \|_{p,\P_j}.$$
\end{lem}

\begin{proof}
Let $\theta: \Sigma(\cP^{F_1}) \to \Sigma(\cP^{F_2})$ be a homomorphism satisyfing $\mu(P \vartriangle \theta(P))=0$ for every $P \in \cP^{F_1}$ and $\theta(P)=P$ for every $P \in \cP$. The definition of sofic approximation implies that $\sigma(f)=\sigma(\beta(f))$ for every $f \in F_1$ and $\P_j$-a.e. $\sigma$. We claim that if $\phi \in \Hom(\pi,\sigma,\cP,F_2,\delta)$ then $\phi\circ \theta \in \Hom(\pi,\sigma,\cP,F_1,3\delta)$. Clearly,
$$\sum_{P \in \cP^{F_1}} | | \phi(\theta(P))|_d - \mu(\theta(P))| \le \sum_{P \in \cP^{F_2}} | | \phi(P)|_d - \mu(P)|  <\delta.$$
Also for any $f\in F_1$,
\begin{eqnarray*}
&&\sum_{P \in \cP} | \sigma_f \cdot \phi(\theta(P)) \vartriangle \phi(\theta(f\cdot P))|_d \\
&\le& \sum_{P \in \cP} | \sigma(\beta(f)) \cdot \phi(\theta(P)) \vartriangle \phi(\beta(f)\cdot \theta(P))|_d +\sum_{P \in \cP}  | \phi(\beta(f)\cdot \theta(P)) \vartriangle \phi(\theta(f\cdot P))|_d \\
&<& \delta + \sum_{P \in \cP} | \phi(\beta(f)\cdot \theta(P)) \vartriangle \phi(\theta(f\cdot P))|_d = \delta + \sum_{P \in \cP} | \phi(\beta(f)\cdot P  \vartriangle \theta(f\cdot P))|_d \\
&=& \delta + \left| \phi\left(\bigcup_{P \in \cP} \beta(f)\cdot P  \setminus \theta(f\cdot P) \right) \right|_d + \left| \phi\left(\bigcup_{P \in \cP} \theta(f\cdot P) \setminus \beta(f)\cdot P   \right)  \right|_d  < 3\delta.
\end{eqnarray*}
%The last equality occurs because, for example, the sets $\{\beta(f)\cdot P  \setminus \theta(f\cdot P)\}_{P \in \cP}$ are pairwise disjoint so $$ \left| \phi\left(\bigcup_{P \in \cP} \beta(f)\cdot P  \setminus \theta(f\cdot P) \right) \right| = \sum_{P \in \cP} |\phi\left(\beta(f)\cdot P  \setminus \theta(f\cdot P) \right) |_d.$$
The last inequality occurs because each of the sets $ \beta(f)\cdot P  \setminus \theta(f\cdot P)$ and $\theta(f\cdot P) \setminus \beta(f)\cdot P$ has measure zero and since $\phi \in  \Hom(\pi,\sigma,\cP,F_2,\delta)$, if $X \in \Sigma(\cP^{F_2})$ is any set with measure zero then $|\phi(X)|_d<\delta$. This inequality proves $\phi\circ \theta \in \Hom(\pi,\sigma,\cP,F_1,3\delta)$ as claimed. 

Because $\theta(P)=P$ for every $P \in \cP$, we also have $\theta(Q)=Q$ for every $Q \in \cQ$. Therefore, if $\phi_1,\phi_2 \in \Hom(\pi,\sigma,\cP,F_2,\delta)$ and $\phi_1 \circ \theta$ restricted to $\cQ$ equals $\phi_2 \circ \theta$ restricted to $\cQ$ then $\phi_1|_\cQ=\phi_2|_\cQ$. So 
%$$| \{ \phi \in  \Hom(\pi,\sigma,\cP,F_2,\delta):~ \phi|_\cQ = \phi_1|_\cQ\}| \le |\cQ|^{\delta d}.$$
%Therefore,
$$|\Hom(\pi,\cdot,\cP,F_2,\delta)|_\cQ \le |\Hom(\pi,\cdot,\cP,F_1,3\delta)|_\cQ.$$
This implies one inequality in the lemma. The opposite inequality follows by symmetry.
\end{proof}

\begin{proof}[Proof of Theorem \ref{thm:conull}]
%By considering $\sG_1 \cap \sG_2$ and $\sH_1\cap \sH_2$, we see that it suffices to consider the special case in $\sG_2 \subset \sG_1$ and $\sH_2 \subset \sH_1$. We may view $\lb \sH_2 \rb$ as a subset of $\lb \sH_1\rb$. 

%For $i=1,2$ let $\cC_i$ be the collection of all subsets of $\sG^0$ of the form $X \cup Y$ where $X$ is a Borel subset of $\sG^0_1 \cap \sG^0_2$ and $Y$ is either empty or equal to $\sG^0_i \setminus \sG^0_{3-i}$. Note that $\cC_i$ is $\pi_i$-generating. Also let $\cF_i = \lb \sH_1 \cap \sH_2 \rb \cup \{\sH_i^0 \setminus \sH_{3-i}^0\}$. By Theorem \ref{thm:K} and Lemma \ref{lem:measure-zero1},
%\begin{eqnarray*}
 %h_{\P,\mu}(\pi_i) &=&  h_{\P,\mu}(\pi_i, \cC_i) \\
 %&=&   \inf_{\cQ \subset \cC_i} \inf_{\cQ \le \cP \subset \cC_i} \inf_{F \subset_f \cF_i} \inf_{\delta>0} \lim_{j\to \beta} d_j^{-1} \log \| |\Hom(\pi_i,\cdot, \cP,F,\delta)|_\cQ \|_{p,\P_j}.
% \end{eqnarray*}
 % &=&  \inf_{\cQ \subset \cC_1} \inf_{\cQ \le \cP \subset \cC_1} \inf_{F \subset_f \cF_1} \inf_{\delta>0} \lim_{j\to \beta} d_j^{-1} \log \| |\Hom(\pi_1,\cdot, \cP,F,\delta)|_\cQ \|_{p,\P_j}\\
  %----
Let 
\begin{itemize}
\item $\cQ_1 \le \cP_1$ be finite Borel partitions of $\sG^0_1$ with $\sG^0_1 \cap \sG^0_2 \in \Sigma(\cQ_1)$,
\item $\cP_2=\{P\cap \sG^0_2:~P\in \cP_1\} \cup \{\sG^0_2 \setminus \sG^0_1\}$, $\cQ_2=\{Q \cap \sG^0_2:~Q \in \cQ_1\}\cup \{\sG^0_2 \setminus \sG^0_1\}$
\item $F \subset_f \lb \sH_1 \cap \sH_2 \rb$ with $\sH^0_1 \cap \sH^0_2 \in F$,
\item $F_i = F \cup \{\sH_i^0\}$ for $i=1,2$,
\item $\sigma:\lb \sH \rb \to \lb d_j \rb$ be $(\{\sH^0\},\delta)$-multiplicative,
\item  $0<\delta<1/100$  and $\phi \in \Hom(\pi_1,\sigma,\cP_1,F_1,\delta)$.
\end{itemize}
We claim that for every $P \in \Sigma(\cP_2^{F_2})$, $P\cap \sG^0_1 \in \Sigma(\cP_1^{F_1})$. It suffices to check this in the special case that $P=f\cdot P_2$ for some $f \in F_2, P_2 \in \cP_2$. If $f  = \sH_2^0$ then $f\cdot P_2=P_2$ and it is obvious in this case. So we may assume $f \ne \sH_2^0$ which imples $f \in F_1$. If $P_2 = \sG^0_2 \setminus \sG^0_1$ then $f \cdot P_2 = \emptyset$ so it is true. Otherwise, $P_2 = P_1 \cap \sG^0_2$ for some $P_1 \in \cP_1$. Since $\sG^0_1 \cap \sG^0_2 \in \Sigma(\cP_1)$, this implies either $P_2 = P_1$ or $P_2 =\emptyset$. So we may assume $P_2 \in \cP_1$. Since $f \in F_1$, $f \cdot P_2  \cap \sG^0_1 \in \Sigma(\cP_1^{F_1})$ as required.% Note that we have actually shown more: if $P \in \cP_2^{F_2}$ then either $P \in \cP_1^{F_1}$ or $P \cap \sG^0_1 = \emptyset$.

Choose an element $P_0 \in \cP_2^{F_2}$ and define $\Psi_\phi:\Sigma(\cP_2^{F_2}) \to \cB(\Delta^0_{d_j})$ by
\begin{displaymath}
\Psi_\phi(P) = \left\{ \begin{array}{cc}
\phi(P \cap \sG^0_1) \cup  \phi(\sG_1^0 \setminus \sG_2^0 ) & P = P_0 \\
\phi(P \cap \sG^0_1) & \textrm{ otherwise. }
\end{array}\right.\end{displaymath}
Note $\Psi_\phi$ is a homomorphism. 

\noindent {\bf Claim 1}. $\Psi_\phi \in \Hom(\pi_2,\sigma,\cP_2,F_2,6\delta)$ (for $\P_j$-a.e. $\sigma$). 

\begin{proof}[Proof of Claim 1] 
Because $\phi \in \Hom(\pi_1,\sigma,\cP_1,F_1,\delta)$, %if $X \in \Sigma(\cP_1^{F_1})$ has measure zero then $|\phi(X)|_{d_j} < \delta$. Therefore,
%\begin{eqnarray}\label{eqn:remember}
%|\Psi_\phi(P_1 \cap \sG^0_2) \vartriangle \phi(P_1)|_d < \delta, \quad |\Psi_\phi(P_2) \vartriangle \phi(P_2\cap \sG^0_1)|_d < \delta
%\end{eqnarray}
%for any $P_1 \in \Sigma(\cP_1^{F_1}), P_2 \in \Sigma(\cP_2^{F_2})$. 
$$|\phi(\sG_1^0 \setminus \sG_2^0)|_{d_j} < \delta + \mu(\sG_1^0 \setminus \sG_2^0) = \delta.$$
So $|\Psi_\phi(P) \vartriangle \phi(P \cap \sG^0_1)|_{d_j} < \delta $ for every $P \in \Sigma(\cP_2^{F_2})$. This implies
\begin{eqnarray*}
\sum_{P \in \cP_2^{F_2}} | \Psi_\phi(P)_{d_j} - \mu(P) | &<& \delta + \sum_{P \in \cP_2^{F_2}} | \phi(P \cap \sG_1^0)_{d_j} - \mu(P \cap \sG_1^0) |\\
&\le& \delta + \sum_{P \in \cP_1^{F_1}} | \phi(P )_{d_j} - \mu(P) |<2\delta.
\end{eqnarray*}
Also, if $f \in F$ then
\begin{eqnarray*}
\sum_{P \in \cP_2} | \sigma_f \cdot \Psi_\phi(P) \vartriangle  \Psi_\phi(f\cdot P) |_{d_j} &<& 2\delta +\sum_{P \in \cP_2} | \sigma_f \cdot \phi(P \cap \sG^0_1) \vartriangle  \phi(f\cdot (P \cap \sG^0_1) ) |_{d_j} \\
&\le& 2\delta +\sum_{P \in \cP_1} | \sigma_f \cdot \phi(P) \vartriangle  \phi(f\cdot P ) |_{d_j}  < 3\delta.
\end{eqnarray*}
Lemma \ref{lem:basic-formulas2} implies $|\sigma(\sH^0) \vartriangle (\sigma(\sH^0)\cap \Delta^0_{d_j})|_{d_j} < \delta$. Since $\sigma(\sH^0_1)=\sigma(\sH^0)$ by definition of sofic approximation and because $\phi \in \Hom(\pi_1,\sigma,\cP_1,F_1,\delta)$,
\begin{eqnarray*}
\delta & > & \sum_{P \in \cP_1} |\sigma(\sH^0_1) \cdot \phi(P) \vartriangle \phi(\sH^0_1 \cdot P) |_{d_j} \\
&>& -\delta + \sum_{P \in \cP_1} | (\sigma(\sH^0) \cap \Delta^0_d) \cdot \phi(P) \vartriangle \phi(P) |_{d_j} \\
&=& - \delta + 1 - |\sigma(\sH^0) \cap \Delta^0_d|_{d_j}
\end{eqnarray*}
implies $|\sigma(\sH^0) \vartriangle \Delta^0_{d_j}|_{d_j} < 4\delta$.  

By definition of sofic approximation, for $\P_j$-a.e. $\sigma$, $\sigma(\sH^0_2)=\sigma(\sH^0)$. So
\begin{eqnarray*}
\sum_{P \in \cP_2} | \sigma(\sH^0_2) \cdot \Psi_\phi(P) \vartriangle  \Psi_\phi(\sH^0_2\cdot P) |_{d_j} &<& 2\delta +\sum_{P \in \cP_2} | \sigma(\sH^0) \cdot \phi(P \cap \sG^0_1) \vartriangle  \phi(P \cap \sG^0_1 ) |_{d_j}\\
&<& 6\delta.
\end{eqnarray*}
This shows  $\Psi_\phi \in \Hom(\pi_2,\sigma,\cP_2,F_2,6\delta)$ as claimed.
\end{proof}

\noindent {\bf Claim 2}. Suppose that $\phi,\psi \in \Hom(\pi_1,\sigma,\cP_1,F_1,\delta)$ and $\Psi_\phi(Q)=\Psi_\psi(Q)$ for every $Q \in \cQ_2$. Then
$$\left|\bigcup_{Q \in \cQ_1} \phi(Q) \vartriangle \psi(Q)\right|_{d_j} < 2\delta.$$

\begin{proof}[Proof of Claim 2]
Let $Q \in \cQ_1$. By definition of $\cQ_2$, $Q \cap \sG^0_2 \in \cQ_2$. Either $\Psi_\phi(Q \cap \sG^0_2) = \phi(Q \cap \sG^0_2)$ or $\Psi_\phi(Q \cap \sG^0_2) = \phi(Q \cap \sG^0_2) \cup \phi(\sG^0_1\setminus \sG^0_2)$. In either case,
$$\phi(Q) \vartriangle \Psi_\phi(Q \cap \sG^0_2) \subset \phi(\sG^0_1 \setminus \sG^0_2).$$
A similar statement holds for $\psi$ in place of $\phi$. So
\begin{eqnarray*}
\bigcup_{Q \in \cQ_1} \phi(Q) \vartriangle \psi(Q) \subset \phi(\sG^0_1 \setminus \sG^0_2) \cup \psi(\sG^0_1 \setminus \sG^0_2).
\end{eqnarray*}
Since $\phi,\psi \in \Hom(\pi_1,\sigma,\cP_1,F_1,\delta)$ and $\mu(\sG^0_1 \setminus \sG^0_2)  =0$,
$$|\phi(\sG^0_1 \setminus \sG^0_2) \cup \psi(\sG^0_1 \setminus \sG^0_2)|_{d_j}<2\delta$$
which implies the claim.
\end{proof}
Claim 2 and Stirling's formula implies
\begin{eqnarray*}
|\Hom(\pi_1,\sigma,\cP_1,F_1,\delta)|_{\cQ_1} &\le& |\Hom(\pi_2,\sigma,\cP_2,F_2,6\delta)|_{\cQ_2}  {d_j \choose \lfloor 2\delta d_j \rfloor}|\cQ_1|^{2\delta d_j}\\
&\le& |\Hom(\pi_2,\sigma,\cP_2,F_2,6\delta)|_{\cQ_2}  \exp( h(3\delta) d_j ) |\cQ_1|^{2\delta d_j}
\end{eqnarray*}
where $h(x)=-x\log(x) - (1-x)\log(1-x).$ Because $\sigma$ is arbitrary, we obtain
\begin{eqnarray*}
\inf_{\delta>0} \lim_{j\to\beta}d_j^{-1}\log \| |\Hom(\pi_2,\cdot,\cP_2,F_2,6\delta)|_{\cQ_2} \|_{p,\P_j} &\ge& \inf_{\delta>0} \lim_{j\to\beta}d_j^{-1}\log \| |\Hom(\pi_1,\cdot,\cP_1,F_1,\delta)|_{\cQ_1} \|_{p,\P_j} \\
&\ge& h_{\P,\mu}(\pi_1, \cQ_1, \cB(\sG^0_1)).
\end{eqnarray*}
We now take the infimum over $\delta, F, \cP_1$ and the supremum over $\cQ_1$ to obtain $h_{\P,\mu}(\pi_2,\cC) \ge h_{\P,\mu}(\pi_1)$ where $\cC$ is the collection of all subsets of $\sG^0$ of the form $X \cup Y$ where $X$ is a Borel subset of $\sG^0_1 \cap \sG^0_2$ and $Y$ is either empty or equal to $\sG^0_2 \setminus \sG^0_{1}$.  This uses Lemma \ref{lem:measure-zero1} with $\cF=\lb \sH_1 \cap \sH_2 \rb \cup \{\sH^0_2\}$. By Theorem \ref{thm:K},  $h_{\P,\mu}(\pi_2,\cC) = h_{\P,\mu}(\pi_2)$. So $h_{\P,\mu}(\pi_2) \ge h_{\P,\mu}(\pi_1)$. By symmetry, this implies the result.

\end{proof}

\section{Bernoulli shifts}\label{sec:bernoulli}

Let $(\sH,\nu)$ be a pmp discrete groupoid and $(K,\kappa)$ be a standard probability space. The {\em Bernoulli shift} over $(\sH,\nu)$ with base space $(K,\kappa)$ is a class-bijective pmp extension $\pi:(\sG,\mu) \to (\sH,\nu)$ constructed as follows. An element of $\sG^0$ is a pair $(x, \omega)$ where $x \in \sH^0$ and $\omega \in K^{\so^{-1}(x)}$, which denotes the set of all functions $\omega:\so^{-1}(x) \to K$. We let $\cB(\sG^0)$ be the smallest sigma-algebra on $\sG^0$ so that the projection map $(x,\omega) \mapsto x$ is measurable and for every $h\in [\sH]$, the map $(x,\omega) \mapsto \omega(hx)$ is measurable. The measure $\mu$ on $\sG^0$ is defined by
$$d\mu(x,\omega) = d\kappa^{\so^{-1}(x)}(\omega)d\nu(x)$$
where $\kappa^{\so^{-1}(x)}$ is the product measure on $K^{\so^{-1}(x)}$. 

Define $\pi:\sG^0 \to \sH^0$ to be the projection map $\pi(x,\omega)=x$. The elements of $\sG$ are pairs $(h,\omega)$ where, if $x=\so(h)$ then $(x,\omega) \in \sG^0$. The source and range maps are defined by $\so(h,\omega) = (\so(h),\omega)$, $\ra(h,\omega) = (y,\psi)$ where $\ra(h)=y$ and $\psi: \so^{-1}(y) \to K$ is defined by $\psi( f) := \omega( f  h).$ The composition map is defined by $(f,\psi)  (h,\omega) = (fh, \omega)$. The main result of this section is:

\begin{thm}\label{thm:bernoulli}
Let $(\sH,\nu)$ be any pmp groupoid, $\pi:(\sG,\mu) \to (\sH,\nu)$ be the Bernoulli shift over $(\sH,\nu)$ with base space $(K,\kappa)$.  Then $h_{\P, \mu}(\pi)  = H(K,\kappa)$ where $H(K,\kappa):= -\sum_{k \in K'} \mu(\{k\}) \log(\mu(\{k\}))$ where $K'$ is any countable subset of $K$ with $\kappa(K') = 1$. If no such countable set exists then $H(K,\kappa):=+\infty$. 
\end{thm}

We first identify a generating sigma-algebra.
\begin{lem}\label{lem:generating}
Let $\pi:(\sG,\mu) \to (\sH,\nu)$ and $(K,\kappa)$ be as in Theorem \ref{thm:bernoulli}. Let $E:\sG^0 \to K$ be the evaluation map $E(x,\omega):=\omega(x)$. Let $\cF = E^{-1}(\cB_K)$ be the inverse image of the Borel sigma-algebra $\cB_K$ of $K$. Then $\cF$ is $\pi$-generating.
\end{lem}

\begin{proof}
Let $\Sigma_{\pi}(\cF)$ be the smallest sigma-sub-algebra of Borel subsets of $\sG^0$ containing $\{f\cdot P:~ f\in \lb\sH\rb, P \in \cF\}$. For every Borel set $Q \subset \sH^0$, we have $Q \in \lb\sH\rb$ and $Q  \cdot \sG^0 = \pi^{-1}(Q)$. It follows that $\Sigma_\pi(\cF)$ contains $\pi^{-1}(\cB(\sH^0))$ where $\cB(\sH^0)$ is the Borel sigma-algebra on $\sH^0$. Therefore, $\pi:\sG^0 \to \sH^0$ is $\Sigma_{\pi}(\cF)$-measurable. Because the Borel sigma-algebra of $\sG^0$ is generated by $\pi$ and $\{E \circ f:~f\in \lb\sH\rb\}$, this implies the lemma.
\end{proof}

Next we show that the inequality $h_{\P, \mu}(\pi)  \le H(K,\kappa)$  holds under general conditions.
\begin{lem}\label{lem:Shannon-bound}
Let $(\sH,\nu)$ be any discrete pmp groupoid, and $\pi:(\sG,\mu) \to (\sH,\nu)$ any pmp class-bijective extension. For any finite Borel partitions $\cQ \le \cP$ of $\sG^0$ we have 
$$h_{\P, \mu}(\pi,\cQ,\cP)  \le H_\mu(\cQ):= -\sum_{Q \in \cQ} \mu(Q)\log\mu(Q).$$
\end{lem}

\begin{proof}
We will use the partition definition of measure entropy.  Let $\sigma:\lb \sH \rb \to \lb d \rb$ be a map, $F \subset_f \lb \sH \rb$ with $\sH^0 \in F$ and $\delta>0$. Let $\Psi:\Hom(\pi,\sigma,\cP,F,\delta) \to \N^\cQ$ be the map $\Psi(\phi)(Q):=|\phi(Q)|$ (for $Q \in \cQ$). 

%For any $v\in \N^\cQ$, 
%$$|\Psi^{-1}(v)|_{\cQ} \le \prod_{Q \in \cQ} { d \choose v(Q) }.$$

For each $\phi \in \Hom(\pi,\sigma,\cP,F,\delta)$, $|d^{-1}|\phi(Q)| - \mu(Q)| \le \delta$ for every $Q \in \cQ$. So $\Psi(\phi)(Q) \in [ (\mu(Q)-\delta) d, (\mu(Q) + \delta) d]$. Therefore
$$| \Psi(  \Hom(\pi,\sigma,\cP,F,\delta) ) | \le (2\delta d+1)^{|\cQ|}.$$
If $v$ is in the image of $\Psi$ then by Stirling's approximation, if $d$ is sufficiently large then
$$|\Psi^{-1}(v)|_\cQ \le \frac{d!}{\prod_{Q \in \cQ} v(Q)! } \le \exp( H_\mu(\cQ)d + \delta' d)$$
where $\delta'>0$ and $\delta' \to 0$ as $\delta\to 0$. We now have
\begin{eqnarray*}
|\Hom(\pi,\sigma,\cP,F,\delta) |_\cQ &=& \sum_{v \in \N^\cQ} |\Psi^{-1}(v)|_{\cQ}  \le  (2\delta d+1)^{|\cQ|}\exp( H_\mu(\cQ)d + \delta' d).
\end{eqnarray*}
Therefore,
\begin{eqnarray*}
h_{\P,\mu}(\pi,  \cQ, \cP, F, \delta) &=& \lim_{j\to\beta} d_j^{-1}\log \||\Hom(\pi,\cdot,\cP,F,\delta) |_\cQ \|_{p,\P_j}\le  H_\mu(\cQ) + \delta'.
\end{eqnarray*}
Because $\delta' \to 0$ as $\delta \to 0$ and $F$ is arbitrary, this proves $h_{\P,\mu}(\pi,  \cQ, \cP) \le H_\mu(\cQ)$. 
\end{proof}

\begin{cor}\label{cor:first-inequality}
Let $\pi:(\sG,\mu) \to (\sH,\nu)$ and $(K,\kappa)$ be as in Theorem \ref{thm:bernoulli}. Then $h_{\P,\mu}(\pi) \le H(K,\kappa)$.
\end{cor}

\begin{proof}
By Lemma \ref{lem:generating}, $E^{-1}(\cB_K)=\cF$ is $\pi$-generating. By Lemma \ref{lem:Shannon-bound}, for any finite Borel partitions $\cQ\le \cP$ of $K$ we have
$$h_{\P,\mu}(\pi,E^{-1}(\cQ), E^{-1}(\cP)) \le H_\kappa(\cQ).$$
Take the infimum over all such $\cP$ and then the supremum over all such $\cQ$ to obtain
$$h_{\P,\mu}(\pi) = h_{\P,\mu}(\pi,\cF) \le H(K,\kappa).$$
The first equality above is Theorem \ref{thm:K}. 
\end{proof}

The next lemma shows that there is at least one good homomorphism for the trivial partition if $\sigma$ is sufficiently good. 

\begin{lem}\label{lem:zero}
Let $(\sH,\nu)$ be a discrete pmp groupoid, $\pi_0:\sH \to \sH$ the identity map and $\P$ a sofic approximation to $(\sH,\nu)$. Then $h_{\P,\nu}(\pi_0)=0$. Moreover, for every $F \subset_f \lb\sH\rb$ with $\sH^0 \in F$ and $\delta>0$ there exists $F' \subset_f \lb\sH\rb$ and $\delta'>0$ such that if $\sigma$ is $(F',\delta')$-multiplicative and $(F',\delta')$-trace preserving then $\Hom(\pi_0,\sigma,\cT, F,\delta) \ne \emptyset$
where $\cT=\{\sH^0,\emptyset\}$ is the trivial partition of $\sH^0$.
\end{lem}

\begin{proof}
Observe that $\cT$ is $\pi_0$-generating. So Theorem \ref{thm:K} implies 
$$h_{\P,\nu}(\pi_0) = \inf_{F \subset_f \lb\sH\rb} \inf_{\delta>0}\lim_{j\to\beta} d_j^{-1}\log \|  |\Hom(\pi_0,\cdot,\cT,F,\delta) |_\cT \|_{p,\P_j}.$$
Because $|\Hom(\pi_0,\sigma,\cT,F,\delta)|_\cT \le 1$ for every $\sigma,F,\delta$, we must have $h_{\P,\nu}(\pi_0)\le 0$.

It may be useful to review the notation in Example \ref{example:d}. Let $F \subset_f \lb \sH \rb$ with $\sH^0 \in F$, $\delta>0$ and $\sigma:\lb \sH\rb \to \lb d \rb$. Observe that $\cT^F$ is the smallest partition of $\sH^0$ containing $\ra(f)$ for every $f\in F$. Define $\phi: \Sigma(\cT^F) \to \cB(\Delta^0_d)$ as follows. First enumerate $\cT^F=\{P_1,\ldots, P_n\}$. If $n=1$ then define $\phi(P_1)=\Delta^0_d$. Otherwise define 
\begin{displaymath}
\phi(P_i) = \left\{ \begin{array}{cc}
\sigma(P_1) \cap \Delta^0_d & \textrm{ if } i=1\\
(\sigma(P_i) \cap \Delta^0_d) \setminus \bigcup_{1\le j < i} \phi(P_j) & \textrm{ if } 1 < i < n\\
 \Delta^0_d \setminus \bigcup_{1\le j < n} \phi(P_j) & \textrm{ if } i=n.
 \end{array}\right.\end{displaymath}
For $I \subset \{1,\ldots,n\}$, let $\phi(\cup_{i\in I} P_i) = \cup_{i\in I} \phi(P_i)$.

%For any $f\in F$, 
%\begin{eqnarray*}
%\sum_{P \in \cT} |\sigma_f\cdot\phi(P) \vartriangle \phi( f\cdot P) |_d &=& | \ra(\sigma_f) \vartriangle \phi(\ra(f))|_d

\noindent {\bf Claim 1}. Let $0<\delta'< (1/14)\delta |\cT^F|^{-3}$. Let $F' \subset_f \lb\sH\rb$ be a set with $\cT^F \subset F'$ and $\sH^0,f,f^{-1},\ra(f),\so(f) \in F'$ for every $f\in F$. If $\sigma$ is $(F', \delta')$-multiplicative and $(F', \delta')$-trace preserving then $\phi \in \Hom(\pi_0,\sigma,\cT,F,\delta)$. 

\begin{proof}[Proof of Claim 1]
Suppose $n=1$. Then $\ra(f)=\sH^0$ for every $f \in F$. By Lemma \ref{lem:basic-formulas2} for any $f\in F$,
\begin{eqnarray*}
\sum_{P \in \cT} |\sigma_f \cdot \phi(P) \vartriangle \phi(f\cdot P) |_d  &=& |\ra(\sigma_f) \vartriangle \Delta^0_d|_d \le |(\sigma(\ra(f)) \cap \Delta^0_d) \vartriangle \Delta^0_d|_d + 11\delta'\\
 &=&|(\sigma(\sH^0) \cap \Delta^0_d) \vartriangle \Delta^0_d|_d + 11\delta' =  |\Delta^0_d \setminus \sigma(\sH^0)|_d +11\delta' \\
 &=& 1-\tr_d(\sigma(\sH^0)) + 11\delta' \le 12\delta' <\delta.
\end{eqnarray*}
This implies the claim. Let us now suppose that $n>1$.

If $P,Q \in \cT^F$ and $P\ne Q$ then $P \cap Q =\emptyset$. So $P  Q = \emptyset$. Because $\sigma$ is $(\cT^F, \delta')$-multiplicative, 
\begin{eqnarray*}
\delta'&\ge & |\sigma(P) \sigma(Q) \vartriangle \sigma(P  Q)|_d = |\sigma(P) \sigma(Q)|_d.
\end{eqnarray*}
If $i<n$ then $\phi(P_i) \vartriangle (\sigma(P_i) \cap \Delta^0_d) \subset \bigcup_{ j<i} \sigma(P_i)  \sigma(P_j).$ So 
\begin{eqnarray}\label{eqn:n0}
|\phi(P_i) \vartriangle (\sigma(P_i) \cap \Delta^0_d)|_d \le \delta' |\cT^F| \quad i<n.
\end{eqnarray}
On the other hand,
\begin{eqnarray}\label{eqn:n}
\phi(P_n) \vartriangle (\sigma(P_n)\cap \Delta^0_d) &\subset& \left(\sigma(P_n) \cap \Delta^0_d \cap \left( \bigcup_{i<n} \phi(P_i) \right) \right) \cup \left(\Delta^0_d \setminus \cup_{i=1}^n \sigma(P_i)\right).
\end{eqnarray}
Now
\begin{eqnarray*}
\left| \sigma(P_n) \cap \Delta^0_d \cap \left( \bigcup_{i<n} \phi(P_i) \right)  \right|_d &\le&  \left| \bigcup_{i<n} \Delta^0_d \cap \sigma(P_i) \cap \sigma(P_n)\right|_d \\
&\le&  \left| \bigcup_{i<n} \sigma(P_i)\sigma(P_n)\right|_d\le \delta' |\cT^F|.
\end{eqnarray*}
By inclusion-exclusion,
\begin{eqnarray*}
\left| \Delta^0_d \setminus \cup_{i=1}^n \sigma(P_i)\right|_d &\le & 1 - \sum_{i=1}^n |\sigma(P_i) \cap \Delta^0_d| + \sum_{i\ne j} |\sigma(P_i) \cap \sigma(P_j) \cap \Delta^0_d|_d\\
&\le&1 - \sum_{i=1}^n |\sigma(P_i) \cap \Delta^0_d| + \sum_{i\ne j} |\sigma(P_i)\sigma(P_j)|_d \\
&\le& 1 - \sum_{i=1}^n \tr_d(\sigma(P_i)) + \delta' |\cT^F|^2 \le 2\delta' |\cT^F|^2.
\end{eqnarray*}
Equations (\ref{eqn:n0}, \ref{eqn:n}) now imply 
$$|\phi(P_i) \vartriangle (\sigma(P_i) \cap \Delta^0_d)|_d \le 3\delta' |\cT^F|^2$$
for every $i$. Therefore, $|\phi(P) \vartriangle (\sigma(P) \cap \Delta^0_d)|_d \le 3\delta' |\cT^F|^3$ for any $P \in \Sigma(\cT^F)$. By Lemma \ref{lem:basic-formulas2},  for any $f\in F$,
\begin{eqnarray*}
\sum_{P \in \cT} |\sigma_f\cdot\phi(P) \vartriangle \phi( f\cdot P) |_d &=& | \ra(\sigma_f) \vartriangle \phi(\ra(f))|_d \le | \sigma(\ra(f)) \vartriangle \phi(\ra(f))|_d  + 10\delta' \\
&\le&  | \sigma(\ra(f)) \vartriangle (\sigma(\ra(f)) \cap \Delta^0_d)|_d  + 13\delta' |\cT^F|^3 \le 14\delta' |\cT^F|^3 <\delta.
\end{eqnarray*}

Because $\sigma$ is $(F', \delta')$-trace-preserving, $|\tr_\sH(P)-\tr_d(\sigma(P))| < \delta'$ for every $P\in \cT^F$. By definition, for $P \subset \sH^0$, $\tr_\sH(P)=\nu(P)$ and $|\sigma(P) \cap \Delta^0_d|_d = \tr_d(\sigma(P))$. Thus
\begin{eqnarray*}
\sum_{P\in \cT^F} | |\phi(P)|d^{-1} - \nu(P) |  &\le&3\delta' |\cT^F|^3 + \sum_{P\in \cT^F} | |\sigma(P) \cap \Delta^0_d|_d - \tr_\sH(P) | \\
&=&3\delta' |\cT^F|^3 + \sum_{P\in \cT^F} | \tr_d(\sigma(P)) - \tr_\sH(P) |  \le 4\delta' |\cT^F|^3  < \delta.
\end{eqnarray*}
This implies Claim 1.
\end{proof}
Claim 1 and the definition of sofic approximation implies 
$$\lim_{j\to\beta} \|  |\Hom(\pi_0,\cdot,\cT,F,\delta) | \|_{p,\P_j} = 1.$$
Therefore $h_{\P,\nu}(\pi_0)\ge 0$. Since we have already obtained the opposite inequality, this implies $h_{\P,\nu}(\pi_0)= 0$. The last statement follows from Claim 1.

\end{proof}

For the proof of Theorem \ref{thm:bernoulli} we will need the next simple estimate.
\begin{lem}\label{lem:basic-formulas4}
Let $F \subset \lb \sH\rb$, $\delta>0$ and  $\sigma:\lb\sH\rb\to\lb d \rb$ is $(F,\delta)$-multiplicative and $(F,\delta)$-trace-preserving. Suppose $f \in F$ and $f \cap \sH^0 \in F$. Then
$$|\sigma(f \cap \sH^0) \vartriangle (\sigma(f)\cap \Delta^0_d)|_d < 9\delta.$$
\end{lem}

\begin{proof}
Because $\sigma$ is $(F,\delta)$-multiplicative,
$$\delta > |\sigma( f(f\cap \sH^0) ) \vartriangle \sigma(f)\sigma(f\cap \sH^0)|_d.$$
Observe that $f(f\cap \sH^0) = f\cap \sH^0$. By Lemma \ref{lem:basic-formulas2}
$$| \sigma(f\cap \sH^0) \vartriangle (\sigma(f\cap \sH^0) \cap \Delta^0_d)|_d <\delta.$$
So,
\begin{eqnarray*}
3\delta&>& |(\sigma(f\cap \sH^0) \cap \Delta^0_d) \vartriangle \sigma(f)(\sigma(f\cap \sH^0) \cap \Delta^0_d)|_d\\
&\ge&|(\sigma(f\cap \sH^0) \cap \Delta^0_d) \setminus \sigma(f)(\sigma(f\cap \sH^0) \cap \Delta^0_d)|_d.
\end{eqnarray*}
However,
$$(\sigma(f\cap \sH^0) \cap \Delta^0_d) \setminus (\sigma(f)\cap \Delta^0_d) \subset (\sigma(f\cap \sH^0) \cap \Delta^0_d) \setminus \sigma(f)(\sigma(f\cap \sH^0) \cap \Delta^0_d).$$
So
\begin{eqnarray}\label{eqn:3d}
3\delta > | (\sigma(f\cap \sH^0) \cap \Delta^0_d) \setminus (\sigma(f)\cap \Delta^0_d) |_d.
\end{eqnarray}
%which implies, by Lemma \ref{lem:basic-formulas2} again, that
%$$4\delta >  | (\sigma(f\cap \sH^0) \setminus \sigma(f)\cap \Delta^0_d) |_d.$$
Because $\sigma$ is $(F,\delta)$-trace-preserving,
\begin{eqnarray*}
\delta&>& | \tr_\sH( f) - \tr_d(\sigma(f))| =| \nu(f \cap \sH^0) - |\sigma(f)\cap \Delta^0_d|_d |,\\
\delta &>& | \tr_\sH( f \cap \sH^0) - \tr_d(\sigma(f \cap \sH^0))| =| \nu(f \cap \sH^0) - |\sigma(f \cap \sH^0)\cap \Delta^0_d|_d |.
\end{eqnarray*}
Therefore,
$$2\delta > | |\sigma(f)\cap \Delta^0_d |_d -  |\sigma(f \cap \sH^0)\cap \Delta^0_d|_d |.$$
By (\ref{eqn:3d}) this implies
\begin{eqnarray*}
| (\sigma(f\cap \sH^0) \cap \Delta^0_d) \cap (\sigma(f)\cap \Delta^0_d) |_d  &=&   | \sigma(f\cap \sH^0) \cap \Delta^0_d |_d - | (\sigma(f\cap \sH^0) \cap \Delta^0_d) \setminus (\sigma(f)\cap \Delta^0_d) |_d \\
&>& |\sigma(f)\cap \Delta^0_d |_d -5\delta.
\end{eqnarray*}
So
\begin{eqnarray*}
| (\sigma(f) \cap \Delta^0_d) \setminus (\sigma(f \cap \sH^0)\cap \Delta^0_d) |_d &=& | \sigma(f) \cap \Delta^0_d|_d  - | (\sigma(f) \cap \Delta^0_d) \cap (\sigma(f \cap \sH^0)\cap \Delta^0_d) |_d < 5\delta.
\end{eqnarray*}
By (\ref{eqn:3d}) this implies
$$8\delta > | (\sigma(f\cap \sH^0) \cap \Delta^0_d) \vartriangle (\sigma(f)\cap \Delta^0_d) |_d.$$
By Lemma \ref{lem:basic-formulas2}, $9\delta >  | (\sigma(f\cap \sH^0) \vartriangle (\sigma(f)\cap \Delta^0_d) |_d.$
\end{proof}

\begin{proof}[Proof of Theorem \ref{thm:bernoulli}]
%By Theorem \ref{thm:strong-equivalence} and Proposition \ref{prop:normalization}, we may assume $\P$ is normal.

 By Corollary \ref{cor:first-inequality}, we need only to prove that $h_{\P,\mu}(\pi) \ge H(K,\kappa)$. Let $E:\sG^0 \to K$ be the evaluation map $E(x,\omega)=\omega(x)$.

Let $F \subset_f [\sH]$  and $\cR$ be a finite partition of $\sH^0$. We assume $\sH^0 \in F, F=F^{-1}$ and $\Sigma(\cR)$ contains $f_1^{-1}f_2 \cap \sH^0$ for every $f_1,f_2\in F$. Let $\cQ \le \cP$ be finite Borel partitions of $K$. Let $\bar{\cP} = E^{-1}(\cP)$, $\bar{\cQ} = E^{-1}(\cQ)$ and $\bar{\cR}=\pi^{-1}(\cR)$. 

By Lemma \ref{lem:zero} there exist a finite set $\tilde{F} \subset \lb\sH\rb$ and $0<\tilde{\delta}<\delta$ such that if $\sigma:\lb \sH \rb \to \lb d\rb$ is $(\tilde{F},\tilde{\delta})$-multiplicative and $(\tilde{F},\tilde{\delta})$-trace-preserving then there exists a homomorphism $\phi \in \Hom(\pi_0,\sigma, \cT,F \Sigma(\cR),\delta)$  where $\pi_0:\sH \to \sH$ is the identity. Let us choose such a map $\sigma$. We also require that $F\Sigma(\cR) \subset \tilde{F}$ and $f_2^{-1}f_1 \in \tilde{F}$ for every $f_1,f_2 \in F$. Because $f_2^{-1}f_1 \cap \sH^0 \in \Sigma(\cR)$ for every $f_1,f_2\in F$, it follows that $F \subset \tilde{F}$ and $f_2^{-1}f_1 \cap \sH^0 \in \tilde{F}$ for every $f_1,f_2 \in F$.

 %, which we also assume to be normal.

%By Lemma \ref{lem:basic-formulas3},
%$$4\delta \ge |\sigma_f\cdot\phi(R) \vartriangle \phi( f\cdot R )|_d$$
%for any $f \in F$ and $R \in \cR$. Also
%\begin{eqnarray*}
%\delta &\ge& \sum_{P \in \cT^{F \cR}}  | |\phi(P)|_d - \mu(P)| = \sum_{R \in \cR^{F}}  | |\phi(R)|_d - \mu(R)|.
%\end{eqnarray*} 

Let $z:\Delta^0_d \to K$. Define $\psi_z:\Sigma(\cP) \to \cB(\Delta^0_d)$ by $\psi_z(P) = \{ u \in \Delta^0_d:~z(u) \in P\}$. Observe that $\psi_z$ is a homomorphism.

Observe that $\bar{\cP}^{F\cR} = (\bar{\cP}^{\cR})^F = ( \bar{\cP} \vee \pi^{-1}(\cR) )^F=(\bar{\cP} \vee \bar{\cR})^F$. Then for any atom $\bar{P} \in (\bar{\cP} \vee \bar{\cR})^F$, there exist atoms $\bar{P}_f \in \bar{\cP}, \bar{R}_f \in \bar{\cR}$ such that $\bar{P}  = \bigcap_{f\in F} f^{-1}\cdot(\bar{R}_f \cap \bar{P}_f)$. This assignment is unique (because $F \subset [\sH])$. Define $\phi_z:\bar{\cP}^{F \cR} \to \cB(\Delta^0_d)$ by
$$\phi_z(\bar{P}) :=   \phi\left(\bigcap_{f\in F} f^{-1}\cdot\pi( \bar{R}_f) \right) \cap \bigcap_{f\in F} \sigma_f^{-1}\cdot\psi_z(E(\bar{P}_f)).$$
 For any $S \subset \bar{\cP}^{F \cR}$, define $\phi_z(\bigcup_{P \in S} P) := \bigcup_{P\in S} \phi_z(P)$.
 This defines $\phi_z$ on $\Sigma(\bar{\cP}^{F \cR} )$.
 
Observe that $\phi_z$ is well-behaved with respect to unions and intersections in the sense that $\phi_z(A \cup B) = \phi_z(A) \cup \phi_z(B)$ and $\phi_z(A \cap B) = \phi_z(A) \cap \phi_z(B)$ for any $A,B \in \Sigma(\bar{\cP}^{F \cR} )$. However, $\phi_z(\sG^0)$ need not equal $\Delta^0_d$. So it may not be a homomorphism. Still, we will show that it is close to a homomorphism. 

To manage error terms we will use big $O$ notation. The constant implicit in the notation is allowed to depend on $F,\cR,\cP,\cQ,K,\kappa$ but not on $\delta$, $d$ or $\sigma$.

{\bf Claim 1}. Let $F' \subset F$ and let $R'$ be the set of all $x\in \sH^0$ such that for every $f \in F$ there exists a unique $f' \in F'$ such that $\so^{-1}(x) \cap f = \so^{-1}(x) \cap f' \ne \emptyset$. Similarly, let $R''$ be the set of all $u \in \Delta^0_d$ such that for every $f \in F$ there exists a unique $f' \in F'$ such that $\so^{-1}(u) \cap \sigma(f) = \so^{-1}(u) \cap \sigma(f') \ne \emptyset$. Then $R' \in \Sigma(\cR)$ and $|\phi(R') \vartriangle R''|_d  = O(\delta)$.

\begin{proof}
Let $\lambda:F \to F'$ be a function whose restriction to $F'$ is the identity map. Let 
$$R'_\lambda = \left(\bigcap_{f\in F} f^{-1}  \lambda(f) \cap \sH^0\right) \setminus  \left(\bigcap_{f_1 \ne f_2 \in F'} f_2^{-1}  f_1 \cap \sH^0\right).$$
Then $R' = \bigcup_\lambda R'_\lambda$. Since each $R'_\lambda \in \Sigma(\cR)$ by choice of $\cR$, this shows $R' \in \Sigma(\cR)$. We claim that for each $f_1,f_2 \in F$,
%By Lemmas  \ref{lem:basic-formulas3} for each $f_1,f_2 \in F$,
\begin{eqnarray*}
&&| \phi( f_2^{-1}  f_1 \cap \sH^0) \vartriangle (\sigma(f_2)^{-1}  \sigma(f_1) \cap \Delta^0_d)|_d\\
 &\le& | \phi( f_2^{-1}  f_1\cap \sH^0) \vartriangle \sigma(f_2^{-1} f_1 \cap \sH^0)|_d+ | \sigma(f_2^{-1} f_1 \cap \sH^0) \vartriangle (\sigma(f_2^{-1} f_1 )\cap \Delta^0_d) |_d \\
  &&+ | \sigma(f_2^{-1} f_1) \cap \Delta^0_d ) \vartriangle  (\sigma(f_2)^{-1}  \sigma(f_1) \cap \Delta^0_d)|_d\\
 &\le& 3\delta + 9\tdelta +  \tdelta + 15\tdelta  \le 30\delta.
\end{eqnarray*}
This uses Lemmas \ref{lem:basic-formulas3}, \ref{lem:basic-formulas4} and \ref{lem:basic-formulas2}. So if 
$$R''_\lambda= \left(\bigcap_{f\in F} \sigma(f)^{-1}  \sigma(\lambda(f)) \cap \Delta^0_d\right) \setminus  \left(\bigcap_{f_1 \ne f_2 \in F'} \sigma(f_2)^{-1}  \sigma(f_1) \cap \Delta^0_d\right)$$
 then $|\phi(R'_\lambda) \vartriangle R''_\lambda|_d \le 60\delta |F|^2$. Observe that $R''=\bigcup_\lambda R''_\lambda$.  So
\begin{eqnarray*}
 |\phi(R') \vartriangle R''|_d &\le& \sum_\lambda | \phi(R'_\lambda) \vartriangle R''_\lambda |_d \le 60\delta |F|^{|F|+2} = O(\delta).
 \end{eqnarray*}
 
\end{proof}

{\bf Claim 2}. Let $z: \Delta^0_d \to K$ be random with law equal to the product measure $\kappa^{\Delta^0_d}$. Then there is a constant $C>0$ such that for any $P \in \bar{\cP}^{F \cR}$ and $\epsilon>0$ the probability that $||\phi_z(P)|_d - \mu(P) | \le  \epsilon + C\delta$ is at least $1-O(\frac{ \delta}{\epsilon^2})$. The constant $C$ as well as the constant implicit in the $O(\cdot)$ notation may depend on $P,\cR,\cP,\cQ,F,K,\kappa$ but not on $\epsilon, \delta, d, \sigma$. 

\begin{proof}[Proof of Claim 2]
For $f\in F$, let $R_f \in \cR$ and $P_f \in \cP$ be such that 
$$P = \bigcap_{f\in F} f^{-1} \cdot(\pi^{-1}(R_f) \cap E^{-1}(P_f) ).$$
Let $R_P = \bigcap_{f\in F} f^{-1}\cdot R_f$. We define an equivalence relation $\sim_P$ on $F$ by $f_1 \sim_P f_2$ if $f_1x = f_2x$ for every $x\in R_P$. Equivalently, $f_1 \sim_P f_2$ if $R_P \subset f_2^{-1}f_1 \cap \sH^0$. Let $F_P \subset F$ be a set containing exactly one element from each $\sim_P$ equivalence class. Because $\Sigma(\cR) \supset f_2^{-1}f_1 \cap \sH^0$ for every $f_1,f_2 \in F$, it follows that for every $f \in F$ and $x \in R_P$, there exists a unique $f' \in F_P$ such that $\so^{-1}(x) \cap f = \so^{-1}(x) \cap f'$ (equivalently, such that $fx=f'x$). 

Observe that
$$P=\{(x,\omega)\in \sG^0:~ x\in \pi^{-1}(R_P) \textrm{ and } \omega(f \cap \so^{-1}(x)) \in P_f~\forall f\in F_P\}.$$
So the definition of $\mu$ implies
$$\mu(P) = \nu(R_P) \prod_{f\in F_P} \kappa(P_f).$$
Note $$\phi_z(P) = \phi(R_P) \cap \bigcap_{f\in F} \sigma_f^{-1}\cdot\psi_z(P_f)=\{q \in \phi(R_P):~z(\sigma_f \cdot q) \in P_f~\forall f\in F\}.$$
%Note
%Because $\phi \in \Hom(\pi_0,\sigma, \cT,F \cR,\delta)$, $|\phi(f\cdot R_f) \vartriangle \sigma_f\cdot\phi(R_f)|_d \le \delta$ for every $f\in F$. So 
%$$\left| \bigcap_{f\in F} \sigma_f  \phi(R_f) \vartriangle \phi(R_P) \right|_d = O(\delta)$$
%which implies 
%$$\left| \phi_z(P) \vartriangle \phi(R_P) \cap \bigcap_{f\in F} \sigma_f\cdot\psi_z(P_f)\right|_d = O(\delta).$$
For $q \in \Delta^0_d$, let $X_q=1$ if $q \in \phi_z(P)$ and $X_q=0$ otherwise. Note that $X_q=0$ if $q \notin \phi(R_P)$. So
$$\EE[ |\phi_z(P)| ] = \sum_{q \in \Delta^0_d} \EE[X_q] =  \sum_{q \in \phi(R_P)} \EE[X_q]$$
where $\EE[\cdot]$ denotes expected value with respect to $z$.

Let $T_P$ be the set of all $q \in \Delta^0_d$ such that for every $f \in F$ there is a unique $f' \in F_P$ with $\sigma(f) q = \sigma(f')q$. By Claim 1, $|\phi(R_P) \setminus T_P|_d = O(\delta)$. Therefore,
$$\EE[ |\phi_z(P)| ] = O(\delta d) + \sum_{q \in \phi(R_P) \cap T_P} \EE[X_q] .$$
If $q \in \phi(R_P) \cap T_P$ then
$$\EE[X_q] = \textrm{Prob}(X_q = 1) = \prod_{f\in F_P} \textrm{Prob}(z(\sigma_f\cdot q) \in P_f) = \prod_{f\in F_P} \kappa(P_f) = \frac{\mu(P)}{\nu(R_P)}.$$
Because $\phi \in \Hom(\pi_0,\sigma, \cT,F \cR,\delta)$, $|\phi(R_P)|_d = \nu(R_P) +O( \delta)$. So
\begin{eqnarray*}
\sum_{q \in \phi(R_P) \cap T_P} \EE[X_q] &=& |  \phi(R_P) \cap T_P | \prod_{f\in F_P} \kappa(P_f) =  O(\delta d)+ |\phi(R_P)| \prod_{f\in F_P} \kappa(P_f) \\
&=&  O(\delta d)+ \nu(R_P) d \prod_{f\in F_P} \kappa(P_f) =  O(\delta d)+ \mu(P)d .
\end{eqnarray*}
Thus
$$ \EE[ |\phi_z(P)| ] = \mu(P)d  + O(\delta d).$$

Next we estimate the variance of $|\phi_z(P)|$, which we denote by $\textrm{Var}(|\phi_z(P)|) = \EE[ |\phi_z(P)|^2 ] - \EE[ |\phi_z(P)| ]^2$. Observe:
\begin{eqnarray*}
\EE[ |\phi_z(P)|^2 ] &=& \sum_{u,v \in \phi(R_P)} \EE[X_u X_v].
\end{eqnarray*}
If $u,v \in \phi(R_P)$ and $\sigma(f_1) \cdot u \ne \sigma(f_2)\cdot v$ for any $f_1,f_2 \in F$ then $X_u$ and $X_v$ are independent. If in addition  $u,v \in \phi(R_P)\cap T_P$ then
$$\EE[X_uX_v] = \EE[X_u] \EE[X_v] = \prod_{f\in F_P} \kappa(P_f)^2 = \frac{\mu(P)^2}{\nu(R_P)^2}.$$
 On the other hand $X_uX_v \le 1$ almost surely (regardless of whether or not they are independent) and the number of pairs $(u,v) \in \phi(R_P) \times \phi(R_P)$ such that either $u\notin T_P, v \notin T_P$ or $\sigma(f_1)\cdot u=\sigma(f_2)\cdot v$ for some $f_1,f_2\in F$ is at most 
 $$|F|^2 |\phi(R_P)| + 2|\phi(R_P)\setminus T_P|\cdot |\phi(R_P)| =|F|^2|\phi(R_P)| + O(\delta d^2).$$
So
 $$ \sum_{u,v \in \phi(R_P)} \EE[X_u X_v] \le |F|^2 |\phi(R_P)| + |\phi(R_P)|^2\frac{\mu(P)^2}{\nu(R_P)^2} + O(\delta d^2).$$
 Because $|\phi(R_P)|_d = \nu(R_P) +O( \delta)$,
 \begin{eqnarray*}
 \textrm{Var}(|\phi_z(P)|) &\le&  |F|^2 |\phi(R_P)| + |\phi(R_P)|^2\frac{\mu(P)^2}{\nu(R_P)^2} - d^2\mu(P)^2 + O(\delta d^2)\\
 &\le& |F|^2 \nu(R_P)d + O(\delta d^2).
 \end{eqnarray*}
By Chebyshev's inequality, for any $\epsilon>0$,
 \begin{eqnarray*}
 \textrm{Prob}(| |\phi_z(P)|_d - \EE[ |\phi_z(P)|_d ] | > \epsilon  ) &= & \textrm{Prob}(| |\phi_z(P)| - \EE[ |\phi_z(P)| ] | > d\epsilon) \\
 &\le& \frac{ \textrm{Var}(|\phi_z(P)|) }{d^2\epsilon^2}  \\
 &\le& \frac{ |F|^2 \nu(R_P)d + O(\delta d^2) }{d^2 \epsilon^2}  = O\left(\frac{ \delta}{\epsilon^2}\right).
  \end{eqnarray*}
  Because $\mu(P) = \EE[ |\phi_z(P)|_d ] + O(\delta)$, this implies the claim.
  \end{proof}

Claim 2 implies that with probability $\ge 1-O(\frac{ \delta}{\epsilon^2})$,
\begin{eqnarray*}
\sum_{P\in \bar{\cP}^{F \cR}} ||\phi_z(P)|_d - \mu(P) | = O(\epsilon + \delta).
\end{eqnarray*}
In this case, 
\begin{eqnarray}\label{eqn:complement}
|\phi_z(\sG^0)|_d \ge 1 - O(\epsilon + \delta).
\end{eqnarray}
Now define a homomorphism $\phi'_z:\Sigma( \bar{\cP}^{F \cR} ) \to \cB(\Delta^0_d)$ by choosing $P_0 \in \bar{\cP}^{F \cR}$ and defining
\begin{displaymath}
\phi'_z(P) = \left\{\begin{array}{cc}
\phi_z(P)  & \textrm{ if } P \ne P_0\\
\phi_z(P_0)\cup (\Delta^0_d \setminus \phi_z(\sG^0)) & \textrm{ if } P=P_0
\end{array}\right.
\end{displaymath}
It follows that 
\begin{eqnarray}\label{eqn:phi'}
\sum_{P \in\bar{\cP}^{F \cR}} | \phi'_z(P) \vartriangle \phi_z(P) | = O(\delta + \epsilon).
\end{eqnarray}
Therefore,
\begin{eqnarray}\label{eqn:good2}
\sum_{P\in \bar{\cP}^{F \cR}} ||\phi'_z(P)|_d - \mu(P) | = O(\epsilon + \delta).
\end{eqnarray}

{\bf Claim 3}. Let $z: \Delta^0_d \to K$ be random with law $\kappa^{\Delta^0_d}$ and $\epsilon>0$. Then with probability at least $1-O(\frac{ \delta}{\epsilon^2})$,  $$\sum_{P \in \cP} |\phi'_z( E^{-1}(P)) \vartriangle \psi_z(P)|_d  = O(\epsilon + \delta).$$

\begin{proof}[Proof of Claim 3]
The definition of $\phi_z$ immediately implies 
$$ \phi_z(E^{-1}(P)) \cap \sigma(\sH^0) \subset \sigma(\sH^0)\cdot \phi_z(E^{-1}(P))  \subset \psi_z(P)$$
for every $P\in \cP$. Because $\phi_z(A\cap B)=\phi_z(A)\cap\phi_z(B)$ for any $A,B \in \bar{\cP}^{F\cR}$,
\begin{eqnarray*}
\phi_z( E^{-1}(P)) \vartriangle \psi_z(P) &\subset&  (\Delta^0\setminus \sigma(\sH^0)) \cup  \left[ (\phi_z(E^{-1}(P)) \cap \sigma(\sH^0)) \vartriangle  (\psi_z(P) \cap \sigma(\sH^0) ) \right] \\
&\subset&  (\Delta^0\setminus \sigma(\sH^0)) \cup  \left[ \psi_z(P) \setminus  \left(\phi_z(E^{-1}(P)) \cap \sigma(\sH^0) \right)\right]  \\
&\subset& (\Delta^0\setminus \sigma(\sH^0)) \cup  \left[ \Delta^0_d \setminus (\phi_z(\sG^0)  \cap \sigma(\sH^0)) \right] \\
&=& \Delta^0_d \setminus (\phi_z(\sG^0)  \cap \sigma(\sH^0))
\end{eqnarray*}
which implies
\begin{eqnarray*}
\bigcup_{P\in \cP} \phi_z( E^{-1}(P)) \vartriangle \psi_z(P) &\subset&  \Delta^0_d \setminus (\phi_z(\sG^0)  \cap \sigma(\sH^0)).
\end{eqnarray*}
Because $\sigma$ is $(\tilde{F},\tilde{\delta})$-trace-preserving,
$$\tr_d(\sigma(\sH^0)) > \tr_\sH(\sH^0) - \tilde{\delta} > 1 -\delta.$$
So 
\begin{eqnarray}\label{eqn:small}
|\sigma(\sH^0) \vartriangle \Delta^0_d|_d  = 1 - |\sigma(\sH^0) \cap \Delta^0_d|_d  = 1-\tr_d(\sigma(\sH^0)) < \delta.
\end{eqnarray}
% &\le& 1- |\sigma(\sH^0)_d| + \tilde{\delta} \le 1- \tr_d(\sigma(\sH^0)) + \tilde{\delta} < 2\delta.
% \end{eqnarray*}
%By Lemma \ref{lem:basic-formulas2},
%\begin{eqnarray*}
%\delta &>& |\sigma(\sH^0) \vartriangle (\sigma(\sH^0)\cap \Delta^0_d) |_d = |\sigma(\sH^0) \setminus (\sigma(\sH^0)\cap \Delta^0_d) |_d \\
%&=&  |\sigma(\sH^0)|_d - \tr_d(\sigma(\sH^0)) \ge 
By (\ref{eqn:phi'}, \ref{eqn:complement}) 
\begin{eqnarray*}
\sum_{P \in \cP} |\phi'_z( E^{-1}(P)) \vartriangle \psi_z(P)|_d &\le & O(\epsilon+\delta) + \sum_{P \in \cP} |\phi_z( E^{-1}(P)) \vartriangle \psi_z(P)|_d \\
&\le & O(\epsilon+\delta) +  |\Delta^0_d \setminus(\sigma(\sH^0)\cup \phi_z(\sG^0))|_d = O(\epsilon + \delta)
 \end{eqnarray*}
 with probability at least $1-O(\frac{ \delta}{\epsilon^2})$.
  \end{proof}

{\bf Claim 4}.  Let $z: \Delta^0_d \to K$ be random with law $\kappa^{\Delta^0_d}$ and $\epsilon>0$. Then with probability at least $1-O(\frac{ \delta}{\epsilon^2})$,  for every $P \in \cP$, $R \in \cR$ and $f\in F$, 
$$|\phi'_z( f\cdot(\pi^{-1}(R) \cap E^{-1}(P))) \vartriangle \sigma_f\cdot(\phi(R) \cap \psi_z(P))|_d = O(\epsilon+\delta).$$

\begin{proof}[Proof of Claim 4]
Because $\phi \in \Hom(\pi_0,\sigma, \cT,F \cR,\delta)$, $|\phi(\sH^0)|_d > 1-\delta$. Also if $R \in \cR$ then
$$\delta > |\phi(R \cdot \sH^0) \vartriangle \sigma_R\cdot \phi(\sH^0)|_d > |\phi(R) \vartriangle \sigma_R|_d - \delta.$$
So if $f \in F$ then by $(\tilde{F},\tilde{\delta})$-multiplicativity of $\sigma$,
\begin{eqnarray*}
\delta &>& |\phi(fR \cdot \sH^0) \vartriangle \sigma_{fR}\cdot \phi(\sH^0)|_d >  |\phi(f\cdot R) \vartriangle \sigma_{fR}|_d -\delta\\
&>&  |\phi(f\cdot R) \vartriangle \sigma_f\sigma_R|_d -2\delta>  |\phi(f\cdot R) \vartriangle \sigma_f\cdot \phi(R)|_d -4\delta.
\end{eqnarray*}
By (\ref{eqn:phi'}) it suffices to prove Claim 4 with $\phi_z$ in place of $\phi'_z$. By definition of $\phi_z$, $\phi_z( f\cdot\pi^{-1}(R)) = \phi(f\cdot R) \cap \phi_z(\sG^0)$. By (\ref{eqn:complement}) this implies %and $\phi_z(\pi^{-1}(R)) = \phi(R) \cap \phi_z(\sG^0)$ .%and $|\phi(f\cdot R) \vartriangle \sigma_f\cdot\phi(R)|_d \le \delta$. 
\begin{eqnarray*}
|\phi_z(f\cdot\pi^{-1}(R)) \vartriangle \sigma_f\cdot\phi(R)|_d &=& O(\epsilon+\delta) + |\phi(f\cdot R) \vartriangle \sigma_f\cdot\phi_z(R)|_d = O(\epsilon+\delta)
\end{eqnarray*}
with  probability at least $1-O(\frac{ \delta}{\epsilon^2})$.

%\subset \sigma_f\cdot\phi(R)$. Because $\sigma_f \in [d]$ and $\phi$ is a homomorphism, it follows that $\sigma_f\cdot\phi(R) \setminus \phi_z(f\cdot\pi^{-1}(R)) \subset \Delta^0_d \setminus \phi_z(\sG^0)$. Because $|\Delta^0_d \setminus \phi_z(\sG^0)|_d = O(\epsilon+\delta)$ with probability at least $1-O(\frac{ \delta}{\epsilon^2})$, it follows that $|\phi_z(f\cdot\pi^{-1}(R)) \vartriangle \sigma_f\cdot\phi(R)| = O(\epsilon+\delta)$ with  probability at least $1-O(\frac{ \delta}{\epsilon^2})$.

By definition of $\phi_z$, $\phi_z(f \cdot E^{-1}(P))= \sigma(f^{-1})^{-1}\cdot\psi_z(P) \cap \phi_z(\sG^0)$. By Lemma \ref{lem:basic-formulas2}, $|\sigma(f^{-1})^{-1}\vartriangle \sigma(f)|_d =O(\delta)$. So (\ref{eqn:complement}) implies 
$$|\phi_z(f\cdot E^{-1}(P)) \vartriangle \sigma_f\cdot\psi_z(P)|_d = O(\epsilon+\delta)$$ 
with  probability at least $1-O(\frac{ \delta}{\epsilon^2})$.

From the previous two paragraphs we obtain 
\begin{eqnarray*}
O(\epsilon+\delta) &=& |\phi_z(f\cdot\pi^{-1}(R)) \vartriangle \sigma_f\cdot\phi(R)|_d + |\phi_z(f\cdot E^{-1}(P)) \vartriangle \sigma_f\cdot\psi_z(P)|_d \\
&\ge& |\phi_z( f\cdot(\pi^{-1}(R) \cap E^{-1}(P))) \vartriangle \sigma_f\cdot(\phi(R) \cap \psi_z(P))|_d
\end{eqnarray*}
with  probability at least $1-O(\frac{ \delta}{\epsilon^2})$. This implies the claim.
\end{proof}
 
 {\bf Claim 5}. Given $\epsilon>0$, we have that with probability at least $1-O(\frac{ \delta}{\epsilon^2})$, for every $f\in F$ and $\bar{R}\in \Sigma(\pi^{-1}(\cR))$, if $R = \pi(\bar{R})$ then
 $$\sum_{\bar{P} \in E^{-1}(\cP)} | \phi'_z( (fR)\cdot\bar{P}) \vartriangle \sigma_{fR} \cdot\phi'_z(\bar{P}) |_d = O(\epsilon+\delta).$$
 
 \begin{proof}[Proof of Claim 5]
 Let $R = \pi(\bar{R})$ and $P=E(\bar{P})$ for any $\bar{P} \in E^{-1}(\cP)$. Using Claims 3 and 4, we see that with probability at least $1-O(\frac{ \delta}{\epsilon^2})$, \begin{eqnarray*}
&& \sum_{\bar{P} \in E^{-1}(\cP)}   | \phi'_z( (f R) \cdot \bar{P}) \vartriangle \sigma_{f  R} \cdot \phi'_z(\bar{P}) |_d   =  \sum_{\bar{P} \in E^{-1}(\cP)} | \phi'_z( f\cdot (\bar{P} \cap \bar{R})) \vartriangle \sigma_{f  R}\cdot \phi'_z(\bar{P}) |_d  \\
 &\le& O(\epsilon+\delta) + \sum_{\bar{P} \in E^{-1}(\cP)} |\sigma_f\cdot(\phi(R) \cap \psi_z(P)) \vartriangle \sigma_{f}  \sigma_{R} \cdot\phi'_z(\bar{P}) |_d  +   |\sigma_{f}  \sigma_{R} \cdot \phi'_z(\bar{P}) \vartriangle  \sigma_{f  R} \cdot \phi'_z(\bar{P}) |_d  \\
  &\le& O(\epsilon+\delta) + \sum_{\bar{P} \in E^{-1}(\cP)} |(\phi(R) \cap \psi_z(P)) \vartriangle  \sigma_{R} \cdot\phi'_z(\bar{P}) |_d \\
  &\le& O(\epsilon+\delta) + \sum_{\bar{P} \in E^{-1}(\cP)} |(\phi(R) \cap \psi_z(P)) \vartriangle  \sigma_{R} \cdot\psi_z(P) |_d \\
     &\le&  O(\epsilon+\delta) + \sum_{\bar{P} \in E^{-1}(\cP)} \left|(\phi(R) \cap \psi_z(P)) \vartriangle  \left((\sigma_{R} \cap \Delta^0_d) \cdot\psi_z(P) \right) \right|_d \\
    &=&  O(\epsilon+\delta) + \sum_{\bar{P} \in E^{-1}(\cP)} \left|(\phi(R) \cap \psi_z(P)) \vartriangle  \left(\sigma_{R} \cap \psi_z(P) \right) \right|_d  = O(\epsilon+\delta).
     \end{eqnarray*}
 The first inequality uses Claim 4, the next one uses the $(F\Sigma(\cR),\delta)$-multiplicativity of $\sigma$ and Lemma \ref{lem:basic-formulas1}. The third inequality uses Claim 3. The fourth inequality uses Lemma \ref{lem:basic-formulas2} and the last equality uses Lemma \ref{lem:basic-formulas3}. 
 \end{proof}
 
 Let $S$ be the set of all maps $z:\Delta^0_d \to K$ such that $\phi'_z \in \Hom(\pi,\sigma,\bar{\cP}, F \Sigma(\cR), O(\epsilon+\delta))$. Of course this depends on the constant implicit in the $O(\cdot)$ notation but for simplicity we leave this dependence implicit. It follows from Claims 2 and 5 that $\kappa^{\Delta^0_d}(S) \ge 1-O(\frac{ \delta}{\epsilon^2})$.
 
 Define $\pi_\cQ:K \to \cQ$ by $\pi_\cQ(k) = Q$ if $k \in Q$. If $d$ is sufficiently large and $\frac{\delta}{\epsilon^2}$ is sufficiently small then by the asymptotic equipartition property,
 $$| \{\pi_\cQ \circ s:~s\in S\} |  \ge (1/2)\exp( d H_\kappa(\cQ) - o(d))$$
where $H_\kappa(\cQ) = -\sum_{Q \in \cQ} \kappa(Q) \log(\kappa(Q))$.
 
 If $s,t \in S$ and there is some $q \in \phi_s(\sG^0) \cap \phi_t(\sG^0)\cap \sigma(\sH^0)$ such that $\pi_\cQ(s(q)) \ne \pi_\cQ(t(q))$ then $\phi'_s|_{E^{-1}(\cQ)} \ne \phi'_t|_{E^{-1}(\cQ)}$. For any fixed $s\in S$, the number of elements in the set $\{\pi_\cQ \circ t:~t \in S\}$ such that $\pi_\cQ(s(q)) = \pi_\cQ(t(q))$ for every $q \in \phi_s(\sG^0) \cap \phi_t(\sG^0) \cap \sigma(\sH^0)$ is at most 
 $$|\cQ|^{O(\epsilon+\delta)d} \exp( h(O(\epsilon+\delta)) d)$$
where $h(x)=-\log(x)-(1-x)\log(1-x)$. This uses (\ref{eqn:complement}, \ref{eqn:small}). So
 \begin{eqnarray*}
 | \Hom(\pi,\sigma,\bar{\cP}, F \Sigma(\cR), O(\epsilon+\delta)) |_{\bar{\cQ}} &\ge& (1/2)|\cQ|^{-O(\epsilon + \delta)d} \exp( d H_\kappa(\cQ)-h(O(\epsilon+\delta))d - o(d)).
 \end{eqnarray*}

Because $\sigma$ is an arbitrary $(\tilde{F},\tilde{\delta})$-multiplicative, $(\tilde{F},\tilde{\delta})$-trace-preserving map (and the constant implicit in the $O(\cdot)$-notation does not depend on $\sigma$) it follows that
$$\lim_{j\to J} d_j^{-1} \log \| |\Hom(\pi,\cdot,\bar{\cP}, F \Sigma(\cR), O(\epsilon+\delta))|_{\bar{\cQ}} \|_{\P_j,p} \ge H_\kappa(\cQ) - h(O(\epsilon+\delta)).$$
Because $\epsilon>0$ is arbitrary, by taking the infimum over all $\cR, F,\delta$ and noting that every finite subset of $\lb \sH \rb$ is contained in a set of the form $F \Sigma(\cR)$ for some $F \subset_f [\sH], \cR$ a Borel partition of $\sH^0$ (up to a measure zero set), we see that $h_{\P,\mu}(\pi,\bar{\cQ},\bar{\cP}) \ge  H_\kappa(\cQ)$. This uses Lemma \ref{lem:measure-zero1}. We can now take the infimum over all $\cP$ and the supremum over all $\cQ$ to obtain the theorem.
\end{proof}

\section{Non-free Bernoulli shifts}\label{sec:non-free}

This section answers a question of Benjy Weiss on non-free Bernoulli shifts. To explain, we need some terminology. So let $G$ be a countable group, $2^G$ denote the space of all subsets of $G$ with the product topology and  $\textrm{Sub}_G \subset 2^G$ be the space of subgroups of $G$. Because $2^G$ is a compact metrizable space and $\textrm{Sub}_G$ is closed in $2^G$,  $\textrm{Sub}_G$ is also a compact metrizable space and $G$ acts on $\textrm{Sub}_G$ by conjugation. An {\em invariant random subgroup} (IRS) is a random subgroup $H \in \textrm{Sub}_G$ with conjugation-invariant law. This terminology was introduced in \cite{AGV12}. We will be interested in Bernoulli shifts over the coset space $G/H$ of an invariant random subgroup.

Given a Borel space $K$, let $\textrm{Sub}_G\otimes K$ be the set of all pairs $(H,\omega)$ where $H \in \textrm{Sub}_G$ and $\omega:G/H \to K$. We can embed $\textrm{Sub}_G \otimes K$ into $\textrm{Sub}_G \times K^G$ via the map
$$(H,\omega) \mapsto (H, \tilde{\omega} )$$
where $\tilde{\omega} \in K^G$ is defined by $\tilde{\omega}(g)= \omega(gH)$. We give $\textrm{Sub}_G \otimes K$ the Borel structure induced by this embedding. Observe that $G$ acts on $\textrm{Sub}_G\otimes K$ by $g(H,\omega) = (gHg^{-1}, g\omega)$
where $g\omega: G/gHg^{-1} \to K$ is defined by $g\omega(f gHg^{-1}) := \omega(fg H)$. 

Let $M(\textrm{Sub}_G)$ be the space of all Borel probability measures on $\textrm{Sub}_G$ and let $M_{inv}(\textrm{Sub}_G)$ be the set of all $\eta \in M(\textrm{Sub}_G)$ that are invariant under conjugation.  Let $\eta \in M_{inv}(\textrm{Sub}_G)$ and $\kappa$ be a Borel probability measure on $K$. We define a probability measure $\eta \otimes \kappa$ on $\textrm{Sub}_G\otimes K$ by
$$d(\eta\otimes \kappa)(H,\omega) = d\kappa^{G/H}(\omega)d\eta(H)$$
where $\kappa^{G/H}$ is the product measure on $K^{G/H}$. This measure is invariant under the action $G \cc (\textrm{Sub}_G \otimes K, \eta \otimes \kappa)$ which is called the {\em non-free Bernoulli shift over $G$ with stabilizer distribution $\eta$ and base space $(K,\kappa)$}. 

\begin{example}
If $N \vartriangleleft G$ is a normal subgroup and $\eta=\delta_N$ is concentrated on $\{N\}$ then $G \cc (\textrm{Sub}_G \otimes K, \eta \otimes \kappa)$ is measurably conjugate to the action of $G$ on the product space $(K,\kappa)^{G/N}$.
\end{example}

\begin{defn}\label{defn:IRS-sofic}
We say that $\eta \in M_{inv}(\textrm{Sub}_G)$ is {\em sofic} if for every $\delta>0$, finite set $K \subset G$ and open neighborhood $\Omega \subset M(\textrm{Sub}_G)$ of $\eta$ there exists a map $\sigma:G \to [d]$ (for some integer $d>0$) such that
\begin{enumerate}
\item for any $g,h \in K$,
$$d^{-1}  |\{ q\in \Delta^0_{d}:~ \sigma(g)\sigma(h)\cdot q = \sigma(gh)\cdot q\}| \ge 1-\delta,$$
\item if $u_{d}$ is the uniform probability measure on $\Delta^0_{d}$ and $\textrm{Stab}_\sigma:\Delta^0_{d} \to 2^G$ is the map $\textrm{Stab}_\sigma(q) := \{g\in G:~\sigma(g)\cdot q=q\}$ then $(\textrm{Stab}_\sigma)_*u_{d} \in \Omega$.
\end{enumerate}
\end{defn}

\begin{exercise}
If $N \vartriangleleft G$ is a normal subgroup and $\eta=\delta_N\in M_{inv}(\textrm{Sub}_G)$ is concentrated on $\{N\}$ then $\eta$ is sofic if and only if $G/N$ is a sofic group.
\end{exercise}

The main result of this section is:
\begin{thm}\label{thm:nonfree}
Let $G$ be a countable group, $\eta$ a sofic conjugation-invariant Borel probability measure on $\textrm{Sub}_G$ and $(K,\kappa), (L,\lambda)$ two probability spaces. If $G \cc (\textrm{Sub}_G \otimes K, \eta \otimes \kappa)$ is measurably conjugate to $G \cc (\textrm{Sub}_G \otimes L, \eta \otimes \lambda)$ relative to the common factor $G \cc (\textrm{Sub}_G, \eta)$ then $H(K,\kappa)=H(L,\lambda)$.
\end{thm}

\begin{remark}
The hypotheses above mean that there are conull $G$-equivariant Borel subsets $X \subset \textrm{Sub}_G \otimes K, Y \subset \textrm{Sub}_G \otimes L$ and a $G$-equivariant measure-space isomorphism $\phi:X \to Y$ such that if $\pi_X:X \to \textrm{Sub}_G, \pi_Y:Y \to \textrm{Sub}_G$ denote the projection maps then $\pi_Y \phi = \pi_X$ almost everywhere.

\end{remark}

\begin{remark}
In several recent talks, Benjy Weiss has proven the following converse: if $G$ is a countable group, $\eta$ an ergodic non-atomic conjugation-invariant Borel probability measure on $\textrm{Sub}_G$ and $(K,\kappa), (L,\lambda)$ two probability spaces with $H(K,\kappa)=H(L,\lambda)$ then $G \cc (\textrm{Sub}_G \otimes K, \eta \otimes \kappa)$ is measurably conjugate to $G \cc (\textrm{Sub}_G \otimes L, \eta \otimes \lambda)$ (relative to the common factor $G \cc (\textrm{Sub}_G,\eta)$). The proof uses ideas similar to \cite{Bo12}. 
\end{remark}

To prove Theorem \ref{thm:nonfree} we will transfer the problem to a problem about principal groupoids defined next.
\begin{defn}
Let $G$ be a countable group and $G \cc (X,\mu)$ a measure-preserving action on a standard probability space. The {\em principal groupoid} $(\sH,\nu)$ for this action is defined by:
\begin{itemize}
\item $\sH = \{ (x,y) \in X \times X:~ \exists g \in G ~(gx=y)\}$
\item $\sH^0 = \{(x,x):~x\in X\} \subset \sH$;
\item $\nu$ is the pushforward of $\mu$ under the map $x \mapsto (x,x)$.
\item $\so(x,y)=(y,y), \ra(x,y)=(x,x)$, $(x,y)(y,z)=(x,z)$ and $(x,y)^{-1}=(y,x)$.
\end{itemize}
\end{defn}

Now let $(Z,\zeta)$ be a non-atomic standard probability space, $(\sH,\nu)$ be the principal groupoid for the action $G \cc (\textrm{Sub}_G\otimes Z, \eta \otimes \zeta)$, $(\sG,\mu)$ be the principal groupoid  for the action $G \cc (\textrm{Sub}_G\otimes (Z \times K), \eta \otimes (\zeta \times \kappa))$, $\pi_Z: Z \times K \to Z$ be the projection map and $\pi:\sG \to \sH$ be the map $\pi( (H_1, \omega_1), (H_2,\omega_2)) := ((H_1, \pi_Z\omega_1), (H_2,\pi_Z\omega_2))$. This map is class-bijective almost everywhere and measure-preserving. So its entropy is well-defined by Theorem \ref{thm:conull}.

\begin{thm}\label{thm:helper}
If $\pi:(\sG,\mu) \to (\sH,\nu)$ is as above then $h_{\P,\mu}(\pi) = H(K,\kappa)$ for any sofic approximation $\P$ to $(\sH,\nu)$.
\end{thm}

\begin{proof}
This follows from Theorem \ref{thm:bernoulli} because $\pi:(\sG,\mu) \to (\sH,\nu)$ is isomorphic to the Bernoulli shift over $(\sH,\nu)$ with base space $(K,\kappa)$.
\end{proof}

%?? Is $(\sH,\nu)$ above sofic (assuming $G$ is sofic)? 
In order to use the result above, we need to know that $(\sH,\nu)$ is sofic:
\begin{prop}\label{prop:is-sofic}
If $\eta$ is a sofic conjugation-invariant Borel probability measure on $\textrm{Sub}_G$ then $(\sH,\nu)$, as defined above, is sofic.
\end{prop}

%In order to prove this proposition, we need a few results, some of which may be of independent 
We will derive this proposition as a consequence of a more general result (Lemma \ref{lem:is-sofic}). First we need a definition.

\begin{defn}\label{defn:sofic-mod-stab}
Let $G$ be a countable group, $G \cc (X,\mu)$ be a probability-measure-preserving action, $\cP$ be a finite Borel partition of $X$, $K \subset G$ be a finite set,  $M((2^G\times \cP)^K)$ denote the space of Borel probability measures on $(2^G\times \cP)^K$, $\textrm{Stab}(x)=\{g \in G:~gx=x\}$ (for $x\in X$), $\cP(x)$ be the element of $\cP$ containing $x$ and $\Psi=\Psi(K,\cP):X \to (2^G \times \cP)^K$ be the map
$$\Psi(x)(k)  := (\textrm{Stab}(kx), \cP(kx)).$$
Note $\Psi_*\mu \in M((2^G\times \cP)^K)$.

We say that $G \cc (X,\mu)$ is {\em sofic with stabilizers} if for every
\begin{itemize}
\item finite $K \subset G$;
\item finite Borel partition $\cP$ of $X$;
\item open neighborhood $\Omega$ of $\Psi_*\mu$ in $M((2^G\times \cP)^K)$;
\item $\epsilon>0$;
\end{itemize}
there exist $\sigma:G \to [d]$ and $\phi: \Delta^0_d \to \cP$  (for some integer $d>0$) such that
\begin{itemize}
\item $ | \sigma(g)  \sigma(h) \vartriangle \sigma(g h) |_d < \epsilon\quad \forall g,h \in K$;
\item if $\textrm{Stab}_\sigma: \Delta^0_d \to 2^G$ is the map $\textrm{Stab}_\sigma(q) = \{g \in G:~\sigma(g)\cdot q=q\}$ and $\Phi: \Delta^0_d \to (2^G \times \cP)^K$ is the map $\Phi(q)(k) = (\textrm{Stab}_\sigma(\sigma(k)\cdot q), \phi(\sigma(k)\cdot q))$ then $\Phi_*u_d \in \Omega$ where $u_d$ is the uniform probability measure on $\Delta^0_d$.
\end{itemize}
\end{defn}

\begin{lem}\label{lem:Bernoulli-sofic}
If $\eta \in M_{inv}(\textrm{Sub}_G)$ is sofic and $(L,\lambda)$ is any nontrivial standard probability space then $G \cc (\textrm{Sub}_G \otimes L, \eta \otimes \lambda)$ is sofic with stabilizers.
\end{lem}
\begin{proof}
The proof is similiar to the proof of Theorem \ref{thm:bernoulli}, so we only explain the general idea. Let $K, F \subset G$ be finite, $\tau:F \to L$, $\cQ$ be a finite Borel partition of $L$ and $\cP$ be the finite partition of $\textrm{Sub}_G \otimes L$ defined by $\cP(H_1,\omega_1) = \cP(H_2,\omega_2)$ if $H_1 \cap F = H_2 \cap F$ and $\cQ(\omega_1(fH_1)) = \cQ(\omega_2(fH_2))$ for all $f \in F$. Also, let $\Omega$ an open neighborhood of  $\Psi_*(\eta \otimes \lambda)$ in $M((2^G\times \cP)^K)$ and $\epsilon>0$. Using arguments similar to the proof of Theorem \ref{thm:bernoulli} it can be shown that there exists a finite set $K' \subset G$, $\delta>0$ and an integer $D$ such that if $\sigma:G \to [d]$ with $d>D$ satisfies the conditions of Definition \ref{defn:IRS-sofic}, $\psi:\Delta^0_d \to L$ is chosen at random with law $\lambda^{\Delta^0_d}$ and $\phi:\Delta^0_d \to \cP$ is defined by $\phi(q)=P$ where $P\in \cP$ is the set of all $(H,\omega) \in \textrm{Sub}_G \otimes L$ satisfying $H \cap F = \{f \in F:~\sigma(f)\cdot q = q\}$, $\cQ(\omega(fH)) = \cQ(\psi(\sigma(f)\cdot q)) ~\forall f \in F$ then with positive probability $(\sigma,\phi)$ satisfies the conditions of Definition \ref{defn:sofic-mod-stab}. Moreover, partitions of the form above are dense in the Borel sigma-algebra of $\textrm{Sub}_G \otimes L$ in the sense that for any Borel $A \subset \textrm{Sub}_G \otimes L$, $\epsilon'>0$, there exists a partition $\cP$ of the form above and $A' \in \Sigma(\cP)$ such that $\eta \otimes \lambda(A \vartriangle A')<\epsilon'$. Using this it can be shown that the conditions of Definition \ref{defn:sofic-mod-stab} can be satisfied for {\em any} finite Borel partition $\cP$ (and any $K,\Omega,\epsilon$).
\end{proof}

\begin{lem}\label{lem:is-sofic}
If $G \cc (X,\mu)$ is sofic with stabilizers then the principal groupoid $(\sH,\nu)$ for the action $G \cc (X,\mu)$ is sofic.
\end{lem}

\begin{proof}
Let $F \subset_f \lb\sH\rb$ and $\delta>0$. For simplicity, we require $\sH^0 \in F$. It suffices to show the existence of a map $\tilde{\sigma}:F \to \lb d \rb$ such that 
\begin{itemize}
\item $| \tilde{\sigma}(f_1) \tilde{\sigma}(f_2) \vartriangle \tilde{\sigma}(f_1  f_2) |_d < \delta$
\item $| \tr_d(\tilde{\sigma}(f)) - \tr_\sH(f) | < \delta$
\end{itemize}
for every $f,f_1,f_2\in F$.

%Indeed suppose that for each 

To simplify notation, we identify $X$ with $\sH^0$ and $\mu$ with $\nu$ in the obvious way. The first step of the proof is to choose $\epsilon, K, \cP, \Omega$ and then apply Definition \ref{defn:sofic-mod-stab} to obtain $\sigma$ and $\phi$ out of which we will construct $\tilde{\sigma}$. 

Let $\epsilon>0$ and $F^2 = \{f_1  f_2:~ f_1,f_2 \in F\}$. For $f \in F^2, g\in G$ let
$$A'(f,g)=\{x \in \so(f):~f\cdot x=gx\}.$$
Note that $\cup_{g\in G} A'(f,g) = \so(f)$. So there exist a finite set $K_f \subset G$ and a collection $\{ A(f,k):~k \in K_f\}$ of pairwise disjoint Borel sets such that  
\begin{itemize}
\item $A(f,k) \subset A'(f,k)$ for every $k\in K_f$,
\item $e\in K_f$,
\item $A(f,e) = A'(f,e)=\{x \in \so(f):~ f\cdot x=x\}$.
\item $\nu( \cup\{ A(f,k):~ k\in K_f\}) \ge \nu(\so(f)) -\epsilon$.
\end{itemize}
To simplify notation, we set $A(f,k) = \emptyset$ if $k \notin K_f$.

Let $K = (\bigcup_{f \in F^2} K_f)^{-1}(\bigcup_{f \in F^2} K_f)$. Because $e \in K_f$ (for every $f$), $K \supset \bigcup_{f \in F^2} K_f$. %For $g,h,k \in K$ let 
%$$X(g,h,k) = \{x \in X:~ g^{-1}kh \in \textrm{Stab}(x)\}.$$
Let $\cP$ be a finite Borel partition of $X$ such that $A(f,g) \in \Sigma(\cP)$ for every $f\in F^2, g \in K$.
%\begin{itemize}
%\item $X(g,h,k) \in \Sigma(\cP)$ for every $g,h,k \in K$;
%\item $A(f,k) \in \Sigma(\cP)$ for every $f \in F^2, k \in K_f$;
%\item $A(f_1 f_2, g) \cap A(f_2, h) \cap k^{-1} A(f_1,k) \in \Sigma(\cP)$ for every $f_1,f_2 \in F$ and $g,h,k \in K$.
%\end{itemize}
Abusing notation, we will occasionally find it convenient to identify $A(f,g)$ with the set of all $P \in \cP$ such that $P \subset A(f,g)$. This should cause no confusion.

In order to define $\Omega$, for $g_1,g_2,g_3 \in K$ and every $f_1,f_2 \in F$ let $Z(f_1,f_2,g_1,g_2,g_3)$ be the set of all $\Upsilon \in (2^G \times \cP)^K$ such that
\begin{itemize}
\item  $\Upsilon(e) =(H,P)\textrm{ for some } H \in 2^G \textrm{ with } g^{-1}_3g_2g_1 \in H$ and $P \subset  A(f_2f_1,g_3) \cap A(f_1,g_1)$,
\item $\Upsilon(g_1) \in 2^G \times A(f_2,g_2)$.
\end{itemize}
Also let 
$$Y(f_1,f_2,g_1,g_2)=\{ \Upsilon \in (2^G \times \cP)^K:~\Upsilon(e) \in 2^G \times A(f_1,g_1), \Upsilon(g_1) \in 2^G \times A(f_2,g_2)\}.$$ 
Let $\Psi:X \to (2^G \times \cP)^K$ be as in Definition \ref{defn:sofic-mod-stab}. Let $\Omega$ be the set of all $\omega \in M((2^G \times \cP)^K)$ such that every $f,f_1,f_2 \in F$, $g,g_1,g_2,g_3,h,k \in K$ with $h\ne g$, $k\ne e$ and $g^{-1}h \in K$,
\begin{eqnarray}
\label{eqn:Z} \frac{\epsilon}{|K|^{3}} &>& |\omega(Z(f_1,f_2,g_1,g_2,g_3)) - \Psi_*\mu(Z(f_1,f_2,g_1,g_2,g_3))|\\
\label{eqn:Y} \frac{\epsilon}{|K|^3} &>& |\omega(Y(f_1,f_2,g_1,g_2)) - \Psi_*\mu(Y(f_1,f_2,g_1,g_2))|\\ 
\label{eqn:3} \frac{\epsilon}{|K|^3} &>& | \omega(\{\Upsilon \in(2^G \times \cP)^K:\Upsilon(e)\in 2^G \times A(f,g)\}) - \mu(A(f,g)) | \\ 
\label{eqn:weird} \frac{\epsilon}{|K|^3}&>& \omega(\{ \Upsilon \in (2^G \times \cP)^K: \Upsilon(g^{-1}h) \in 2^G \times A(f,g), ~\Upsilon(e) \in 2^G \times A(f,h)\})\\
\label{eqn:weird2} \frac{\epsilon}{|K|^3}&>& \omega(\{ \Upsilon \in (2^G \times \cP)^K:\Upsilon(e) =(H,P), k\in H, P \subset A(f,k) \}).
\end{eqnarray}
\noindent {\bf Claim 1}. $\Omega$ is an open neighborhood of $\Psi_*\mu$ in $M((2^G \times \cP)^K)$.

\begin{proof}[Proof of Claim 1]
Observe that $\{H \in 2^G:~g_3^{-1}g_2g_1\in H\}$ is a clopen subset of $2^G$. It follows that $Z(f_1,f_2,g_1,g_2,g_3)$ and $Y(f_1,f_2,g_1,g_2)$ are clopen subsets of $(2^G \times \cP)^K$. From this and similar considerations, it is easy to see that $\Omega$ is open.

We need to check that $\Psi_*\mu \in \Omega$. It is immediate that $\omega=\Psi_*\mu$ satisfies (\ref{eqn:Z}, \ref{eqn:Y}). Inequality (\ref{eqn:3}) holds by definition of $\Psi$. To check (\ref{eqn:weird2}), suppose $x\in X$ and $\Psi(x)(e)=(H,P), k\in H, P \subset A(f,k)$. Then $k \in \textrm{Stab}(x)$ and $x\in A(f,k) \Rightarrow fx=kx=x$. But this implies $x \in A(f,e)$. Since $k\ne e$, $A(f,e)$ and $A(f,k)$ are disjoint. This contradiction implies the image of $\Psi$ is disjoint from $\{ \Upsilon \in (2^G \times \cP)^K:~ \Upsilon(e) =(H,P), k\in H, P \subset A(f,k) \}$ which implies (\ref{eqn:weird2}) with $\omega=\Psi_*\mu$. 

To check (\ref{eqn:weird}), let $x \in X$ and, to obtain a contradiction, suppose that $\Psi(x)(g^{-1}h)  \in 2^G \times A(f,g)$ and $\Psi(x)(e) \in 2^G \times A(f,h)$. The first condition implies $\cP(g^{-1}hx) \subset A(f,g)$ and the second $\cP(x) \subset A(f,h)$. Therefore, $f g^{-1}hx =gg^{-1}hx=hx$ and $fx=hx$. Therefore $g^{-1}hx=x$. So we have $\cP(x) \subset A(f,g) \cap A(f,h)$ which contradicts the hypothesis that $g \ne h$ (since $A(f,g), A(f,h)$ are disjoint). This contradiction shows the image of $\Psi$ is disjoint from $\{ \Upsilon \in (2^G \times \cP)^K:~ \Upsilon(g^{-1}h) \in 2^G \times A(f,g), ~\Upsilon(e) \in 2^G \times A(f,h)\}$ which implies (\ref{eqn:weird}) with $\omega=\Psi_*\mu$. 
\end{proof}

By Definition \ref{defn:sofic-mod-stab}, there exist $\sigma:G \to [d]$ and $\phi: \Delta^0_d \to \cP$ (for some $d>0$)  such that
\begin{itemize}
\item $ | \sigma(g)  \sigma(h) \vartriangle \sigma(g h) |_d < \epsilon |K|^{-3}\quad \forall g,h \in K$;
\item if $\textrm{Stab}_\sigma: \Delta^0_d \to 2^G$ is the map $\textrm{Stab}_\sigma(q) = \{g \in G:~\sigma(g)\cdot q=q\}$ and $\Phi: \Delta^0_d \to (2^G \times \cP)^K$ is the map $\Phi(q)(k) = (\textrm{Stab}_\sigma(\sigma(k)\cdot q), \phi(\sigma(k)\cdot q))$ then $\Phi_*u \in \Omega$ where $u$ is the uniform probability measure on $\Delta^0_d$.
\end{itemize}
Without loss of generality, we may assume $\sigma(e)=\Delta^0_d$. For $f \in F^2$ let 
\begin{eqnarray*}
\textrm{BAD}(f) &=& \left\{q \in \Delta^0_d:~ \exists g \ne h \in K_f \textrm{ such that } q \in \sigma(g)\cdot \phi^{-1}(A(f,g)) \cap \sigma(h)\cdot \phi^{-1}(A(f,h)) \right\},\\
\textrm{GOOD}(f) &=& \{q \in \Delta^0_d:~ \exists g \in K_f \textrm{ such that } q \in \phi^{-1}(A(f,g)) \textrm{ and } \sigma(g)\cdot q \notin \textrm{BAD}(f)\}.
\end{eqnarray*}
For $f \in F^2$, we let 
$$\tilde{\sigma}(f)  = \left\{ (\sigma(g)\cdot q,q):q \in \textrm{GOOD}(f) \cap \phi^{-1}A(f,g), g \in K_f\right\} \subset \Delta^0_d.$$
%$$\tilde{\sigma}(f)  = \bigcup_{g\in K_f} \bigcup_{q \in \textrm{GOOD}(f) \cap \phi^{-1}(f,g)} (q, \sigma(g)\cdot q) \subset \Delta_d.$$
% be the set of all $\tau \in \Delta_d$ with $\so(\tau) \in \textrm{GOOD}(f)$ and $\ra(\tau) = \sigma(g)\cdot \so(\tau)$ where $g \in K_f$ is such that $\so(\tau) \in \phi^{-1}(A(f,g))$. 
 We claim that $\tilde{\sigma}(f) \in \lb d \rb$. It suffices to check that the range map restricted to $\tilde{\sigma}(f)$ is injective. Suppose that $q_1,q_2 \in \so(\tilde{\sigma}(f))$ and $\tilde{\sigma}(f)\cdot q_1 = \tilde{\sigma}(f)\cdot q_2=z$ for some $z \in \Delta^0_d$. Then there exist $g_1,g_2 \in K_f$ such that $\phi(q_i) \in A(f,g_i)$,  and $\tilde{\sigma}(f)\cdot q_i = \sigma(g_i)\cdot q_i =z \notin \textrm{BAD}(f)$ for $i=1,2$. Because $z \notin \textrm{BAD}(f)$ and $z \in \sigma(g_1)\cdot \phi^{-1}(A(f,g_1)) \cap \sigma(g_2) \cdot \phi^{-1}(A(f,g_2))$ it must be that $g_1=g_2$. Since $\sigma(g_i)\cdot q_i = z$ this implies $q_1=q_2 = \sigma(g_i)^{-1} \cdot z$. Because $q_1,q_2$ are arbitrary, the range map restricted to $\tilde{\sigma}(f)$ is injective as required.

%The definition of $\textrm{GOOD}(f)$ implies the range map restricted to $\tilde{\sigma}(f)$ is injective, so $\tilde{\sigma}(f)\in \lb d \rb$ as required.

To manage error terms we will use big $O$ notation. The implied constant is allowed to depend on $F$ but not on $\epsilon, \delta, K, \sigma$.

\noindent {\bf Claim 2}. For any $f \in F^2$, $|\textrm{BAD}(f)|_d =O(\epsilon |K|^{-1})$.

\begin{proof}[Proof of Claim 2]
For any $f \in F^2$, and any $g \ne h \in K_f$, 
\begin{eqnarray*}
&&\Phi\Big(  \sigma(g^{-1}h)^{-1} \cdot \phi^{-1}\big(A(f,g)\big) \cap \phi^{-1}\big(A(f,h)\big) \Big)\\
&\subset& \Big\{ \Upsilon \in (2^G \times \cP)^K:~ \Upsilon(g^{-1}h) \in 2^G \times A(f,g), ~\Upsilon(e) \in 2^G \times A(f,h)\Big\}.
\end{eqnarray*}
Because $\Phi_*u_d \in \Omega$, (\ref{eqn:weird}) implies
$$\Phi_*u_d\Big(\big\{ \Upsilon \in (2^G \times \cP)^K:~ \Upsilon(g^{-1}h) \in 2^G \times A(f,g), ~\Upsilon(e) \in 2^G \times A(f,h)\big\}\Big) < \epsilon |K|^{-3}.$$
Therefore,
$$ \big| \sigma(g^{-1}h)^{-1}\cdot \phi^{-1}(A(f,g)) \cap \phi^{-1}(A(f,h)) \big|_d < \epsilon |K|^{-3}.$$
By Lemma \ref{lem:basic-formulas2}, $|\sigma(g^{-1}h)^{-1} \vartriangle \sigma(h)^{-1}\sigma(g)|_d=O(\epsilon |K|^{-3})$. So
$$ \big| \sigma(h)^{-1}\sigma(g)\cdot \phi^{-1}(A(f,g)) \cap \phi^{-1}(A(f,h)) \big|_d =O(\epsilon |K|^{-3}).$$
Recall that $\sigma$ maps $G$ into $[d]$ (not just $\lb d\rb$). So we can multiply by $\sigma(h)$ to obtain
$$\big|\sigma(g)\cdot \phi^{-1}(A(f,g)) \cap \sigma(h)\cdot \phi^{-1}(A(f,h))\big|_d =O(\epsilon |K|^{-3}).$$
So
\begin{eqnarray*}
|\textrm{BAD}(f)|_d &=& \left| \bigcup_{g\ne h \in K_f} \sigma(g) \cdot \phi^{-1}(A(f,g)) \cap \sigma(h)\cdot \phi^{-1}(A(f,h))\right|_d =O(\epsilon |K|^{-1}).
\end{eqnarray*}

 \end{proof}

\noindent {\bf Claim 3}. For any $f \in F$, 
$$| \tr_d(\tilde{\sigma}(f)) - \tr_\sH(f) | =O(\epsilon).$$

\begin{proof}[Proof of Claim 3]
By definition of $A(f,e)$, $\tr_\sH(f) = \nu(A(f,e))$. Now $\tr_d(\tilde{\sigma}(f)) = |\tilde{\sigma}(f) \cap \Delta^0_d|_d$. Note that if $q \in \tilde{\sigma}(f) \cap \Delta^0_d$ then either $q \in \phi^{-1}(A(f,e))$ or $q \in \phi^{-1}(A(f,g))$ for some $g \ne e$ in which case $\Phi(q)(e)=(H,P)$, for some $H \in 2^G, P \in \cP$ with $g\in H, P \subset A(f,g)$. By (\ref{eqn:weird2}) we now have
$$\left| (\tilde{\sigma}(f) \cap \Delta^0_d) \setminus \phi^{-1}(A(f,e)) \right|_d = O(\epsilon).$$
The definition of $\tilde{\sigma}$ implies
$$\phi^{-1}(A(f,e))  \setminus \textrm{BAD}(f) \subset \tilde{\sigma}(f) \cap \Delta^0_d.$$
So Claim 2 implies
$$\left| (\tilde{\sigma}(f) \cap \Delta^0_d )\vartriangle \phi^{-1}(A(f,e)) \right|_d = O(\epsilon).$$
Because $\Phi_*u \in \Omega$, (\ref{eqn:3}) implies
$$| | \phi^{-1}(A(f,e)) |_d - \nu(A(f,e)) | < \epsilon.$$
This implies Claim 3.
\end{proof}

For the rest of the proof, we fix $f_1,f_2 \in F$. Because of Claim 3, it suffices to show that $| \tilde{\sigma}(f_1) \tilde{\sigma}(f_2) \vartriangle \tilde{\sigma}(f_1  f_2) | =O(\epsilon)$. For any $g_1,g_2,g_3 \in K$, let 
\begin{eqnarray*}
P(g_1,g_2,g_3) &:=& A(f_2f_1,g_3) \cap A(f_1,g_1) \cap g_1^{-1} A(f_2,g_2)  = \Psi^{-1}(Z(f_1,f_2,g_1,g_2,g_3))\\
Q(g_1,g_2,g_3) &:=& \phi^{-1}(A(f_2f_1,g_3)) \cap \phi^{-1}(A(f_1,g_1)) \cap \sigma(g_1)^{-1}\cdot\phi^{-1}(A(f_2,g_2))\\
&& \cap \{q \in \Delta^0_d:~ \textrm{Stab}_\sigma(q) \ni g^{-1}_3g_2g_1\}\\
&=& \Phi^{-1}(Z(f_1,f_2,g_1,g_2,g_3)).
\end{eqnarray*}
By (\ref{eqn:Z}) we have
\begin{eqnarray}\label{eqn:PQ}
| \nu(P(g_1,g_2,g_3)) - |Q(g_1,g_2,g_3)|_d| < \epsilon |K|^{-3}.
\end{eqnarray}
Note that $P(g_1,g_2,g_3) = A(f_2f_1,g_3) \cap A(f_1,g_1) \cap f_1^{-1} A(f_2,g_2)$. So  
$$\bigcup_{g_1,g_2,g_3 \in K} P(g_1,g_2,g_3) = \left[ \bigcup_{g_3\in K} A(f_2f_1,g_3)\right] \cap \left[\bigcup_{g_1 \in K} A(f_1,g_1)\right] \cap f_1^{-1} \left[ \bigcup_{g_2 \in K} A(f_2,g_2)\right].$$
By choice of $A(\cdot,\cdot)$, this implies
$$\nu\left(\so(f_2f_1) \vartriangle \bigcup_{g_1,g_2,g_3 \in K} P(g_1,g_2,g_3) \right) = O(\epsilon).$$
Since the families $\{P(g_1,g_2,g_3)\}_{g_1,g_2,g_3}$ and $\{Q(g_1,g_2,g_3)\}_{g_1,g_2,g_3}$ are each pairwise disjoint, (\ref{eqn:PQ}) now implies
\begin{eqnarray}\label{eqn:Q}
\sum_{g_1,g_2,g_3} |Q(g_1,g_2,g_3)|_d = O(\epsilon) + \nu(\so(f_2f_1)).
\end{eqnarray}

%---here----

\noindent {\bf Claim 4}.
$$\left| \so(\tilde{\sigma}(f_2f_1)) \vartriangle \bigcup_{g_1,g_2,g_3 \in K} Q(g_1,g_2,g_3)\right|_d =O(\epsilon).$$

\begin{proof}[Proof of Claim 4]
By definition, 
$$ \so(\tilde{\sigma}(f_2f_1)) = \bigcup_{g_3 \in K_f} \phi^{-1}(A(f_2f_1,g_3))\setminus \sigma(g_3)^{-1}\cdot\textrm{BAD}({f_2f_1}).$$
So Claim 2 implies
\begin{eqnarray*}
\left| \so(\tilde{\sigma}(f_2f_1)) \vartriangle \bigcup_{g_1,g_2,g_3\in K} Q(g_1,g_2,g_3) \right|_d &\le& O(\epsilon) + \sum_{g_3\in K}  \left|\phi^{-1}(A(f_2f_1,g_3)) \vartriangle \bigcup_{g_1,g_2 \in K} Q(g_1,g_2,g_3)\right|_d.
\end{eqnarray*}
By definition, $Q(g_1,g_2,g_3) \subset \phi^{-1}(A(f_2f_1,g_3))$ for every $g_1,g_2$. Also the sets $Q(g_1,g_2,g_3)$ are pairwise disjoint. So
\begin{eqnarray*}
\left| \so(\tilde{\sigma}(f_2f_1)) \vartriangle \bigcup_{g_1,g_2,g_3} Q(g_1,g_2,g_3)\right|_d &\le& O(\epsilon) + \sum_{g_3 \in K}  |\phi^{-1}(A(f_2f_1,g_3))|_d - \sum_{g_1,g_2,g_3} |Q(g_1,g_2,g_3)|_d\\
&\le&O(\epsilon) - \nu(\so(f_2f_1))+ \sum_{g_3 \in K}  |\phi^{-1}(A(f_2f_1,g_3))|_d\\
&\le&O(\epsilon) - \nu(\so(f_2f_1))+ \sum_{g_3 \in K}  \nu(A(f_2f_1,g_3)) \le O(\epsilon).
\end{eqnarray*}
The second inequality above comes from (\ref{eqn:Q}), the third follows from (\ref{eqn:3}) while the last is implied by the choice of $A(\cdot,\cdot)$. 
\end{proof}

\noindent {\bf Claim 5}.
$$\left| \so(\tilde{\sigma}(f_2)\tilde{\sigma}(f_1)) \vartriangle \bigcup_{g_1,g_2,g_3 \in K} Q(g_1,g_2,g_3)\right|_d =O(\epsilon).$$

\begin{proof}[Proof of Claim 5]
By definition,
$$\so(\tilde{\sigma}(f_2)\tilde{\sigma}(f_1))  = \so(\tilde{\sigma}(f_1)) \cap \tilde{\sigma}(f_1)^{-1} \cdot \so(\tilde{\sigma}(f_2)).$$
By definition of $\tilde{\sigma}$,
\begin{eqnarray*}
\so(\tilde{\sigma}(f_1)) &=& \bigcup_{g_1\in K_{f_1}} \phi^{-1}(A(f_1,g_1)) \setminus \sigma(g_1)^{-1}\cdot\textrm{BAD}({f_1})\\
\tilde{\sigma}(f_1)^{-1}\cdot \so(\tilde{\sigma}(f_2)) &=& \tilde{\sigma}(f_1)^{-1}\cdot \left(\bigcup_{g_2\in K_{f_2}} \phi^{-1}(A(f_2,g_2)) \setminus \sigma(g_2)^{-1}\cdot\textrm{BAD}({f_2})\right).
\end{eqnarray*}
Claim 2 and the disjointness properties of $A(\cdot,\cdot)$ imply
\begin{eqnarray*}
O(\epsilon)&=&\left|\so(\tilde{\sigma}(f_2)\tilde{\sigma}(f_1)) \vartriangle  \left(\bigcup_{g_1\in K} \phi^{-1}(A(f_1,g_1)) \cap \tilde{\sigma}(f_1)^{-1}\cdot\bigcup_{g_2\in K} \phi^{-1}(A(f_2,g_2)) \right)\right|_d\\
                 &=&\left|\so(\tilde{\sigma}(f_2)\tilde{\sigma}(f_1)) \vartriangle \left(\bigcup_{g_1,g_2\in K} \phi^{-1}(A(f_1,g_1)) \cap \tilde{\sigma}(f_1)^{-1}\cdot \phi^{-1}(A(f_2,g_2)) \right)\right|_d.
\end{eqnarray*}
We would like to replace the $\tilde{\sigma}(f_1)^{-1}$ above with $\sigma(g_1)^{-1}$. To see why this is possible, observe that if $q \in \phi^{-1}(A(f_1,g_1)) \setminus \sigma(g_1)^{-1}\cdot \textrm{BAD}(f_1)$ then $\tilde{\sigma}(f_1)\cdot q = \sigma(g_1)\cdot q$. Claim 2 now implies
$$O(\epsilon)=\left|\so(\tilde{\sigma}(f_2)\tilde{\sigma}(f_1)) \vartriangle \left(\bigcup_{g_1,g_2\in K} \phi^{-1}(A(f_1,g_1)) \cap \sigma(g_1)^{-1}\cdot \phi^{-1}(A(f_2,g_2)) \right)\right|_d.$$
%**the equation above needs an explanation**

Observe that $Q(g_1,g_2,g_3) \subset \phi^{-1}(A(f_1,g_1)) \cap \sigma(g_1)^{-1}\cdot\phi^{-1}(A(f_2,g_2))$. Moreover the sets $Q(g_1,g_2,g_3)$ are pairwise disjoint. Therefore,
\begin{eqnarray*}
&&\left| \so(\tilde{\sigma}(f_2)\tilde{\sigma}(f_1)) \vartriangle \bigcup_{g_1,g_2,g_3 \in K} Q(g_1,g_2,g_3)\right|_d\\
 &=& O(\epsilon) + \sum_{g_1,g_2 \in K} \left| \left( \phi^{-1}(A(f_1,g_1)) \cap \sigma(g_1)^{-1}\cdot \phi^{-1}(A(f_2,g_2)\right) \setminus \bigcup_{g_3 \in K} Q(g_1,g_2,g_3)\right|_d\\
&=& O(\epsilon) + \sum_{g_1,g_2 \in K}\left| \phi^{-1}(A(f_1,g_1)) \cap \sigma(g_1)^{-1}\cdot \phi^{-1}(A(f_2,g_2))\right| - \sum_{g_1,g_2,g_3 \in K}  |Q(g_1,g_2,g_3)|_d\\
&=& O(\epsilon) - \nu(\so(f_1f_2)) + \sum_{g_1,g_2,g_3\in K} \left| \phi^{-1}(A(f_1,g_1)) \cap \sigma(g_1)^{-1}\cdot \phi^{-1}(A(f_2,g_2))\right|_d.
\end{eqnarray*}
The last equality uses (\ref{eqn:Q}). Observe that 
$$\phi^{-1}(A(f_1,g_1)) \cap \sigma(g_1)^{-1}\cdot\phi^{-1}(A(f_2,g_2)) = \Phi^{-1}\big( Y(f_1,f_2,g_1,g_2)\big).$$
So (\ref{eqn:Y}) implies
$$| \phi^{-1}(A(f_1,g_1)) \cap \sigma(g_1)^{-1}\cdot\phi^{-1}(A(f_2,g_2))|_d = O(\epsilon |K|^{-3}) + \nu(A(f_1,g_1) \cap g_1^{-1} A(f_2,g_2)).$$
By choice of $A(\cdot, \cdot)$, this implies
$$\sum_{g_1,g_2 \in K}| \phi^{-1}(A(f_1,g_1)) \cap \sigma(g_1)^{-1}\cdot\phi^{-1}(A(f_2,g_2))|_d = O(\epsilon) + \nu(\so(f_2f_1))$$
which implies the claim.
\end{proof}

\noindent {\bf Claim 6}. 
 $$| \tilde{\sigma}(f_2) \tilde{\sigma}(f_1) \vartriangle \tilde{\sigma}(f_2 f_1) | =O(\epsilon).$$

\begin{proof}[Proof of Claim 6]
Let $Q'(g_1,g_2,g_3)$ be the set of all $q \in Q(g_1,g_2,g_3)$ such that $\sigma(g_3)\cdot q=\sigma(g_2)\sigma(g_1)\cdot q$. Since $q \in Q(g_1,g_2,g_3)$ implies $\sigma(g_3^{-1}g_2g_1)\cdot q=q$, it follows from Lemma \ref{lem:basic-formulas2} that 
$$|Q'(g_1,g_2,g_3)\vartriangle Q(g_1,g_2,g_3)|_d =O(\epsilon |K|^{-3}).$$
By Claims 1, 3 and 4, it suffices to show that if 
$$q \in \left(\bigcup_{g_1,g_2,g_3 \in K} Q'(g_1,g_2,g_3)\right) \cap \so(\tilde{\sigma}(f_2 f_1)) \cap \so( \tilde{\sigma}(f_2) \tilde{\sigma}(f_1) )$$
then
%$\tilde{\sigma}(f_2 f_1)\cdot q$ and $\tilde{\sigma}(f_2) \tilde{\sigma}(f_1)\cdot q $ are defined then 
$$\tilde{\sigma}(f_2) \tilde{\sigma}(f_1)\cdot q = \tilde{\sigma}(f_2f_1)\cdot q.$$
By definition, if $q \in Q'(g_1,g_2,g_3)\cap \so(\tilde{\sigma}(f_2f_1))$ then $\tilde{\sigma}(f_2f_1)\cdot q = \sigma(g_3)\cdot q$. If also $q \in \so( \tilde{\sigma}(f_2) \tilde{\sigma}(f_1) )$ then $\tilde{\sigma}(f_2) \tilde{\sigma}(f_1)\cdot q = \sigma(g_2)\sigma(g_1)\cdot q$. If $q \in Q'(g_1,g_2,g_3)$ then $\sigma(g_2)\sigma(g_1)\cdot q = \sigma(g_3)\cdot q$. So $\tilde{\sigma}(f_2) \tilde{\sigma}(f_1)\cdot q = \tilde{\sigma}(f_2  f_1)\cdot q$ as required.
\end{proof}
Because $\epsilon$ is arbitrary, Claims 3 and 6 imply the Lemma.

\end{proof}

\begin{proof}[Proof of Proposition \ref{prop:is-sofic}]
This is immediate from Lemmas \ref{lem:Bernoulli-sofic} and \ref{lem:is-sofic}.
\end{proof}

\begin{proof}[Proof of Theorem \ref{thm:nonfree}]
Suppose $G \cc (\textrm{Sub}_G \otimes K, \eta \otimes \kappa)$ is measurably conjugate to $G \cc (\textrm{Sub}_G \otimes L, \eta \otimes \lambda)$ relative to the common factor $G \cc (\textrm{Sub}_G, \eta)$. Let $(Z,\zeta)$ be a non-atomic standard probability space. Let $(\sH,\nu)$ be the principal groupoid for the action $G \cc (\textrm{Sub}_G\otimes Z, \eta \otimes \zeta)$, $(\sG_K,\mu_K)$ be the principal groupoid  for the action $G \cc (\textrm{Sub}_G\otimes (Z \times K), \eta \otimes (\zeta \times \kappa))$ and $(\sG_L,\mu_L)$ be the principal groupoid  for the action $G \cc (\textrm{Sub}_G\otimes (Z \times L), \eta \otimes (\zeta \times \lambda))$. Let $\pi^K_Z: Z \times K \to Z$ and $\pi^L_Z: Z \times L \to Z$ be the projection maps. Define $\pi_K:\sG_K \to \sH$ by $\pi_K( (H_1, \omega_1), (H_2,\omega_2)) = ((H_1, \pi^K_Z\omega_1), (H_2,\pi^K_Z\omega_2))$. This is an almost-everywhere-class-bijective measure-preserving extension. Define $\pi_L:\sG_L \to \sH$ similarly.

Because $G \cc (\textrm{Sub}_G \otimes K, \eta \otimes \kappa)$ is measurably conjugate to $G \cc (\textrm{Sub}_G \otimes L, \eta \otimes \lambda)$ relative to the common factor $G \cc (\textrm{Sub}_G, \eta)$, it follows that $G \cc  (\textrm{Sub}_G\otimes (Z \times K), \eta \otimes (\zeta \times \kappa))$ is measurably conjugate to $G \cc (\textrm{Sub}_G\otimes (Z \times L), \eta \otimes (\zeta \times \lambda))$ relative to the common factor $G \cc (\textrm{Sub}_G \otimes Z, \eta \otimes \zeta)$ from which it follows that $\pi_K:(\sG_K,\mu_K) \to (\sH,\nu)$ and  $\pi_L:(\sG_L,\mu_L) \to (\sH,\nu)$ are isomorphic. 

%Let $\Phi:\textrm{Sub}_G \otimes K \to \textrm{Sub}_G \otimes L$ be a measure conjugacy between the respective $G$-actions. There is an obvious measure conjugacy $\Phi': \textrm{Sub}_G \otimes (Z\times K) \to \textrm{Sub}_G \otimes (Z\times L)$ obtained from $\Phi$. Define $\Phi'': \sG_K \to \sG_L$ by $\Phi''( (H_1,\omega_2), (H_2,\omega_2)) =  ( \Phi'(H_1,\omega_1), \Phi'(H_2,\omega_2))$. This 

By Proposition \ref{prop:is-sofic}, there exists a sofic approximation $\P$ to $(\sH,\nu)$. Theorem \ref{thm:helper} implies that 
$$h_{\P,\mu_K}(\pi_K) = H(K,\kappa),\quad h_{\P,\mu_L}(\pi_L) = H(L,\lambda).$$
Because $\pi_K$ and $\pi_L$ are isomorphic, $h_{\P,\mu_K}(\pi_K) = h_{\P,\mu_L}(\pi_L)$ which implies the Theorem.

%$H(K,\kappa)=H(L,

\end{proof}

%\appendix{Junk?}

\linespread{1.4}

\end{document}